\documentclass[reqno,a4paper]{siamart171218}
\usepackage{amsmath,amssymb,amsfonts,mathtools,comment}
\usepackage{tikz}
\usepackage{algorithm}
\usepackage{algpseudocode}
\usepackage{graphicx,epsfig,color,scalerel,subfigure,booktabs}
\usepackage[super]{nth}
\usepackage{xparse}
\usepackage{siunitx}
\usepackage{afterpage}

\tikzset{shift entire picture/.style n args={2}{execute at end picture={
\pgfmathtruncatemacro{\tmpx}{sign(#1)}
\pgfmathtruncatemacro{\tmpy}{sign(#2)}
\ifnum\tmpx=1
  \ifnum\tmpy=1
   \path[use as bounding box] ([xshift=-#1,yshift=-#2]current bounding box.south west) rectangle 
(current bounding box.north east);
  \else
   \path[use as bounding box] ([xshift=-#1]current bounding box.south west) rectangle 
([yshift=-#2]current bounding box.north east);
  \fi
\else  
  \ifnum\tmpy=1
   \path[use as bounding box] ([yshift=-#2]current bounding box.south west) rectangle 
([xshift=-#1]current bounding box.north east);
  \else
   \path[use as bounding box] (current bounding box.south west) rectangle 
([xshift=-#1,yshift=-#2]current bounding box.north east); 
  \fi
\fi}}}

\definecolor{lightgreen}{rgb}{0.2,0.6,0.2}
\definecolor{lightblue}{rgb}{0.15,0.15,0.85}
\definecolor{darkred}{rgb}{0.85,0.15,0.15}

\newtheorem{remark}{Remark}

\newcommand{\bn}{\boldsymbol{n}}

\newcommand{\bu}{\boldsymbol{u}}
\newcommand{\bv}{\boldsymbol{v}}
\newcommand{\bw}{\boldsymbol{w}}

\newcommand{\fb}{\boldsymbol{f}}

\newcommand{\bt}{\boldsymbol{t}}
\newcommand{\bx}{\boldsymbol{x}}
\newcommand{\by}{\boldsymbol{y}}

\newcommand\beps{\boldsymbol{\varepsilon}}
\newcommand\bsigma{\boldsymbol{\sigma}}
\newcommand\bzeta{\boldsymbol{\zeta}}
\newcommand\bzero{\boldsymbol{0}}
\newcommand\bxi{\boldsymbol{\xi}}
\newcommand\btau{\boldsymbol{\tau}}
\newcommand\bnabla{\boldsymbol{\nabla}}
\newcommand\bPi{\boldsymbol{\Pi}}

\newcommand{\bH}{\mathbf{H}}

\newcommand{\bL}{\mathbf{L}}

\newcommand{\bV}{\mathbf{V}}
\newcommand{\bQ}{\mathbf{Q}}

\newcommand{\bbC}{\mathbb{C}}
\newcommand{\bbL}{\mathbb{L}}
\newcommand{\bbH}{\mathbb{H}}
\newcommand{\bbM}{\mathbb{M}}
\newcommand{\bbI}{\mathbb{I}}
\newcommand{\bbR}{\mathbb{R}}

\newcommand{\cC}{\mathcal{C}}
\newcommand{\cS}{\mathcal{S}}

\newcommand{\cA}{\mathcal{A}}

\newcommand\bdiv{\mathop{\mathbf{div}}\nolimits}
\newcommand\vdiv{\mathop{\mathrm{div}}\nolimits}
\newcommand\vrot{\mathop{\mathrm{rot}}\nolimits}
\newcommand\tr{\mathop{\mathrm{tr}}\nolimits}

\DeclarePairedDelimiter\norm{\lVert}{\rVert}

\numberwithin{equation}{section}
\numberwithin{remark}{section}
\numberwithin{assumption}{section}
\numberwithin{figure}{section}
\numberwithin{table}{section}

\newcommand{\cred}{}
\newcommand{\cblue}{}
\newcommand{\cgreen}{}

\allowdisplaybreaks

\headers{Robust VEM for stress-assisted diffusion}{Khot, Rubiano \& Ruiz-Baier}

\title{Robust virtual element methods for coupled stress-assisted diffusion problems\thanks{\textbf{Updated:} \today.\funding{This work has been partially supported by  the Australian Research Council through the \textsc{Future Fellowship} grant FT220100496 and \textsc{Discovery Project} grant DP22010316, and by the Ministry of Science and Higher Education of the Russian Federation within the framework of state support for the creation and development of World-Class Research Centers ``Digital biodesign and personalized healthcare'' No. 075-15-2022-304.}}}

\author{Rekha Khot\thanks{School of Mathematics, Monash University, 9 Rainforest Walk, Melbourne, VIC 3800, Australia. Present address: Inria 75012, Paris, France (\email{rekha.khot@inria.fr}).}
\and
Andr\'es E. Rubiano\thanks{School of Mathematics, Monash University, 9 Rainforest Walk, Melbourne, VIC 3800, Australia (\email{Andres.RubianoMartinez@monash.edu}).}
\and  
Ricardo Ruiz-Baier\thanks{School of Mathematics, Monash University, 9 Rainforest Walk, Melbourne, VIC 3800, Australia; and
World-Class Research Center ``Digital biodesign and personalized healthcare'', Sechenov First Moscow State Medical University, Moscow, Russia; and 
Universidad Adventista de Chile, Casilla 7-D, Chill\'an, Chile (\email{Ricardo.RuizBaier@monash.edu}).}
}
\date{\today}

\begin{document}

\maketitle
\begin{abstract}
    This paper aims first to perform robust continuous analysis of a mixed nonlinear formulation for stress-assisted diffusion of a solute that interacts with an elastic material, and second to propose and analyse a virtual element formulation of the model problem. The two-way coupling mechanisms between the  Herrmann formulation for linear elasticity and the reaction-diffusion equation (written in mixed form) consist of diffusion-induced active stress and stress-dependent diffusion. The two sub-problems are analysed using the extended  Babuška--Brezzi--Braess theory for perturbed saddle-point problems. The well-posedness of the nonlinearly coupled system is established using a Banach fixed-point strategy under the smallness  assumption on data. The virtual element formulations for the uncoupled sub-problems are proven uniquely solvable by a fixed-point argument in conjunction with appropriate projection operators. We derive the  a priori error estimates, and test the accuracy and performance of the proposed method through computational simulations.
\end{abstract}

\section{Introduction}
The process of diffusion within solid matter can result in the creation of mechanical stresses within the solid material, commonly referred to as chemical or diffusion-induced stresses. This phenomenon highlights a reciprocal relationship between chemical and mechanical driving forces, both of which play a key role in either facilitating or hindering the diffusion process. Stress-assisted diffusion for common materials has been introduced as a general formalism in \cite{wilson1982theory,unger1983theory}. Important examples include Lithium-ion battery cells, silicon rubber and hydrogen diffusion, polymer-based coatings, semiconductor fabrication, oxidation of silicon nanostructures, and enhancing of conductivity properties in soft living tissue. In these processes, heterogeneity (and possibly also anisotropy) of stress localisation affects the patterns of diffusion. Our contribution addresses one of the simplest models for such an interaction, incorporating stress-dependency directly in the diffusion coefficient for a given solute, and assuming that in the absence of stress one recovers Fickian diffusion (see, e.g., \cite{cherubini17, zohdi99}). We consider the regime of linear elastostatics and assume that the total Cauchy stress  is also due to a (possibly nonlinear) diffusion-induced isotropic stress. On the other hand, the diffusion coefficient in the reaction-diffusion equation governing the distribution of a solute is taken as a nonlinear function of the Cauchy stress. The well-posedness of primal formulations of the coupled problem has been studied in, e.g., \cite{lewicka2016local,malaeke2023mathematical}.

Some of the materials mentioned above are nearly incompressible, \cred{which is why we opt for a mixed formulation of the elastostatics in terms of displacement and Herrmann pressure \cite{boffi13,herrmann65}. This} allows us to have stability robustly with respect to the Lam\'e parameter $\lambda$ that goes to infinity when the Poisson \cblue{ratio} approaches $\frac12$. Robustness with respect to the Lam\'e parameter $\mu$ can be obtained from appropriate scaling of the inf-sup condition for the divergence operator, as done in, e.g., \cite{khan2019robust,olshanskii06}. Robustness with respect to the Lam\'e parameters \cred{is also required for the reaction-diffusion sub-problem, which is a priori not trivial since the diffusion coefficient depends on the Cauchy stress and in turn on the Lam\'e coefficients}. Similarly to the diffusion-assisted elasticity, the stress-assisted reaction-diffusion problem is written as a perturbed saddle-point problem, and robustness is sought also with respect to the uniform bounds of the stress-assisted diffusion and of the reaction parameter. 
\cred{To achieve such robustness properties we work with parameter-weighted norms and use the perturbed saddle-point theory from \cite{boon21,braess96penalty}. The continuous analysis is then adequately adapted to the discrete case, here using virtual element methods (VEMs). Note that using a primal formulation for the elasticity equations can lead to volumetric locking for some parameter regimes (and this is observed even using VE discretisations of low order \cite[Section 3]{beirao13elast}).}

There has been quite a few works on mixed FEMs for the stress-assisted diffusion problems, see e.g., \cite{gatica18,gatica19,gatica22,gomez23}. The standard finite element methods (FEM) are  well-adapted for many realistic models, but from the past decade, polygonal methods are fast emerging due to their capabilities for mesh flexibility and a common framework for higher-degrees and for higher-dimensions. The VEM  first introduced in \cite{beirao13} is one of the popular polygonal methods, which can be viewed as the generalisation of FEM. The VEM setting takes into consideration the typical terminologies in FEM like conforming, nonconforming, or mixed methods, and that allows to modify the well-developed tools from FEM in somewhat similar manner. The local VE spaces  are defined as a set of solutions to  problem-dependent partial differential equations combined with boundary conditions based on the underline nature of the method we are interested in (e.g., conforming or nonconforming).  Of course a leap from triangular/rectangular elements  to fairly general polygonal-shaped elements forces to include locally the non-polynomial shape functions in addition to polynomials. But VEM stands different from the other methods in a sense that it does not require to compute explicitly the possibly complex shape functions  and the complete analysis can be performed through degrees of freedom (DoFs) of VE functions and  their appropriate projections from the local discrete space to the polynomial subspace. A vast literature is available on VEM and we name just a few here \cite{brezzi_2014, daveiga15-transport, daveiga15-stokes, daveiga2022stability} which are relevant to the model problem in this paper. Applications of  VEM in  coupled multiphysics are still emerging, see \cite{antonietti2023virtual,gatica2021mixed} for the Navier--Stokes \cred{equations coupled with the heat equation, and \cite{verma2022virtual} for advection-diffusion systems coupled with Biot poroelasticity}. To  the  best of authors'  knowledge, the present paper is the first one addressing coupled stress-assisted diffusion problems. We restrict VE spaces and analysis to 2D for simplicity, but extension to 3D is possible. For example, we refer to \cite{beirao2020stokes} and \cite{dassi2022bend} for Stokes and mixed Darcy problems in three dimensions. We also mention that \cred{we have written an analysis that uses an adaptable common framework for general polynomial degree. Another evident advantage of VEM for this type of problems is the divergence-conforming character of the discretisation for the elasticity equations (which is non-trivial to achieve with, e.g., FEMs or with  other conforming VEM which might require high polynomial degree to obtain this important property).}
\cgreen{We stress that the proofs of well-posedness of the continuous and discrete problems are based on the Banach fixed-point strategy and therefore are associated with a contraction mapping property that holds for sufficiently small data. Should this not be satisfied, the literature contains alternative treatments including augmented Lagrangian techniques (see, e.g., \cite{gatica2021mixed,lovadina2022some} in the context of VEM) and smoothing-type acceleration schemes at the algebraic level (see, e.g., \cite{lott2012accelerated,ramiere2015iterative}). We do not address these issues here, but simply say that the numerical evidence in the paper indicates that the proposed methods are stable and convergent even beyond the parameter regime of small data.}

 \paragraph{Plan of the paper} The contents of this paper have been organised in the following manner. The remainder of this introductory section contains preliminary notational conventions and definitions of useful functional spaces. Section~\ref{sec:model} presents the precise definition of the coupled stress-assisted diffusion model we will address, along with the derivation of its weak formulation in the form of two coupled perturbed saddle-point problems. The unique solvability of the model problem is studied in Section~\ref{sec:wellp}, using the Banach fixed-point approach. Section~\ref{sec:vem} is devoted to designing a family of VEMs for the system under consideration, defining virtual spaces, and recalling appropriate projection operators. In Subsection~\ref{sec:wellp-h} we use an abstract result for discrete perturbed saddle-point problems and again the fixed-point theory, to show that the discrete problem is well-posed. The a priori error analysis for the VE method is carried out in Section~\ref{sec:error-analysis}. Finally, the convergence rates and robustness of the proposed formulation are tested numerically in Section~\ref{sec:results}. \cred{To illustrate mesh flexibility, we also include numerical experiments with different types of polygonal meshes (e.g., kangaroo-shaped geometries).}

\paragraph{Recurrent notation and Sobolev spaces}
Let $D$ be a polygonal Lipschitz bounded domain of $\bbR^2$ 
with boundary $\partial D$. In this paper {we apply} all differential operators row-wise.  Hence, given a tensor function $\bsigma:D\to \bbR^{2\times 2}$ and a vector field $\bu:D\to \bbR^2$, we set the tensor divergence $\bdiv \bsigma:D \to \bbR^2$, the  vector gradient $\bnabla \bu:D \to \bbR^{2\times 2}$, and the symmetric gradient $\beps(\bu) : D \to \bbR^{2\times 2}$ as
$
(\bdiv \bsigma)_i := \sum_j   \partial_j \sigma_{ij}$, $ (\bnabla \bu)_{ij} := \partial_j u_i$,
 {and}  $\beps(\bu) := \frac{1}{2}\left[\bnabla\bu+(\bnabla\bu)^{\tt t}\right]$. 
 The component-wise inner product of two matrices $\bsigma, \,\btau \in\bbR^{2\times 2}$ is defined by $\bsigma:\btau:= \sum_{i,j}\sigma_{ij}\tau_{ij}$. 
 For $s \geq 0$, we denote the usual Hilbertian Sobolev space of scalar functions with domain $D$ by $H^s(D)$, and denote their vector and tensor counterparts as $\bH^s(D)$ and $\bbH^s(D)$, respectively. 
 The norm of $H^s(D)$ is denoted $\norm{\cdot}_{s,D}$ and the corresponding semi-norm $|\cdot|_{s,D}$. We also use the convention  $H^0(D):=L^2(D)$ and the notation $L^2_0(D)$ when the zero mean value condition is imposed, and $(\cdot, \cdot)_D$ to denote the inner product in $L^2(D)$ (similarly for the vector and tensor counterparts). 
The space of vectors in $\bL^2(D)$ with divergence in $L^2(D)$ is denoted $\bH(\vdiv, D)$ and it is a Hilbert space equipped with the corresponding graph norm $\norm{\bzeta}^2_{\vdiv, D}:=\norm{\bzeta}_{0,D}^2+\norm{\vdiv\bzeta}^2_{0, D}$. The same way around with rotational in $L^2(D)$ which is a Hilbert space denoted by $\bH(\vrot, D)$ with the natural norm $\norm{\bzeta}^2_{\vrot,D}:=\norm{\bzeta}_{0,D}^2+\norm{\vrot\bzeta}^2_{0,D}$. Let $\bn$ be the outward unit normal vector to $\partial D$. The Green Formula
can be used to extend the normal trace operator $ \bbC^\infty(\overline D)\ni \bzeta \to (\bzeta|_{\partial D})\cdot \bn$ to a linear continuous mapping $(\cdot|_{\partial D})\cdot \bn:\, \bH(\vdiv, D) \to H^{-\frac{1}{2}}(\partial D)$, where $H^{-\frac{1}{2}}(\partial D)$ is the dual of $H^{\frac{1}{2}}(\partial D)$ (see \cite[Theorem 1.7]{gatica14}). Moreover, the well-known trace inequality
 $\left|\langle (\bzeta|_{\partial D})\cdot \bn, r^* \rangle_{\partial D} \right| \leq {\norm{\bzeta}}_{\vdiv,D}{\norm{r^*}}_{\frac{1}{2},\partial D}$ holds,
where $\langle\cdot,\cdot\rangle_{\partial D}$ denotes the duality pairing between $H^{\frac{1}{2}}(\partial D)$ and $H^{-\frac{1}{2}}(\partial D)$ with respect to the inner product in $L^2(\partial D)$ and the norm in the traces space $H^{\frac{1}{2}}(\partial D)$ is defined as $\norm{r^*}_{\frac{1}{2},\partial D}=\inf_{r|_{\partial D}=r^*}{\norm{r}}_{1,\Omega}$. An analogous argument extends the notion of trace operators to subsets of $\partial D$. 

We shall use the letter $C$ to denote a generic positive constant independent of the mesh size  $h$ and physical constants, which might stand for different values at its different occurrences. Moreover, given any positive expressions $X$ and $Y$, the notation $X \,\lesssim\, Y$  means that \cgreen{a positive constant $C < \infty$ exists such that $X \leq C\, Y$}. 

\section{Governing equations}\label{sec:model}
This section introduces the model problem in detail with the required assumptions on the involved coefficients and the corresponding weak formulation at the end.
\subsection{Model statement}
Let us consider a deformable body occupying  the domain $\Omega$  and satisfying the balance of linear momentum in the stationary regime
\begin{equation}\label{eq:momentum}
-\bdiv\bsigma = \fb \qquad \text{in $\Omega$},
\end{equation}
\cgreen{where $\bsigma$ is the Cauchy stress tensor defined in \eqref{eq:constitutive} below,} and $\fb$ is a vector of external body loads.  The coupling between solid deformation and a solute with concentration $\varphi$ is considered using an active stress approach. Combined with Hooke's law, this  gives the specification 
\begin{equation}\label{eq:constitutive}
 \cgreen{   \bsigma := \cC \beps(\bu) - \ell(\varphi)\bbI = 2\mu \beps(\bu) + \lambda \vdiv\bu\, \bbI- \ell(\varphi)\bbI \qquad \text{in $\Omega$},}
\end{equation}
where $\cC$ is the fourth-order elasticity tensor (symmetric and positive definite), $\mu,\lambda$ are the Lam\'e parameters associated with the material properties of the solid substrate, $\ell$ is a nonlinear function of the solute concentration that modulates the intensity and distribution of (isotropic) active stress (also known as diffusion-induced stress), $\bu$ is the displacement vector, $\beps(\bu)$ is the tensor of infinitesimal strains (symmetrised gradient of displacement), and $\bbI$ denotes the identity tensor in $\bbR^{2\times 2}$. 

Equations \eqref{eq:momentum}-\eqref{eq:constitutive} indicate that the solid deformation will be a response of both external loads and internal stress generation due to the coupling with the species $\varphi$. On the other hand, the solute steady transport in the solid is governed by the reaction-diffusion equation 
\[
    \theta \varphi 
    - \vdiv(\bbM(\bsigma)\nabla\varphi) = g\qquad \text{in $\Omega$},
\]
where $g$ is a given net volumetric source of solute, $\theta$ is a positive model parameter, and $\bbM$ is the stress-assisted diffusion coefficient (a matrix-valued function of total stress, and assumed uniformly bounded away from zero).  This term implies a two-way coupling mechanism between deformation and transport.
We also consider the auxiliary variable of diffusive flux 
\begin{equation}\label{eq:zeta}
\bzeta = \bbM(\bsigma) \nabla \varphi \qquad \text{in $\Omega$}, \end{equation}
and therefore the reaction-diffusion equation reads
\begin{equation}\label{eq:transport}
\theta \varphi 
- \vdiv\bzeta = g \qquad \text{in $\Omega$}.\end{equation}

In \cite{gatica18,gatica19} the contribution of $\ell$ into the active stress is transferred (thanks to \eqref{eq:momentum}) to the loading term as $\fb= \fb(\varphi)$. Here we maintain it as part of the volumetric stress through the Herrmann-type pressure (see \cite{herrmann65} for the classical hydrostatic pressure formulation) which represents the \emph{total volumetric stress} (see  e.g., \cite{kuroki1982boundary}, but note that here we are including both elastic and active components)   
\begin{equation}\label{eq:p}
    \tilde{p} := - \lambda \vdiv\bu + \ell(\varphi) \qquad \text{in $\Omega$},
\end{equation}
giving that the Cauchy stress reads $\bsigma = 2\mu \beps(\bu) - \tilde{p}\bbI$. This approach leads to a system  similar to the total pressure formulation for linear poroelasticity  in parameter-robust and conservative form (see, e.g., \cite{boon21,kumar20}). 

We adopt mixed loading boundary conditions for the coupled problem: the structure is clamped $(\bu = \bzero)$ and a given concentration $\varphi = \varphi_D$ are imposed on $\Gamma_D$,  where the boundary subset $\Gamma_D\subset  \partial \Omega$ is of positive surface measure; and traction and zero solute flux are prescribed ($\bsigma\bn = \bt$ and $\bzeta\cdot\bn = 0$)  on  $\Gamma_N:= \partial\Omega \setminus \Gamma_D$. 

\subsection{Weak formulation}
In view of the boundary conditions, we define the  Hilbert spaces 
\begin{gather*}
\bH^1_D(\Omega):=\{\bv \in \bH^1(\Omega): \bv = \bzero\quad \text{on }\Gamma_D\}, \quad  
\bH_N(\vdiv,\Omega):=\{\bxi \in \bH(\vdiv,\Omega): \bxi\cdot\bn = 0 \quad \text{on }\Gamma_N\},   
\end{gather*}
with the boundary assignment understood in the sense of traces, and consider the following weak formulation for the system composed by \eqref{eq:momentum}, \eqref{eq:p}, \eqref{eq:zeta}, \eqref{eq:transport}. For given $\fb\in \bL^2(\Omega)$, $g\in L^2(\Omega)$, and $\varphi_D\in H^{\frac{1}{2}}(\Gamma_D)$, find 
$(\bu,\tilde{p},\bzeta,\varphi) \in \bH_D^1(\Omega)\times L^2(\Omega) \times \bH_N(\vdiv,\Omega) \times L^2(\Omega)$ such that 
\begin{subequations}\label{eq:weak}
\begin{align}
 2\mu \int_\Omega \beps(\bu):\beps(\bv) - \int_\Omega \tilde{p}\vdiv\bv & = \int_\Omega \fb\cdot\bv \qquad \forall \bv \in   \bH_D^1(\Omega), \\
- \int_\Omega \tilde{q}\vdiv\bu - \frac{1}{\lambda} \int_\Omega \tilde{p}\tilde{q}\cgreen{-\frac{1}{\lambda}\int_\Omega  \ell(\varphi)\tilde{q} } & = \cgreen{0} \qquad \forall \tilde{q}\in L^2(\Omega),\\
\int_\Omega \bbM(\bsigma)^{-1} \bzeta \cdot \bxi + \int_\Omega \varphi \vdiv \bxi & = \langle \varphi_D, \bxi\cdot\bn\rangle_{\Gamma_D} \qquad \forall \bxi\in \bH_N(\vdiv,\Omega),\\ 
\int_\Omega \psi \vdiv\bzeta - \theta \int_\Omega \varphi\psi  
&= - \int_\Omega g\psi \qquad \forall \psi \in L^2(\Omega). 
\end{align}
\end{subequations}
This formulation (mixed for the nonlinear reaction-diffusion part and mixed for the active linear elasticity part) results simply multiplying by test functions and integrating by parts the terms containing the shear and volumetric parts of stress, as well as the diffusive flux. 

\subsection{Two uncoupled perturbed saddle-point problems}
Before we investigate the unique solvability of \eqref{eq:weak}, we will regard the nonlinear coupled problem as two separate perturbed saddle-point problems. Commonly used forms for the stress-assisted diffusion coefficient include (see, e.g., \cite{cherubini17,gatica18,grigoreva19}) 
\[\bbM(\bsigma) = m_0\bbI + m_1 \bsigma + m_2 \bsigma^2, \quad 
\bbM(\bsigma) = m_0\exp(-m_1 \tr\bsigma)\bbI, \quad \bbM(\bsigma) = m_0 \exp(-m_1 \bsigma^{m_2}),  \]
where $m_0,m_1,m_2 \in\mathbb{R}$ are model parameters, with $m_0>0$. On the other hand, a typical form of the active stress include the Hill-type function (i.e., with a modulation saturating at high concentration values see, e.g., \cite{martins1998numerical,murray2003mathematical,murray1984generation}) 
 or a scaling of the given concentration (see \cite{Taralov2015}), defined as
\[\ell(\vartheta)=K_0 + \frac{\vartheta^n}{K_1+\vartheta^n}, \quad \ell(\vartheta)=K_0\vartheta,\]
where $K_0,K_1$ and $n$ (Hill coefficient) are model parameters. From now on, the stress-assisted diffusion will be written explicitly as a function of the infinitesimal strain and of the total volumetric stress $\bbM(\beps(\bw),\tilde{r})$. It is assumed that $\bbM(\cdot,\cdot)$ is symmetric, positive semi-definite and uniformly bounded in $\mathbb{L}^\infty(\Omega)$, likewise for $\bbM^{-1}(\cdot,\cdot)$. More explicitly, there exists a constant $M$ such that $0<\frac{1}{M} \leq M$ and 
\begin{subequations}\label{M-norm}
\begin{align}
\bbM^{-1}(\beps(\bw),\tilde{r}) \in \mathbb{L}^\infty(\Omega) &\qquad \forall \bw\in \bH^1(\Omega),\tilde{r}\in L^2(\Omega), \label{eq:M-Linfty}\\
\frac{1}{M} \bx\cdot \bx \leq \bx\cdot[\bbM^{-1}(\beps(\bw),\tilde{r})\bx] &\qquad \forall \bx \in \mathbb{R}^2,\\
\by\cdot[\bbM^{-1}(\beps(\bw),\tilde{r})\bx]  \leq M \bx\cdot\by &\qquad \forall \bx,\by \in \mathbb{R}^2.
\end{align}

\end{subequations}
For the active stress function $\ell(\cdot)$, we assume that $\ell: L^2(\Omega)\to \cgreen{L^2(\Omega)}$ and satisfies 
\begin{equation}\label{eq:ell-bound}
\norm{\ell(\vartheta)}_{0,\Omega} \lesssim \norm{\vartheta}_{0,\Omega} \qquad \forall \vartheta \in L^2(\Omega). 
\end{equation}
Moreover, we assume that $\bbM^{-1}(\cdot,\cdot)$ and $\ell(\cdot)$ are Lipschitz continuous, i.e., for all $\btau_1,\btau_2\in \bbL^2(\Omega)$ and $\tilde{r}_1,\tilde{r}_2,\vartheta_1,\vartheta_2\in L^2(\Omega)$ there exist positive constants $L_{\bbM}$ and $L_{\ell}$ that satisfy the following bounds
\begin{subequations}
    \begin{align}
        \norm{\bbM^{-1}(\btau_1,\tilde{r}_1)-\bbM^{-1}(\btau_2,\tilde{r}_2)}_{\infty,\Omega}&\leq L_{\bbM}\norm{(\btau_1,\tilde{r}_1)-(\btau_2,\tilde{r}_2)}_{\bbL^2(\Omega) \times L^2(\Omega)},\label{eq:lipschitz-M-bound}\\
        \norm{\ell(\vartheta_1)-\ell(\vartheta_2)}_{0,\Omega}&\leq L_{\ell}\norm{\vartheta_1-\vartheta_2}_{0,\Omega}.\label{eq:lipschitz-l-bound}
    \end{align}
\end{subequations}

For a fixed vector-valued function $\bw \in \bH^1(\Omega)$ and a fixed scalar function $\tilde{r}\in L^2(\Omega)$, we define the following bilinear form $a_2^{\bw,\tilde{r}}:\bH(\vdiv,\Omega)\times \bH(\vdiv,\Omega)\to \bbR$ as 
\[a_2^{\bw,\tilde{r}}(\bzeta,\bxi): =  \int_\Omega \bbM^{-1}(\beps(\bw),\tilde{r}) \bzeta \cdot \bxi \qquad  \forall \bzeta,\bxi\in \bH(\vdiv,\Omega).\]
In addition, for a fixed scalar function $\vartheta\in L^2(\Omega)$, we define the linear functional $G_1^\vartheta: L^2(\Omega)\to \bbR$ as 
\[ G_1^\vartheta(\tilde{q}): = -\frac{1}{\lambda}\int_\Omega  \ell(\vartheta)\tilde{q} \qquad \forall \tilde{q} \in L^2(\Omega).\]
Then set the bilinear forms $a_1: \bH^1(\Omega)\times \bH^1(\Omega)\to \bbR$, $b_1:\bH^1(\Omega)\times L^2(\Omega)\to \bbR$, 
$c_1,c_2:L^2(\Omega)\times L^2(\Omega)\to \bbR$, $b_2:\bH(\vdiv,\Omega)\times L^2(\Omega)\to \bbR$, and linear functionals $F_1 : \bH^1(\Omega)\to \bbR$, $F_2:\bH(\vdiv,\Omega)\to \bbR$, $G_2:L^2(\Omega)\to\bbR$ as follows  
\begin{gather*}
    a_1(\bu,\bv) : = 2\mu \int_\Omega \beps(\bu):\beps(\bv), \quad 
    b_1(\bv,\tilde{q}) :=  - \int_\Omega \tilde{q}\vdiv\bv,\quad
    c_1(\tilde{p},\tilde{q}) : = \int_\Omega \tilde{p}\tilde{q},\quad F_1(\bv):= \int_\Omega \fb\cdot\bv,\\
    b_2(\bxi,\psi):= \int_\Omega \psi \vdiv \bxi,\quad 
    c_2(\varphi,\psi):= \int_\Omega \varphi\psi,\quad F_2(\bxi): = \langle \varphi_D, \bxi\cdot\bn\rangle_{\Gamma_D}, \quad G_2(\psi):= - \int_\Omega g\psi,
\end{gather*}
for all $\bu,\bv \in \bH^1(\Omega)$, $\tilde{p},\tilde{q}, \varphi, \psi \in L^2(\Omega)$, $\bzeta,\bxi\in \bH(\vdiv,\Omega)$, $\varphi_D\in H^{\frac{1}{2}}(\Gamma_D)$.

With these building blocks we define the following system of linear elasticity: for a given $\vartheta \in L^2(\Omega)$ and $\fb \in \bL^2(\Omega)$, find $(\bu,\tilde{p})\in \bH_D^1(\Omega)\times L^2(\Omega)$ such that 
\begin{subequations}\label{eq:weak-elast}
    \begin{align}
  a_1(\bu,\bv) + b_1(\bv,\tilde{p})  &= F_1(\bv) \qquad \forall \bv \in \bH_D^1(\Omega),\\
  b_1(\bu,\tilde{q}) - \frac{1}{\lambda}c_1(\tilde{p},\tilde{q})  &= G_1^\vartheta(\tilde{q}) \qquad \forall \tilde{q} \in L^2(\Omega).
    \end{align}
\end{subequations}
We  also consider the following 
reaction-diffusion equation in weak form: for a given $\bw\in \bH_D^1(\Omega)$, $\tilde{r}\in L^2(\Omega)$, $\varphi_D\in H^{\frac{1}{2}}(\Gamma_D)$ and $g\in L^2(\Omega)$, find $(\bzeta, \varphi)\in \bH_N(\vdiv,\Omega)\times L^2(\Omega)$ such that 
\begin{subequations}\label{eq:weak-transport}
    \begin{align}
  a_2^{\bw,\tilde{r}}(\bzeta,\bxi) + b_2(\bxi,\varphi)  &= F_2(\bxi) \qquad \forall \bxi \in \bH_N(\vdiv,\Omega),\\
  b_2(\bzeta,\psi) - \theta c_2(\varphi,\psi)  &= G_2(\psi) \qquad \forall \psi \in L^2(\Omega).
    \end{align}
\end{subequations}

\section{Well-posedness of the continuous problem}\label{sec:wellp}
The unique solvability of the  elasticity and reaction-diffusion equations will be established using the theory of saddle-point problems with penalty \cite[Lemma 3]{braess96penalty}. First we state an extended version of the Babu\v{s}ka--Brezzi theory from \cite[Theorem 2.1]{boon21}.

\begin{theorem} \label{th:unique-solvability}
Let $V,Q_b$ be Hilbert spaces endowed with the (possibly parameter-dependent) norms $\norm{\cdot}_{V}$ and $\norm{\cdot}_{Q_b}$, let $Q$ be a dense (with respect to the norm $\norm{\cdot}_{Q_b}$) linear subspace of $Q_b$ and three bilinear forms $a(\cdot,\cdot)$ on $V\times V$ (assumed continuous, symmetric and positive semi-definite), $b(\cdot,\cdot)$ on $V\times Q_b$ (continuous), and $c(\cdot,\cdot)$ on $Q\times Q$ (symmetric and positive semi-definite); which define the linear operators $A: V \rightarrow V'$, $B : V \rightarrow Q'_b$ and $C:Q \rightarrow Q'$, respectively. Suppose further that 
\begin{subequations}\label{brezzi-conditions}
\begin{align}\label{brezzi-condition-1}
\norm{\hat{v}}_V^2 \lesssim a(\hat{v},\hat{v}) \qquad \forall \hat{v} \in \mathrm{Ker}(B),\\
\label{brezzi-condition-2}
\norm{q}_{Q_b} \lesssim \sup_{v\in V} \frac{b(v,q)}{{\norm{v}}_V} \qquad \forall q \in Q_b.
\end{align}
\end{subequations}
Assume that $Q$ is complete with respect to the norm $\norm{\cdot}_{Q}^2:=\norm{\cdot}_{Q_b}^2+t^2|\cdot|_c^2$, where $|\cdot|_c^2 := c(\cdot,\cdot)$ is a semi-norm in $Q$. Let $t\in[0,1]$ and set the parameter-dependent energy norm as
$$\norm{(v,q)}_{V\times Q}^2:=\norm{v}_{V}^2+\norm{q}_{Q}^2=\norm{v}_{V}^2+\norm{q}_{Q_b}^2+t^2|q|_c^2.$$
Assume also that the following inf-sup condition holds
\begin{align}\label{braess-condition}
\norm{u}_{V} \lesssim \sup_{(v,q)\in V\times Q} \cblue{\frac{a(u,v)+b(u,q)}{\norm{(v,q)}_{V\times Q}}} \qquad \forall u \in V.
\end{align}
Then, for every $F\in V'$ and $G\in Q'$, there exists a unique $(u,p)\in V\times Q$ satisfying
\begin{subequations}
\begin{align*}
a(u,v)+ b(v,p) = F(v) &\qquad \forall v\in V,\\
b(u,q) - t^2c(p,q) = G(q) &\qquad \forall q\in Q.
\end{align*}
\end{subequations}
Furthermore, the following continuous dependence on data holds 
\begin{align} \label{dependence_data}
    \norm{(u,p)}_{V\times Q} &\lesssim \norm{F}_{V'}+\norm{G}_{Q'}.
\end{align}
\end{theorem}

\begin{remark}
The form $c(\cdot,\cdot)$ does not require to be bounded in the norm $\norm{\cdot}_{Q_b}$. 
The  space $Q_b$ allows us to absorb the dependency of the model parameters directly in the norms, yielding  a robust continuous dependence on data. That is, the hidden constant $C$ associated with \eqref{dependence_data} only depends on the constants associated with the Brezzi--Braess conditions \eqref{brezzi-conditions}-\eqref{braess-condition}, and the continuity constants.
\end{remark}

\subsection{Unique solvability of the decoupled sub-problems}\label{sec:decoupled}
In this section, we verify that the assumptions of Theorem \ref{th:unique-solvability} are satisfied for \eqref{eq:weak-elast} and \eqref{eq:weak-transport}. \cred{Some of the arguments employed in these proofs are standard and we omit them. For the sake of completeness, their explicit derivations can be found in reference \cite{krr_preprint}.}

Let us adopt the following notation for the functional spaces for displacement and total volumetric stress $\bV_1 := \bH_D^1(\Omega)$ and $Q_1 = Q_{b_1} := L^2(\Omega)$,  equipped with the following scaled norms and seminorms 
\begin{align*}
    \norm{\bu}_{\bV_1}^2 := 2\mu\norm{\beps(\bu)}_{0,\Omega}^2, \quad 
    \norm{\tilde{p}}_{Q_1}^2 := \norm{\tilde{p}}_{Q_{b_1}}^2 + \frac{1}{\lambda}|\tilde{p}|_{c_1}^2, \quad
    \norm{\tilde{p}}_{Q_{b_1}}^2 := \frac{1}{2\mu}\norm{\tilde{p}}_{0,\Omega}^2, \quad 
    |\tilde{p}|_{c_1}^2 := \norm{\tilde{p}}_{0,\Omega}^2.
\end{align*}
On the other hand, let us denote the functional spaces for diffusive flux and concentration as $\bV_2 = \bH_N(\vdiv,\Omega)$ and $Q_2 = Q_{b_2} = L^2(\Omega)$, furnished with the following norms and seminorms 
\begin{align*}
    &\norm{\bzeta}_{\bV_2}^2 := \norm{\bzeta}_{\bbM}^2 + M\norm{\vdiv \bzeta}_{0,\Omega}^2, \quad 
    \norm{\varphi}_{Q_2}^2 :=  \norm{\varphi}_{Q_{b_2}}^2 + \theta |\varphi|_{c_2}^2, \\
    &\norm{\bzeta}_{\bbM}^2 := a_2^{\bw,\tilde{r}}(\bzeta,\bzeta),
    \quad
    \norm{\varphi}_{Q_{b_2}}^2 := \frac{1}{M}\norm{\varphi}_{0,\Omega}^2, \quad
    |\varphi|_{c_2}^2 := \norm{\varphi}_{0,\Omega}^2.
\end{align*}
\begin{lemma}\label{lem:elast}
For a fixed $\vartheta\in Q_2$, assume that $1\leq \lambda$ and $0<\mu$. Then, there exists a unique solution $(\bu,\tilde{p})\in \bV_1\times Q_1$ to \eqref{eq:weak-elast}. \cgreen{Furthermore, 
there exists a constant $C_1>0$ independent of the physical parameters, such that 
\begin{align}\label{dependence-data-elast}
    \norm{(\bu,\tilde{p})}_{\bV_1\times Q_1} \leq C_1  \left(\norm{F_1}_{\bV'_1} + \norm{G_1^{\vartheta}}_{Q'_1}\right).
\end{align}
}
\end{lemma}
\begin{proof}
 Clearly, $Q_1$ is a dense linear subspace of $Q_{b_1}$ with the respective norms. 
\cred{Thanks to \cgreen{K\"orn's} inequality 
 we readily have} that $\norm{\beps(\bv)}_{0,\Omega}$ forms  an equivalent norm to the standard $\bH^1_D$-norm in $\bV_1$. Next, we invoke the Cauchy--Schwarz inequality to conclude the \cgreen{boundedness of  $a_1(\cdot,\cdot)$ on $\bV_1$}.
In addition,  by definition of $\norm{\cdot}_{\bV_1}$, we have that $a_1(\cdot,\cdot)$ is symmetric, positive semi-definite and coercive in $\bV_1$ (giving, in particular,  condition \eqref{brezzi-condition-1}). The Cauchy--Schwarz inequality, the triangle inequality  $\norm{\vdiv \bv}_{0,\Omega}\leq \norm{\nabla\bv}_{0,\Omega}$, and  \cgreen{K\"orn's} inequality implies the \cred{boundedness of $b_1(\cdot,\cdot)$ on $\bV_1\times Q_{b_1}$.}
Thanks to the structure of the problem and of the scaled norms, in the context of Theorem~\ref{th:unique-solvability} we can identify $t=\frac{1}{\sqrt{\lambda}}\in (0,1]$, since $1\leq \lambda$. Moreover, $c_1(\cdot,\cdot)=(\cdot,\cdot)_{L^2(\Omega)}$  is symmetric and positive semi-definite. On the other hand, \cgreen{K\"orn's} and  Cauchy--Schwarz inequalities imply that  
\begin{align*}
    \left|F_1(\bv)\right| = \left|\int_{\Omega} \fb \cdot \bv\right| \leq \norm{\fb}_{0,\Omega}\norm{\bv}_{0,\Omega} \lesssim  \norm{\fb}_{0,\Omega}\norm{\beps(\bv)}_{0,\Omega} = \norm{\fb}_{Q_{b_1}}\norm{\bv}_{\bV_1}  \quad \forall \bv \in \bV_1.
\end{align*} 
Similarly, for a given $\vartheta\in Q_2$, the property \eqref{eq:ell-bound} and Cauchy--Schwarz inequality leads to  
\begin{align*}
    \left|G_1^\vartheta(\tilde{q})\right| = \left|-\frac{1}{\lambda} \int_{\Omega} \ell(\vartheta) \tilde{q}\right| \leq \frac{1}{\lambda}\norm{\ell(\vartheta)}_{0,\Omega}\norm{\tilde{q}}_{0,\Omega}\lesssim \norm{\vartheta}_{Q_1}\norm{\tilde{q}}_{Q_1} \quad \forall \tilde{q} \in Q_1,
\end{align*}
and this yields that $G_1^\vartheta(\cdot)$ is continuous in $Q_1$.
\cred{$F_1(\cdot)$ and $G_1^\vartheta(\cdot)$ are continuous in $\bV_1$ and $Q_1$, respectively.} 
For condition \eqref{brezzi-condition-2}, note that  the surjectivity of the divergence operator $\vdiv: \bH_D^1(\Omega) \rightarrow L^2(\Omega)$ (see \cite[Lemma 53.9]{ern22}) gives the following inf-sup condition
$$\norm{\tilde{q}}_{0,\Omega} \lesssim \sup_{\bv\in \bH^1_D(\Omega)} \frac{b_1(\bv,\tilde{q})}{\norm{\nabla \bv}_{0,\Omega}} \qquad \forall \tilde{q} \in L^2(\Omega).$$ 
Scaling this bound with 
 $\frac{1}{\sqrt{2\mu}}$, and using the equivalence between $\norm{\beps(\bv)}_{0,\Omega}$ and $\norm{\nabla \bv}_{0,\Omega}$ we obtain 
$$\norm{\tilde{q}}_{Q_{b_1}} \lesssim \sup_{\bv\in \bV_1} \frac{b_1(\bv,\tilde{q})}{\norm{\bv}_{\bV_1}} \qquad \forall \tilde{q} \in Q_{b_1}.$$ 
Finally, the inf-sup condition \eqref{braess-condition} is obtained as follows, for $\bu \in \bV_1$, let $\bv = \bu$ and $\tilde{q}=0$.  We have
\begin{align*}
    a_1(\bu,\bu)+b_1(\bu,0) = a_1(\bu,\bu) = \norm{\bu}_{\bV_1}^2 \quad \text{ and }  \quad \norm{(\bu,0)}^2_{\bV_1\times Q_1} = \norm{\bu}_{\bV_1}^2.
\end{align*}
Therefore,
$$\norm{\bu}_{\bV_1} \lesssim \sup_{(\bv,\tilde{q})\in \bV_1\times Q_1} \cblue{\frac{a_1(\bu,\bv)+b_1
(\bu,\tilde{q})}{\norm{(\bv,\tilde{q})}_{\bV_1\times Q_1}}} \qquad \forall \bu \in \bV_1,$$
verifying  the conditions of Theorem~\ref{th:unique-solvability}. Then the unique solution \cgreen{satisfies \eqref{dependence-data-elast}, and the proof is complete.}
\end{proof}
\begin{lemma}\label{lem:trans}
For a fixed pair $(\bw,\tilde{r})\in \bV_1\times Q_1$, suppose that 
\cgreen{suppose that $\theta \leq \frac{1}{M}$ 
in addition to the assumptions of Section~\ref{sec:model}}.
Then, there exists a unique pair $(\bxi,\varphi)\in \bV_2\times Q_2$ solution to \eqref{eq:weak-transport}. \cgreen{Further,  the following bound is satisfied 
\begin{align}\label{dependence-data-diff}
    \norm{(\bzeta,\varphi)}_{\bV_2\times Q_2} \leq C_2 \left( \norm{F_2}_{\bV'_2} + \norm{G_2}_{Q'_2}\right),
\end{align}
where the constant $C_2>0$ does not depend on the physical parameters.}
\end{lemma}

\begin{proof}
It is straightforward to see that $Q_{b_2}$ is a dense linear subspace of $Q_2$ with the corresponding norms. From the properties of $\bbM$ in \eqref{M-norm}, for given $(\bw,\tilde{r})\in \bV_1\times Q_1$, the norm $\norm{\cdot}_{\bV_2}$ is a equivalent norm in $\bH_N(\vdiv,\Omega)$. \cred{Indeed, the assumptions $\frac{1}{M}\leq M$ and \eqref{M-norm} lead to
\begin{align*}
        \frac{1}{M}\norm{\bxi}_{\vdiv,\Omega} = \frac{1}{M}\norm{\bxi}_{0,\Omega} + \frac{1}{M}\norm{\vdiv \bxi}_{0,\Omega} \leq \norm{\bxi}_{\mathbb{M}} + M \norm{\vdiv \bxi}_{0,\Omega} = \norm{\bxi}_{\bV_2},\text{ and } \norm{\bxi}_{\bV_2}\leq M\norm{\bxi}_{\vdiv,\Omega.}
    \end{align*}}
Therefore, the bilinear form $a_2^{\bw,\tilde{r}}(\cdot,\cdot)$ is positive semi-definite, continuous over $\bV_2$, and coercive in the nullspace of the linear operator $B_2$ 
\[\mathrm{Ker}(B_2) =\{\bxi \in \bV_2: \int_\Omega \varphi \vdiv \bxi  = 0\text{, } \forall \varphi\in Q_2\}=\{\bxi \in \bV_2: \vdiv \bxi = 0\}. \]
In addition, 
from Cauchy--Schwarz inequality and the definition of the norms \cgreen{we arrive at the continuity of $b_2(\cdot,\cdot)$ on $\bV_2\times Q_{b_2}$.} 
Furthermore, the bilinear operator $c_2(\cdot,\cdot)$ coincides with $(\cdot,\cdot)_{L^2(\Omega)}$ which is symmetric and positive semi-definite and $t = \sqrt{\theta}\in [0,1]$. Next, the trace inequality together with the equivalence between $\norm{\bxi}_{\vdiv,\Omega}$ and $\norm{\bxi}_{\bV_2}$ yields the continuity of $F_2(\cdot)$ \cred{on $\bV_2$}. 
Likewise, $G_2(\cdot)$ \cred{is continuous on $Q_2$.} 
To prove condition \eqref{brezzi-condition-2}, note that the divergence operator $\vdiv: \bH_N(\vdiv,\Omega) \rightarrow L^2(\Omega)$ is surjective (see \cblue{\cite[Lemma 51.2]{ern22})}. Consequently, we obtain an inf-sup condition in non-weighted norms, which 
is then multiplied 
by $\frac{1}{\sqrt{M}}$, giving  
$$\norm{\varphi}_{Q_{b_2}} \lesssim \sup_{\bzeta\in \bV_2} \frac{b_2(\bzeta,\varphi)}{\norm{\bzeta}_{\bV_2}} \qquad \forall \varphi \in Q_{b_2}.$$ 
Lastly, we verify the Braess condition \eqref{braess-condition} as follows: for $\bzeta \in \bV_2$, let $\bxi = \bzeta$ and $\varphi=M\vdiv\bzeta$, then 
\begin{align*}
    &a_2^{\bw,\tilde{r}}(\bzeta,\bzeta)+b_2(\bzeta,M\vdiv\bzeta) = a_2^{\bw,\tilde{r}}(\bzeta,\bzeta) + M b_2(\bzeta,\vdiv\bzeta) = \norm{\bzeta}_{\bV_2}^2,\\
    \norm{(\bzeta,M\vdiv\bzeta)}^2_{\bV_2\times Q_2} &= \norm{\bzeta}_{\bV_2}^2+\norm{M\vdiv\zeta}_{Q_{b_2}}^2+\theta|M\vdiv\bzeta|_{c_2}^2\lesssim \norm{\bzeta}_{\bV_2}^2 + \frac{1}{M}\norm{M\vdiv\bzeta}_{0,\Omega}^2 \lesssim \norm{\bzeta}_{\bV_2}^2.
\end{align*}
Therefore,
$$\norm{\bzeta}_{\bV_2} \lesssim \sup_{(\bxi,\varphi)\in \bV_2\times Q_2} \cblue{\frac{a_2^{\bw,\tilde{r}}(\bzeta,\bxi)+b_2(\bzeta,\varphi)}{\norm{(\bxi,\varphi)}_{\bV_2\times Q_2}}} \qquad \forall \bzeta \in \bV_2.$$
This verifies the conditions in  Theorem~\ref{th:unique-solvability} and therefore we also have \cgreen{that the asserted continuous dependence on data \eqref{dependence-data-diff} holds}. 
\end{proof}
\begin{remark}
The proofs of Lemmas~\ref{lem:elast} and \ref{lem:trans} provide further details on parameter requirements. The conditions $0<\mu$ and $1\leq \lambda$ are needed to properly define weighted norms and the condition over \cblue{$t$} in Theorem~\ref{th:unique-solvability} is carried over to \cblue{$\frac{1}{\sqrt{\lambda}}$}. On the other hand, the bound 
\cgreen{$\theta \leq \frac{1}{M} \leq1$ is the only additional assumption and it is} fundamental in the proof of  \eqref{braess-condition}. The weighted norms are well defined, and \cblue{$\sqrt{\theta}$} plays the role of \cblue{$t$} in Theorem~\ref{th:unique-solvability}. We could relax the dependence of $\theta$  on $M$, but this would require a substantially more involved metric for $\varphi$ \cblue{(see \cite{boon21})}.
\end{remark}

\subsection{Fixed-point strategy} \label{sec:fix-point}
Let us now define the following solution operators 
\[\cS_1: Q_2 \to \bV_1\times Q_1, \quad 
\vartheta \mapsto \cS_1(\vartheta)= (\cS_{11}(\vartheta),\cS_{12}(\vartheta)) := (\bu,\tilde{p}),
\]
where $(\bu,\tilde{p})$ is the unique solution to \eqref{eq:weak-elast}, confirmed in Section~\ref{sec:decoupled}; and 
\[\cS_2: \bV_1\times Q_1 \to \bV_2\times Q_2, \quad 
(\bw,\tilde{r}) \mapsto \cS_2(\bw,\tilde{r}) = (\cS_{21}(\bw,\tilde{r}),\cS_{22}(\bw,\tilde{r})) := (\bzeta,\varphi),
\]
where $(\bzeta,\varphi)$ is the unique solution to \eqref{eq:weak-transport}, also confirmed in Section~\ref{sec:decoupled}. 

The nonlinear problem \eqref{eq:weak} is thus equivalent to the following fixed-point equation:
\[\text{Find $\varphi\in Q_2$ such that $\cA(\varphi) = \varphi$},\]
where $\cA: Q_2\to Q_2$ is defined as 
$\varphi \mapsto \cA(\varphi) := (\cS_{22}\circ \cS_{1})(\varphi)$. 
Next, we show the Lipschitz continuity of $\cA$, i.e., the Lipschitz continuity of the solution operators $\cS_1$ and $\cS_2$.

\begin{lemma}\label{lipschitz-elast}
    The operator $\cS_1$ is Lipschitz continuous with Lipschitz constant $L_{\cS_1} = \sqrt{M} C_1 L_\ell$.
\end{lemma}

\begin{proof}
    Let $\vartheta_1,\vartheta_2 \in Q_2$ and $(\bu_1,\tilde{p}_1),(\bu_2,\tilde{p}_2)$ be the unique solutions to the problem \eqref{eq:weak-elast}. Define the following auxiliary problem: find $(\bu,\tilde{p}) \in \bV_1\times Q_1$ such that
    \begin{align*}
      a_1(\bu,\bv) + b_1(\bv,\tilde{p}) &= 0 \qquad \forall \bv \in \bV_1,\\
      b_1(\bu,\tilde{q}) - \frac{1}{\lambda}c_1(\tilde{p},\tilde{q}) &= (G_{1}^{\vartheta_1}-G_{1}^{\vartheta_2})(\tilde{q}) \qquad \forall \tilde{q} \in Q_1.
    \end{align*}
    Note that  $G_{1}^{\vartheta_1}-G_{1}^{\vartheta_2}
    \in Q_1'$.
       Then, the unique solution is given by $(\bu,\tilde{p})=(\bu_1-\bu_2,\tilde{p}_1-\tilde{p}_2)$. Therefore, we apply the continuous dependence on data \eqref{dependence-data-elast} together with \eqref{eq:lipschitz-l-bound} to obtain
    \begin{align*}
    \norm{\cS_1(\vartheta_1)-\cS_1(\vartheta_2)}_{\bV_1\times Q_1} &= \norm{(\bu_1-\bu_2,\tilde{p}_1-\tilde{p}_2)}_{\bV_1\times Q_1} \\
    &\leq C_1 \sup_{\tilde{q}\in Q_1} \frac{|(G_{1}^{\vartheta_1}-G_{1}^{\vartheta_2})(\tilde{q})|}{\norm{\tilde{q}}_{Q_1}}\leq C_1 \norm{\ell(\vartheta_1)-\ell(\vartheta_2)}_{0,\Omega}\\
    &\leq C_1 L_{\ell} \norm{\vartheta_1-\vartheta_2}_{0,\Omega} \leq \sqrt{M} C_1 L_{\ell} \norm{\vartheta_1-\vartheta_2}_{Q_2},
    \end{align*}
where the norm equivalence has been used in the last step. 
\end{proof}

\begin{lemma}\label{lipschitz-transport}
    The operator $\cS_2$ is Lipschitz continuous with Lipschitz constant
    \begin{equation*}
        L_{\cS_2} = \max\left\{\frac{1}{\sqrt{2\mu}},\sqrt{2\mu} \right\}\sqrt{M^{3}}C_2^2L_{\bbM}\left(\norm{\varphi_D}_{\frac{1}{2},\Gamma_D} + \norm{g}_{0,\Omega}\right).
    \end{equation*} 
\end{lemma}

\begin{proof}
    Analogously, as in Lemma~\ref{lipschitz-elast}, let $(\bw_1,\tilde{r}_1),(\bw_2,\tilde{r}_2) \in \bV_1\times Q_1$ be such that $(\bzeta_1,\varphi_1),(\bzeta_2,\varphi_2)$ are the unique solutions of their respective problem \eqref{eq:weak-transport} and consider the following auxiliary problem: find $(\bzeta,\varphi)\in \bV_2\times Q_2$ such that
    \begin{align*}
      a_2^{\bw_1,\tilde{r}_1}(\bzeta,\bxi) + b_2(\bxi,\varphi)  &= F^*(\bxi) \qquad \forall \bxi \in \bV_2,\\
      b_2(\bzeta,\psi) - \theta c_2(\varphi,\psi)  &= 0 \qquad \forall \psi \in Q_2,
    \end{align*}
    where
    $F^*(\bxi) = \int_{\Omega} (\bbM^{-1}(\beps(\bw_1),\tilde{r}_1)-\bbM^{-1}(\beps(\bw_2),\tilde{r}_2))\bzeta_2 \cdot \bxi.$
    The problem above results from subtracting the problems associated with $(\bw_1,\tilde{r}_1)$ and $(\bw_2,\tilde{r}_2)$ and rewriting  $a_2^{\bw_1,\tilde{r}_1}(\bzeta_2,\bxi)$ in the first equation. Note that $F^*(\bxi)\in \bV_2'$. Therefore, thanks to Lemma~\ref{lem:trans}, the unique solution is given by $(\bzeta,\varphi)=(\bzeta_1-\bzeta_2,\varphi_1-\varphi_2)$. Using \eqref{dependence-data-diff}, H\"older's inequality and \eqref{eq:lipschitz-M-bound} we get
    \begin{align*}
    &\norm{\cS_2(\bw_1,\tilde{r}_1)-\cS_2(\bw_2,\tilde{r}_2)}_{\bV_2\times Q_2} =\norm{(\bzeta_1-\bzeta_2,\varphi_1-\varphi_2)}_{\bV_2\times Q_2}  \leq C_2 \sup_{\bxi\in \bV_2} \frac{|F^*(\bxi)|}{\norm{\bxi}_{\bV_2}}\\
    &\quad \leq \sqrt{M} C_2 \norm{\bbM^{-1}(\beps(\bw_1),\tilde{r}_1)-\bbM^{-1}(\beps(\bw_2),\tilde{r}_2)}_{\infty,\Omega}\norm{\bzeta_2}_{0,\Omega}\\
    &\quad \leq M C_2 L_\bbM \norm{(\beps(\bw_1),\tilde{r}_1)-(\beps(\bw_2),\tilde{r}_2)}_{\bbL^2(\Omega)\times L^2(\Omega)} ( \norm{\bzeta_2}_{\bV_2} + \norm{\varphi_2}_{Q_2})\\
    &\quad \leq \max\left\{\frac{1}{\sqrt{2\mu}},\sqrt{2\mu} \right\}\sqrt{M^{3}} C_2^2L_{\bbM}\left(\norm{\varphi_D}_{\frac{1}{2},\Gamma_D} + \norm{g}_{0,\Omega}\right)\norm{(\bw_1,\tilde{r}_1)-(\bw_2,\tilde{r}_2)}_{\bV_1\times Q_1}.
    \end{align*}  
\end{proof}

\begin{lemma}\label{lipschitz:cA}
    The operator $\cA$ is Lipschitz continuous with Lipschitz constant $L_{\cA}=L_{\cS_2}L_{\cS_1}$.
\end{lemma}
\begin{proof}
We can combine the results from Lemmas~\ref{lipschitz-elast} and \ref{lipschitz-transport} to assert that
    \begin{align*}
        \norm{\cA(\varphi_1)-\cA(\varphi_2)}_{Q_2} &= \norm{\cS_{22}(\cS_1(\varphi_1))-\cS_{22}(\cS_1(\varphi_2))}_{Q_2} \\
        &\leq L_{\cS_2} \norm{\cS_1(\varphi_1)-\cS_1(\varphi_2)}_{\bV_1\times Q_1}\leq L_{\cS_2} L_{\cS_1} \norm{\varphi_1-\varphi_2}_{Q_2},
    \end{align*} for any $\varphi_1,\varphi_2 \in Q_2$, 
    where we have also used the definition of the solution operators. 
\end{proof}
Finally, let us define the following set 
\begin{equation}\label{w}
\cblue{W} =\left\{ \cblue{w} \in Q_2 \colon \norm{w}_{Q_2} \leq \sqrt{M} C_2 \left(\norm{\varphi_D}_{\frac{1}{2},\Gamma_D} + \norm{g}_{0,\Omega}\right) \right\},    
\end{equation}
and note that the operator $\cA$ satisfies that \cgreen{$\cA(Q_2)\subseteq W$}. Indeed, from the continuous dependence on data \eqref{dependence-data-diff}, it is readily seen that 
$${\norm{\cgreen{\mathcal{A}}(\varphi)}}_{Q_2} = \norm{S_{22}(S_1(\varphi))}_{Q_2} \leq \norm{S_2(S_1(\varphi))}_{\bV_2\times Q_2} \leq  \sqrt{M} C_2 \left(\norm{\varphi_D}_{\frac{1}{2},\Gamma_D} + \norm{g}_{0,\Omega}\right),$$
\cgreen{which implies that $\cA(W)\subseteq W$}. The following theorem provides the well-posedness of the coupled problem.
\begin{theorem}\label{coupled-well-poss}
    \cblue{Under the assumptions of Lemmas~\ref{lem:elast}-\ref{lem:trans}, let $W$} be as in \eqref{w} and  \cgreen{the small data be given such that  $$ C_1 L_\ell \max\left\{\frac{1}{\sqrt{2\mu}},\sqrt{2\mu} \right\}M^{2}C_2^2 L_{\bbM}\left(\norm{\varphi_D}_{\frac{1}{2},\Gamma_D} + \norm{g}_{0,\Omega}\right) < 1.$$}
    Then the operator $\cA$ has a unique fixed point \cblue{$\varphi\in W$}. Equivalently, the coupled problem \eqref{eq:weak} has an unique solution $(\bu,\tilde{p},\bzeta,\varphi) \in \bV_1\times Q_1 \times \bV_2 \times Q_2$ and the following continuous dependence on data holds
    \begin{subequations}\label{eq:cont-dep-coupled}
    \begin{align}
        \norm{(\bu,\tilde{p})}_{\bV_1\times Q_1} &\leq C_1 \left( \norm{F_1}_{\bV'_1} + \norm{G^\varphi_1}_{Q'_1} \right),\\
        \norm{(\bzeta,\varphi)}_{\bV_2\times Q_2} &\leq C_2 \left( \norm{F_2}_{\bV'_2} +\norm{G_2}_{Q'_2} \right),
    \end{align}
    \end{subequations}
    where the corresponding constants $C_1$ and $C_2$ do not depend on the physical parameters. 
\end{theorem}
\begin{proof}
\cgreen{The unique solvability} follows directly from Lemma~\ref{lipschitz:cA} together with the Banach fixed-point theorem, \cgreen{whereas the stability bounds \eqref{eq:cont-dep-coupled} are a direct consequence of the a priori estimates from Lemmas~\ref{lem:elast}-\ref{lem:trans}.}
\end{proof}

\begin{remark}\cgreen{
Note that in the regime of large Lam\'e second parameter ($1\ll \mu$), the following bound holds for the given data
    \begin{align*}
        \norm{\varphi_D}_{\frac{1}{2},\Gamma_D} + \norm{g}_{0,\Omega} < \frac{1}{M^2 \sqrt{2\mu}} \frac{1}{C_2^2 C_1 L_\mathcal{M} L_\ell}.
    \end{align*}}

\end{remark}

\section{Virtual element  discretisation}\label{sec:vem}
The aim of this section is to introduce and analyse a VEM for each decoupled problem \eqref{eq:weak-elast} and \eqref{eq:weak-transport} based on \cite{beirao13,daveiga15-transport,daveiga15-stokes}. Next, similarly as for the continuous problem, we employ a discrete fixed-point argument to prove the well-posedness of the coupled discrete problem. 

\paragraph{Mesh assumptions and recurrent notation} Let $\mathcal{T}_h$ be a decomposition of $\Omega$ into polygonal elements $E$ with  diameter $h_E$ and let $\mathcal{E}_h$ be the set of edges $e$ of $\mathcal{T}_h$ with length $h_e$. We assume that for every  $E$ there exists $\rho_E>0$ such that $E$ is star-shaped with respect to every point of a disk with radius $\rho_E h_E$ and in addition $h_e\geq \rho_E h_E$ for every edge $e$ of $E$. When considering a sequence  $\{\mathcal{T}_h\}_h$ we assume that $\rho_E\geq \rho>0$ for some $\rho$ independent of $E$, where $h$ is the maximum diameter over $\mathcal{T}_h$.

Given  an integer $k\geq 0$ the space of polynomials of \cgreen{degree smaller or equal to $k$} 
on $E$ is denoted by $\mathcal{P}_k(E)$ and the space of the gradients of polynomials of grade $\leq k+1$ on $E$ is denoted as $\mathcal{G}_k(E) := \nabla \mathcal{P}_{k+1}(E)$ with standard notation $\mathcal{P}_{-1}(E)=\{0\}$ for $k=-1$. The space $\mathcal{G}_k^\oplus(E)$ denotes the complement of the space $\mathcal{G}_k(E)$ in the vector polynomial space $(\mathcal{P}_k(E))^2$,  that is, $(\mathcal{P}_k(E))^2 = \mathcal{G}_k(E) \oplus  \mathcal{G}_k^\oplus(E)$. \cred{In particular, following \cite{veiga19}, we set $\mathcal{G}_k^\oplus(E):= \bx^\perp \mathcal{P}_{k-1}(E)$ where $\bx^\perp := (y,-x)^{\tt t}$.} Next, we define the continuous space of polynomials along the boundary $\partial E$ of $E$ as
$$\mathcal{B}_k(\partial E) := \left\{ v\in C^0(\partial E) \colon v|_{e} \in \mathcal{P}_k(e), \quad \forall e\subset \partial E \right\}.$$
 Let $\bx_E=(x_E,y_E)^{\tt t
 }$ denote the barycentre of $E$ and let $\mathcal{M}_k(E)$ be the set of $(k+1)(k+2)/2$ scaled monomials
\[\mathcal{M}_k(E):=\left\{ \left( \frac{\bx-\bx_E}{h_E} \right)^{\cblue{\boldsymbol{\alpha}}}, 0\leq |\cblue{\boldsymbol{\alpha}}|\leq k \right\},\]
where \cblue{$\boldsymbol{\alpha}=(\alpha_1,\alpha_2)$} is a non-negative multi-index with \cblue{$|\boldsymbol{\alpha}|=\alpha_1+\alpha_2$ and $\bx^{\boldsymbol{\alpha}}=x^{\alpha_1}y^{\alpha_2}$} for $\bx = (x,y)^{\tt t}$. In particular, we can take the basis of $\mathcal{G}_{k}(E)$ and $\mathcal{G}_{k}^\oplus(E)$ as $\mathcal{M}_{k}^{\nabla}(E):=\nabla\mathcal{M}_{k+1}(E)\setminus\{\mathbf{0}\}$ and $\mathcal{M}_{k}^{\oplus}(E):= \mathbf{m}^\perp \mathcal{M}_{k-1}(E)$, respectively, where $\mathbf{m}^\perp := (\frac{y-y_E}{h_E},\frac{x_E-x}{h_E})^{\tt t}$ and $(\mathcal{M}_k(E))^2 = \mathcal{M}_k^{\nabla}(E) \oplus  \mathcal{M}_k^{\oplus}(E)$ holds. 

\subsection{Discrete formulation for the elasticity problem}
In order to state the discrete spaces for \eqref{eq:weak-elast} we adapt the approach from \cite{daveiga15-stokes} to our formulation which involves the presence of the operator $\beps$ in the bilinear form $a_1(\cdot,\cdot)$ and the  active-stress term $\ell(\vartheta)$ in $G_1^{\vartheta}(\cdot)$.

\paragraph{VE spaces and DoFs} For $k_1\geq 2$, the VE space for displacement locally solves the Stokes problem and is defined by
\begin{align*}
\bV_1^{h,k_1}(E) = \{ \bv\in \bH^1(E) \colon &\bv|_{\partial E}\in (\mathcal{B}_{k_1}(\partial E))^2,\; \vdiv\bv \in \mathcal{P}_{k_1-1}(E),\\ 
&-2\mu\bdiv\beps(\bv) - \nabla s \in \mathcal{G}_{k_1-2}^\oplus(E) \text{ for some } s\in L^2_0(E)\}.
\end{align*}
Observe that $(\mathcal{P}_{k_1}(E))^2\subseteq \bV_1^{h,k_1}(E)$. For the discrete pressure  space $Q_1^{h,k_1}(E)$, we can just  locally consider the polynomials $\mathcal{P}_{k_1-1}(E)$. Then, the global discrete spaces are defined as 
    \begin{align*}
        \bV_1^{h,k_1} &= \{ \bv \in \bV_1 \colon \bv|_E \in \bV_1^{h,k_1}(E), \ \forall E\in \mathcal{T}_h\},\qquad 
        Q_1^{h,k_1} = \{ \tilde{q}\in Q_1 \colon \tilde{q}|_E\in Q_1^{h,k_1}(E), \ \forall E\in \mathcal{T}_h \}.
    \end{align*}
The DoFs for $\bv_h \in \bV_1^{h,k_1}(E)$ can be taken as follows
\begin{subequations}\label{dofsV1h}
\begin{align}
&\bullet \text{The values of } \bv_h \text{ at the vertices of $E$}, \label{D11}\\
&\bullet \text{The values of } \bv_h \text{ on the $k_1-1$ internal Gauss--Lobatto quadrature points on each edge of}\;  E, \label{D12}\\
&\bullet \int_E (\vdiv\bv_h) m_{k_1-1}, \quad \forall m_{k_{1}-1} \in \mathcal{M}_{k_{1}-1}(E)\setminus\{1\},\label{D13}\\
&\bullet \int_E \bv_h \cdot \mathbf{m}_{k_{1}-2}^{\oplus}, \quad \forall \mathbf{m}_{k_{1}-2}^{\oplus} \in \mathcal{M}_{k_{1}-2}^{\oplus}(E).\label{D14}
\end{align}
\end{subequations}
On the other hand, the DoFs for $\tilde{q}_h\in Q_1^{h,k_1}(E)$ are selected as
\begin{equation}
\bullet \int_E \tilde{q}_h m_{k_{1}-1}, \quad \forall m_{k_{1}-1}\in \mathcal{M}_{k_{1}-1}(E). \label{D15}
\end{equation}
 The set of DoFs \eqref{D11}-\eqref{D14} (resp. \eqref{D15}) are unisolvent for the VE space $\bV_1^{h,k_1}(E)$ (resp. $Q_1^{h,k_1}(E)$) (see \cite[Proposition 3.1]{daveiga15-stokes}). Furthermore, if $N_e$ denotes the number of edges of $E$, it is not difficult to check that
$$\dim(\bV^{h,k_1}_1(E))=2N_e k_1 + \frac{(k_1-1)(k_1-2)}{2} + \frac{(k_1+1)k_1}{2}-1, \quad \dim(Q_1^{h,k_1}(E)) = \frac{(k_1+1)k_1}{2}.$$
\paragraph{Projection operators} For $E\in \mathcal{T}_h$, denote the energy projection operator by $\bPi_{1}^{\beps,k_1}: \bH^1(E)\rightarrow (\mathcal{P}_{k_1}(E))^2$, defined locally through 
 \[
        \int_E \beps(\bv_h-\bPi_{1}^{\beps,k_1}\bv_h) \colon \beps(\mathbf{m}_{k_1}) = 0 \quad \forall \mathbf{m}_{k_1} \in (\mathcal{M}_{k_1}(E))^2.
     \]
Note that the above definition of projection involving symmetric gradient leads to $\mathbf{0}$ for polynomials $\mathbf{p}\in \left\{ (1,0),(0,1),(-y,x) \right\}$. Therefore, we need to impose three  additional conditions to uniquely define $\bPi_{1}^{\beps,k_1}$.  For example, we can take the following conditions from \cite{yu2022mvem}
\[
    \sum_{i=1}^{N_v}(\bPi_{1}^{\beps,k_1}\bv_h(z_i),\mathbf{p}(z_i)) = \sum_{i=1}^{N_v}(\bv_h(z_i),\mathbf{p}(z_i)), \quad \forall\mathbf{p}\in \left\{ (1,0),(0,1),(-y,x) \right\}, 
\]
where $N_v$ is the total number of vertices $z_i$ of $E$. Moreover, the projection  $\bPi_{1}^{\beps,k_1}\bv_h$ is computable  for $\bv_h\in \bV_1^{h,k_1}(E)$ from \eqref{dofsV1h} (see \cite[Section 3.2]{daveiga15-stokes}).
We also introduce the $L^2$-projection $\bPi_{1}^{0,k_1}: \bL^2(E)\rightarrow (\mathcal{P}_{k_1}(E))^2$ defined locally, for a given $v_h\in \bL^2(E)$ and $E\in \mathcal{T}_h$, as 
    \begin{align*}
        \int_E (\bv_h-\bPi_{1}^{0,k_1}\bv_h)\cdot \mathbf{m}_{k_{1}} &= 0, \quad \forall \mathbf{m}_{k_{1}} \in (\mathcal{M}_{k_{1}}(E))^{2}.
    \end{align*}
Notice that $\bPi_{1}^{0,k_1-2}\bv_h$ can be computed for $\bv_h\in \bV_1^{h,k_1}(E)$ from \eqref{dofsV1h} (see \cite[Section 3.3]{daveiga15-stokes}). Finally, the discrete right-hand side  is defined as $\fb_h:=\bPi_1^{0,k_1-2}\fb$  for $\fb\in \bL^2(\Omega)$.

The following two lemmas are classical projection and interpolation estimates,  written here in the scaled norms (for sake of robustness with respect to model parameters). Let us first define the semi-norms induced by the spaces $\bV_1$ and $Q_{b_1}$ as
\begin{align*}
    |\bv|_{s_1+1,\bV_1}^2 = 2\mu|\bv|_{s_1+1,\Omega}^2 \quad \text{and} \quad |\tilde{q}|_{s_1+1,Q_{b_1}}^2 = \frac{1}{2\mu}|\tilde{q}|_{s_1+1,\Omega}^2 \quad \text{with} \quad 0\leq s_1\leq k_1.     
\end{align*}

\begin{lemma} \label{estimates-poly-proj-elast}
    For any $\bv \in (\bH^{s_1+1}(\Omega)\cap \bV_1,|\cdot|_{s_1+1,\bV_1})$ and $\tilde{q}\in (H^{s_1+1}(\Omega)\cap Q_{b_1},|\cdot|_{s_1+1,Q_{b_1}})$, the polynomial projections $\bPi_1^{\beps,k_1}\bv$ and $\Pi_1^{0,k_1}\tilde{q}$ satisfy the following estimates
    \begin{subequations} \begin{align}
    \norm{\bv-\bPi_1^{\beps,k_1}\bv}_{\bV_1}&\lesssim h^{s_1}|\bv|_{s_1+1,\bV_1},\label{elliptic-estimate-elast}\\     
    \norm{\tilde{q}-\Pi_1^{0,k_1}\tilde{q}}_{Q_{b_1}}&\lesssim h^{s_1+1}|\tilde{q}|_{s_1+1,Q_{b_1}}. \label{poly-estimate-scalar-elast}
    \end{align}\end{subequations}
\end{lemma}

\begin{proof}
    Note that locally the following estimates hold (see \cite{Brenner1994})
    \begin{subequations} \begin{align*}
    |\bv-\bPi_1^{\beps,k_1}\bv|_{1,E}&\lesssim h_E^{s_1}|\bv|_{s_1+1,E},\quad     
    \norm{\tilde{p}-\Pi_1^{0,k_1}\tilde{q}}_{0,E}\lesssim h_E^{s_1+1}|\tilde{q}|_{s_1+1,E}.    \end{align*}
    \end{subequations}
    Therefore, the equivalence between $|\cdot|_{1,\Omega}$ and $\norm{\cdot}_{\bV_1}$ already discussed in Lemma~\ref{lem:elast} and the additive property of the norm imply that
    \begin{align*}
        \norm{\bv-\bPi_1^{\beps,k_1}\bv}_{\bV_1} \lesssim \sum_{E\in \mathcal{T}_h} \sqrt{2\mu} |\bv-\bPi_1^{\beps,k_1}\bv|_{1,E} \lesssim \sum_{E\in \mathcal{T}_h} h_E^{s_1} \sqrt{2\mu}|\bv|_{s_1+1,E} \leq  h^{s_1} |\bv|_{s_1+1,\bV_1}.
    \end{align*}
    Similarly
    \begin{align*}
        \norm{\tilde{q}-\Pi_1^{0,k_1}\tilde{q}}_{Q_{b_1}} = \sum_{E\in \mathcal{T}_h} \frac{1}{\sqrt{2\mu}} \norm{\tilde{q}-\Pi_1^{0,k_1}\tilde{q}}_{0,E} \lesssim \sum_{E\in \mathcal{T}_h} h_E^{s_1+1} \frac{1}{\sqrt{2\mu}}|\tilde{q}|_{s_1+1,E} \leq h^{s_1} |\tilde{q}|_{s_1+1,Q_{b_1}}.
    \end{align*}
\end{proof}

\paragraph{Interpolation operator}  We define the Fortin  operator $\bPi_1^{F,k_1}:\mathbf{H}^{1+\delta}(E)\rightarrow \bV_1^{h,k_1}(E)$ for $\delta>0$ (see \cite{daveiga2022stability}) through the DoFs  \eqref{dofsV1h} as
\[
\mathrm{dof}_{j}(\bv-\bPi_1^{F,k_1}\bv)=0, \quad \text{for all}\; j=1,\dots,\dim(\bV_1^{h,k_1}(E))\;\text{and for any}\;\bv\in\bH^{1+\delta}(E). 
\]
Let $\Pi_1^{0,k_1-1}$ be the $L^2$-projection of $L^2(E)$ onto the space $\mathcal{P}_{k_1-1}(E)$. From \eqref{D13}, $\vdiv \bPi_1^{F,k_1} \bv \in\mathcal{P}_{k_1-1}(E)$ and the following commutative property holds  
\[
    \vdiv \bPi_1^{F,k_1}\bv = \Pi_1^{0,k_1-1} \vdiv \bv.\]
\begin{lemma}\label{fortin-elast}
Given $\bv\in (\bH^{s_1+1}(\Omega)\cap\bV_1), |\cdot|_{s_1+1,\bV_1})$, the Fortin interpolation operator $\bPi_1^{F,k_1}$ satisfies 
\begin{align}
    \norm{\bv-\bPi_1^{F,k_1}\bv}_{\bV_1}\lesssim h^{s_1}|\bv|_{s_1+1,\bV_1}.\label{interpolation-estimate-elast}
\end{align}
\end{lemma}
\begin{proof}
The proof of the standard interpolation estimate 
\begin{align*}
    |\bv-\bPi_1^{F,k_1}\bv|_{1,E}\lesssim h_E^{s_1}|\bv|_{s_1+1,E}
\end{align*}
on each element $E\in\mathcal{T}_h$  is provided in \cite[Theorem 2.4]{daveiga2022stability}. This with the obvious scaling shows that
\begin{align*}
        \norm{\bv-\bPi_1^{F,k_1}\bv}_{\bV_1} \lesssim \sum_{E\in \mathcal{T}_h} \sqrt{2\mu} |\bv-\bPi_1^{F,k_1}\bv|_{1,E} \lesssim \sum_{E\in \mathcal{T}_h} h_E^{s_1} \sqrt{2\mu}|\bv|_{s_1+1,E} \leq h^{s_1} |\bv|_{s_1+1,\bV_1}.
    \end{align*}
\end{proof}
\paragraph{Discrete forms} For $E\in\mathcal{T}_h$, let $a_{1}^{E}(\cdot,\cdot)$ be the restriction of $a_1(\cdot,\cdot)$ to $E$. Then its discrete counterpart is defined, for all $\bu_h,\bv_h\in \bV_1^{h,k_1}(E)$, by 
\begin{align*}
    a_{1}^{h,E}(\bu_h,\bv_h) := a_1^E(\bPi_1^{\beps,k_1}\bu_h,\bPi_1^{\beps,k_1}\bv_h) + S_1^E(\bu_h-\bPi_1^{\beps,k_1}\bu_h,\bv_h-\bPi_1^{\beps,k_1}\bv_h),
\end{align*}
where $S_1^E(\cdot,\cdot)$ is any symmetric and positive definite bilinear form defined on $\bV_1^{h,k_1}(E)\times \bV_1^{h,k_1}(E)$ that scales like $a_{1}^{E}(\cdot,\cdot)$. In particular, there holds 
\begin{align}
    a_{1}^{E}(\bv_h,\bv_h)\lesssim S_{1}^{E}(\bv_h,\bv_h)\lesssim a_{1}^{E}(\bv_h,\bv_h), \quad \forall \bv_h\in \ker(\bPi_1^{\beps,k_1}). \label{stabilisation-elast}
\end{align}
Set the global discrete bilinear form as $a_{1}^{h}(\bu_h,\bv_h)= \sum_{E\in \mathcal{T}_h} a_{1}^{h,E}(\bu_h,\bv_h)$ for all $\bu_h,\bv_h\in \bV_1^{h,k_1}$.
For the linear functional $F_1$, the corresponding global discrete form $F^h_1(\cdot)$ is defined as
$$F_1^{h}(\bv_h) =  \int_\Omega\fb\cdot\bPi_1^{0,k_1-2}\bv_h\quad\text{for}\;\bv_h\in \bV_1^{h,k_1}.$$
\cblue{The} computability of $\bPi_1^{\beps,k_1}$ and $\bPi_1^{0,k_1-2}$ implies that of  $a_{1}^{h,E}(\bu_h,\bv_h)$ and $F_1^{h,E}(\bv_h)$. On the other hand, the definition of $\bV_1^{h,k_1}$ and $Q_1^{h,k_1}$ allow us to 
\cgreen{state that the bilinear forms $b_1(\cdot,\cdot)$, $c_1(\cdot,\cdot)$ and the functional $G_1^{\vartheta_h}(\cdot)$ (with $\vartheta_h$ belonging to the space $Q_2^{h,k_2}$, to be defined explicitly later in \eqref{Q2h}) are  computable:
\begin{gather*}
b_1(\bv_h,\tilde{q}_h) = \sum_{E\in \mathcal{T}_h} \int_E \tilde{q_h} \vdiv \bv_h 
, \quad c_1(\tilde{p}_h,\tilde{q}_h) = \sum_{E\in \mathcal{T}_h} \int_E \tilde{p}_h \tilde{q}_h, 
\quad \text{and} \quad 
    G_1^{\vartheta_h}(\tilde{q}_h) = \sum_{E\in \mathcal{T}_h} \int_E \ell(\vartheta_h)\tilde{q}_h,
\end{gather*}
where} a proper order in the quadrature rule is used to handle the term $\ell(\vartheta_h)$ in $G_1^{\vartheta_h}(\tilde{p}_h)$. 
\paragraph{Discrete formulation} For a given $\vartheta_h \in Q_2^{h,k_2}$, find $(\bu_h,\tilde{p}_h)\in \bV_1^{h,k_1}\times Q_1^{h,k_1}$ such that 
\begin{subequations}\label{eq:weak-elast-discrete}
    \begin{align}
  a_1^{h}(\bu_h,\bv_h) + b_1(\bv_h,\tilde{p}_h)  &= \cblue{F_1^h(\bv_h)} \qquad \forall \bv_h \in \bV_1^{h,k_1},\\
  b_1(\bu_h,\tilde{q}_h) - \frac{1}{\lambda} c_1(\tilde{p}_h,\tilde{q}_h)  &= G_1^{\vartheta_h}(\tilde{q}_h) \qquad \forall \tilde{p}_h \in Q_1^{h,k_1}.
    \end{align}
\end{subequations} 
\subsection{Virtual element approximation for the reaction-diffusion problem}\label{VEM-transport}
In order to state the discrete counterpart of \eqref{eq:weak-transport} using VE spaces, we adapt the approach from \cite{daveiga15-transport} to our formulation with the  stress-assisted diffusion $\bbM^{-1}(\beps(\bw),\tilde{r})$ \cblue{as} part of the bilinear form $a_2^{\bw,\tilde{r}}(\cdot,\cdot)$. 

\paragraph{VE spaces and DoFs}
For $k_2\geq 0$, the discrete VE space locally solves a div-rot problem, that is, 
\begin{align*}
\bV_2^{h,k_2}(E) = \{ \bxi\in \bH(\vdiv,E)\cap \bH(\vrot,E) \colon &\bxi\cdot \bn|_e\in \mathcal{P}_{k_2}(e) \text{ for all } e\in \partial E,\\
&\vdiv\bxi \in \mathcal{P}_{k_2}(E),\; \vrot\bxi \in \mathcal{P}_{k_2-1}(E)\}.
\end{align*}
In turn, the global discrete spaces are defined as follows
\begin{subequations}
    \begin{align}
        \bV_2^{h,k_2} &= \{ \bzeta \in \bV_2 \colon \bzeta|_E \in \bV_2^{h,k_2}(E), \quad \forall E\in \mathcal{T}_h \},\\
        Q_2^{h,k_2} &= \{ \varphi\in Q_2 \colon \varphi|_E\in \mathcal{P}_{k_2}(E), \quad \forall E\in \mathcal{T}_h \} \label{Q2h}.
    \end{align}
\end{subequations}
The DoFs for $\bxi_h\in \bV_2^{h,k_2}(E)$ can be taken as
\begin{subequations}\label{dofsV2h}
\begin{align}
&\bullet \text{Values of } \bxi_h\cdot \bn \text{ on the $k_2+1$ Gauss--Lobatto quadrature points of each edge of $E$}, \label{D21}\\
&\bullet  \int_E \bxi_h\cdot \mathbf{m}_{k_2-1}^\nabla, \quad \forall \mathbf{m}_{k_2-1}^{\nabla} \in \mathcal{M}_{k_2-1}^{\nabla}(E),\label{D22}\\
&\bullet \int_E \bxi_h \cdot \mathbf{m}_{k_2}^{\oplus}, \quad \forall \mathbf{m}_{k_2}^{\oplus} \in \mathcal{M}_{k_2}^{\oplus}(E),\label{D23}
\end{align}
\end{subequations}
whereas the DoFs for $\psi_h\in Q_2^{h,k_2}(E)$ are 
\begin{align}
&\bullet \int_E \psi_h m_{k_2}, \quad \forall m_{k_2}\in \mathcal{M}_{k_2}(E). \label{D24}
\end{align}
Note that $(\mathcal{P}_{k_2}(E))^2\subseteq \bV_2^{h,k_2}(E)$ and $\mathcal{P}_{k_2}(E) = Q_2^{h,k_2}(E)$. The set of DoFs \eqref{D21}-\eqref{D23} (resp. \eqref{D24}) are unisolvent for the local virtual space $\bV_2^{h,k_2}(E)$ (resp. $Q_2^{h,k_2}(E)$), see \cite[Proposition 3.5]{brezzi_2014} for a proof (although the DoFs are slightly different, the proof follows similarly). Moreover, it is readily seen that $$\dim(\bV^{h,k_2}_2(E))=N_e (k_2+1) + \frac{k_2(k_2-1)}{2}-1 + \frac{(k_2+2)(k_2+1)}{2}, \quad \dim(Q_2^{h,k_2}(E)) = \frac{(k_2+1)k_2}{2}.$$

\paragraph{Projection operators} Let us define the local $L^2$-projection  $\bPi_{2}^{0,k_2}: \bL^2(E)\rightarrow (\mathcal{P}_{k_2}(E))^2$ by 
\[\int_E (\bxi_h-\bPi_{2}^{0,k_2}\bxi_h)\cdot \mathbf{m}_{k_2} = 0, \quad \forall \mathbf{m}_{k_2} \in (\mathcal{M}_{k_2}(E))^2,\quad \forall E\in \mathcal{T}_h.\]
Note that the projection $\bPi_2^{0,k_2}\bxi_h$ is computable for $\bxi_h\in\bV_2^{h,k_2}(E)$ (see \cite[Section 3.2]{daveiga15-transport}).

For sake of the forthcoming analysis, let us define semi-norms induced by the spaces $\bV_2$ and $Q_{b_2}$ as
\begin{align*}
    |\bxi|_{s_2+1,\bV_2}^2 = M|\bxi|_{s_2+1,\Omega}^2 \quad \text{and} \quad |\psi|_{s_2+1,Q_{b_2}}^2 = \frac{1}{M}|\psi|_{s_2+1,\Omega}^2, \quad \text{with} \quad 0\leq s_2\leq k_2.     
\end{align*}
\begin{lemma} \label{estimates-poly-proj-transport}
    For any $\bxi \in (\bH^{s_2+1}(\Omega)\cap \bV_2,|\cdot|_{s_2+1,\bV_2})$ and $\psi \in (H^{s_2+1}(\Omega)\cap Q_{b_2},|\cdot|_{s_2+2,Q_{b_2}})$, the polynomial projections $\bPi_2^{0,k_2}\bxi$ and \cblue{$\Pi_2^{0,k_2}\psi$} satisfy the following estimates
    \begin{align*}
    \norm{\bxi-\bPi_2^{0,k_2}\bxi}_{\bbM}\lesssim h^{s_2+1}|\bxi|_{s_2+1,\bV_2},
   \quad \norm{\psi-\Pi_2^{0,k_2}\psi}_{Q_{b_2}}\lesssim h^{s_2+1}|\psi|_{s_2+1,Q_{b_2}}. 
    \end{align*}
\end{lemma}

\begin{proof}
Classical results (see \cite{Brenner1994}) lead to the bounds 
    \[
    \norm{\bxi-\bPi_2^{0,k_2}\bxi}_{0,E}\lesssim h_E^{s_2+1}|\bxi|_{s_2+1,E},\qquad 
    \norm{\psi-\Pi_2^{0,k_2}\psi}_{0,E}\lesssim h_E^{s_2+1}|\psi|_{s_2+1,E}. 
    \]
Since $\norm{\cdot}_{\bbM}\leq \sqrt{M}\norm{\cdot}_{0,\Omega}$, we can easily see that
    \begin{align*}
        \norm{\bxi-\bPi_2^{0,k_2}\bxi}_{\bbM} & \leq \sum_{E\in \mathcal{T}_h} \sqrt{M}\norm{\bxi-\bPi_2^{0,k_2}\bxi}_{0,E}\lesssim \sum_{E\in \mathcal{T}_h} h_E^{s_2+1} \sqrt{M}|\bxi|_{s_2+1,E} \leq h^{s_2+1}|\bxi|_{s_2+1,\bV_2},\\
       \norm{\psi-\Pi_2^{0,k_2}\psi}_{Q_{b_2}} & = \sum_{E\in \mathcal{T}_h} \frac{1}{\sqrt{M}}\norm{\psi-\Pi_2^{0,k_2}\psi}_{0,E}\lesssim \sum_{E\in \mathcal{T}_h} h_E^{s_2+1} \frac{1}{\sqrt{M}} |\psi|_{s_2+1,E} \leq h^{s_2+1}|\psi|_{s_2+1,Q_{b_2}}.
    \end{align*}
\end{proof}

\paragraph{Interpolation operator} 
The Fortin operator $\bPi_2^{F,k_2}:\mathbf{H}^{1}(E)\rightarrow \bV_2^{h,k_2}(E)$ can be defined directly from \eqref{dofsV2h} as
\begin{align*}
\mathrm{dof}_{j}(\bxi-\bPi_2^{F,k_2}\bxi)=0, \quad \text{for all}\;j=1,\dots, \dim(\bV_2^{h,k_2}(E))\;\text{and for}\;\bxi\in\mathbf{H}^{1}(E).
\end{align*}
Moreover, \eqref{D22} and an integration by parts imply the following commutative property:
\begin{equation*}
    \vdiv \bPi_2^{F,k_2}\bxi =  \Pi_2^{0,k_2} \vdiv \bxi.
\end{equation*}

\begin{lemma}\label{fortin-transport}
The Fortin interpolation operator $\bPi_2^{F,k_2}$ satisfies the following estimate
\begin{align}
    \norm{\bxi - \bPi_2^{F,k_2}\bxi}_{\bbM} \lesssim h^{s_2+1}|\bxi|_{s_2+1,\bV_2} \qquad \forall \bxi \in (\bH^{s_2+1}(\Omega)\cap \bV_2,|\cdot|_{s_2+1,\cblue{\bV_2}}).\label{fortin-transport-estimate}
\end{align}
\end{lemma}
\begin{proof}
The standard interpolation estimate from  \cite{daveiga15-transport} shows \cblue{$\|\bxi - \bPi_2^{F,k_2}\bxi\|_{0,E}  \lesssim h^{s_2+1}|\bxi|_{s_2+1,E}$.
Hence}
\begin{align*}
    \norm{\bxi - \bPi_2^{F,k_2}\bxi}_{\bbM}\leq \sum_{E\in \mathcal{T}_h} \sqrt{M}\|\bxi - \bPi_2^{F,k_2}\bxi\|_{0,E}  \lesssim \sum_{E\in \mathcal{T}_h}h^{s_2+1} \sqrt{M} |\bxi|_{s_2+1,E} \leq h^{s_2+1}|\bxi|_{s_2+1,\bV_2}.
\end{align*}
\end{proof}
\paragraph{Discrete forms} 
  Given $\cred{\overline{\bw}_h:=\bPi_1^{\beps,k_1} \bw_h} \in \bV_1^{h,k_1}$ and $\tilde{r}_h\in Q_1^{h,k_1}$, let $a_{2}^{\cred{\overline{\bw}_h}, \tilde{r}_h,E}(\cdot,\cdot)$  be the restriction of $a_{2}^{\overline{\bw}_h,\tilde{r}_h}(\cdot,\cdot)$ to $E$ for $E\in\mathcal{T}_h$. Then 
\begin{align*}
a_{2}^{\cred{\overline{\bw}_h},\tilde{r}_h,h,E}(\bzeta_h,\bxi_h) :=  a_{2}^{\cred{\overline{\bw}_h},\tilde{r}_h,E}(\bPi_2^{0,k_2}\bzeta_h, \bPi_2^{0,k_2}\bxi_h) \, +  S_2^{\cred{\overline{\bw}_h,\tilde{r}_h},E}(\bzeta_h-\bPi_2^{0,k_2}\bzeta_h,\bxi_h-\bPi_2^{0,k_2}\bxi_h),
\end{align*}
where $S_2^{\cred{\overline{\bw}_h,\tilde{r}_h},E}(\cdot,\cdot)$ is any symmetric and positive definite bilinear form defined on $\bV_2^{h,k_2}(E)\times \bV_2^{h,k_2}(E)$ that scales like  $a_2^{\cred{\overline{\bw}_h},\tilde{r}_h,E}(\cdot,\cdot)$. In particular,
\begin{align}\label{stabilisation-transport}
    a_{2}^{\cred{\overline{\bw}_h},\tilde{r}_h,E}(\bxi_h,\bxi_h)\lesssim S_{2}^{\cred{\overline{\bw}_h,\tilde{r}_h},E}(\bxi_h,\bxi_h)\lesssim a_{2}^{\cred{\overline{\bw}_h},\tilde{r}_h,E}(\bxi_h,\bxi_h), \quad \forall \bxi_h\in \ker(\bPi_2^{0,k_2}). 
\end{align}
The global discrete bilinear form is defined as $$a_2^{\cred{\overline{\bw}_h},\tilde{r}_h,h}(\bzeta_h,\bxi_h) = \sum_{E\in \mathcal{T}_h}a_2^{\cred{\overline{\bw}_h},\tilde{r}_h,h,E}(\bzeta_h,\bxi_h)\quad\forall \bzeta_h,\bxi_h\in\bV_2^{h,k_2}.$$

Finally, the definition of $\bV_2^{h,k_2}$ and $Q_2^{h,k_2}$ allow us to extend the bilinear forms $b_2(\cdot,\cdot)$, $c_2(\cdot,\cdot)$ and the linear forms \cblue{$F_2(\cdot)$}, $G_2(\cdot)$ to \cgreen{make them computable as follows
\begin{gather*}
    b_2(\bxi_h,\varphi_h) = \sum_{E\in \mathcal{T}_h} \int_E \varphi_h \vdiv \bxi_h , \quad c_2(\varphi_h,\psi_h) = \sum_{E\in \mathcal{T}_h} \int_E \varphi_h \psi_h, \\ 
    F_2(\bxi_h) = \sum_{E\in \mathcal{T}_h} \langle \varphi_D, \bxi_h \cdot \bn \rangle_{\partial E\cap \Gamma_D}, \quad G_2(\psi_h) = \sum_{E\in \mathcal{T}_h} \int_E g\psi_h ,
\end{gather*}
where} a proper order in the quadrature rule is used for  $\varphi_D$ and $g$ in $F_2(\bxi)$ and $G_2(\psi_h)$, respectively.
\paragraph{Discrete formulation} Given $\cred{\overline{\bw}_h}\in \bV_1^{h,k_1}$ and $\tilde{r}_h\in Q_1^{h,k_1}$, find $(\bzeta_h,\varphi_h)\in \bV_2^{h,k_2}\times Q_2^{h,k_2}$ such that
\begin{subequations}\label{eq:weak-transport-discrete}
    \begin{align}
  a_2^{\cred{\overline{\bw}_h},\tilde{r}_h,h}(\bzeta_h,\bxi_h) + b_2(\bxi_h,\varphi_h)  &= F_2(\bxi_h) \qquad \forall \bxi_h \in \bV_2^{h,k_2},\\
  b_2(\bzeta_h,\psi_h) - \theta c_2(\varphi_h,\psi_h)  &= G_2(\psi_h) \qquad \forall \psi_h \in Q_2^{h,k_2}.
    \end{align}
\end{subequations}

\cred{\begin{remark}
We point out that this variant of the mixed VEM for the diffusion problem does not satisfy a discrete maximum principle. A possibility to remediate this is to use the recent edge-averaged VEM from \cite{cao2024edgeaveragedvirtualelementmethods} or finite volume - like VEM from \cite{sheng2023virtual}, however their adaptation to the nonlinear diffusion case might not be straightforward and we therefore use a similar formulation as in \cite{daveiga15-transport}.
\end{remark}}

\subsection{Well-posedness analysis of the discrete problem}\label{sec:wellp-h}
Let us recall an abstract result from \cite[Lemma 5.1]{boon21} which adapts  Theorem~\ref{th:unique-solvability} to the discrete setting.
\begin{theorem} \label{th:unique-solvability-discrete}
Let $V^h\subset V,Q_b^h\subset Q$ be Hilbert spaces endowed with the 
norms $\norm{\cdot}_{V^h}$ and $\norm{\cdot}_{Q_b^h}$, let $Q^h$ be a dense (with respect to the norm $\norm{\cdot}_{Q_b^h}$) linear subspace of $Q_b^h$ and three bilinear forms $a^h(\cdot,\cdot)$ on $V^h\times V^h$ (continuous, symmetric and positive semi-definite), $b^h(\cdot,\cdot)$ on $V^h\times Q_b^h$ (continuous), and $c^h(\cdot,\cdot)$ on $Q^h\times Q^h$ (symmetric and positive semi-definite); which define three linear operators $A^h: V^h \rightarrow (V^h)'$, $B^h : V^h \rightarrow (Q^h_b)'$ and $C^h:Q^h \rightarrow (Q^h)'$, respectively. Suppose further that 
\begin{subequations}
\begin{align}\label{brezzi-condition-1-discrete}
\norm{\hat{v_h}}_{V^h}^2 \lesssim a^h(\hat{v}_h,\hat{v}_h) \qquad \forall \hat{v}_h \in \mathrm{Ker}(B^h),\\
\label{brezzi-condition-2-discrete}
\norm{q_h}_{Q_b^h} \lesssim \sup_{v_h\in V^h} \frac{b^h(v_h,q_h)}{{\norm{v_h}}_{V^h}} \qquad \forall q_h \in Q_b^h.
\end{align}
\end{subequations}
Assume that $Q^h$ is complete with respect to the norm $\norm{\cdot}_{Q^h}^2:=\norm{\cdot}_{Q_b^h}^2+t^2|\cdot|_{c^h}^2$, where $|\cdot|_{c^h}^2 := c^h(\cdot,\cdot)$ is a semi-norm in $Q^h$. Let $t\in[0,1]$ and set the parameter-dependent energy norm as
$$\norm{(v_h,q_h)}_{V^h\times Q^h}^2:=\norm{v_h}_{V^h}^2+\norm{q_h}_{Q^h}^2=\norm{v_h}_{V^h}^2+\norm{q_h}_{Q_b^h}^2+t^2|q|_{c^h}^2.$$
Assume also that the following inf-sup condition holds
\begin{align}\label{braess-condition-discrete}
\norm{u_h}_{V^h} \lesssim \sup_{(v_h,q_h)\in V^h\times Q^h} \cblue{\frac{a^h(u_h,v_h)+b^h(u_h,q_h)}{\norm{(v_h,q_h)}_{V^h\times Q^h}}} \qquad \forall u_h \in V^h.
\end{align}
Then, for every $F^h\in (V^h)'$ and $G^h\in (Q^h)'$, there exists a unique $(u_h,p_h)\in V^h\times Q^h$ satisfying
\begin{subequations}
\begin{align*}
a^h(u_h,v_h) + b^h(v_h,p_h) = F^h(v_h) &\qquad \forall v_h\in V^h,\\
b^h(u_h,q_h) - t^2c^h(p_h,q_h) = G^h(q_h) &\qquad \forall q_h\in Q^h.
\end{align*}
\end{subequations}
Furthermore, the following continuous dependence on data holds 
\begin{align} \label{dependence_data-discrete}
    \norm{(u_h,p_h)}_{V^h\times Q^h} &\lesssim \norm{F^h}_{(V^h)'}+\norm{G^h
    }_{(Q^h)'}.
\end{align}
\end{theorem}
The well-posedness of the uncoupled discrete problems \eqref{eq:weak-elast-discrete} and \eqref{eq:weak-transport-discrete} is given next.
\begin{lemma}\label{lem:elast-disc}
Given $\vartheta_h \in Q_2^{h,k_2}$, assume that $1\leq \lambda$ and $0<\mu $. Then, there exists a unique pair $(\bu_h,\tilde{p}_h)\in \bV_1^{h,k_1}\times Q_1^{h,k_1}$ solution to \eqref{eq:weak-elast-discrete}. \cgreen{Furthermore, there holds 
\begin{align}\label{dependence-data-elast-discrete}
\norm{(\bu_h,\tilde{p}_h)}_{\bV_1\times Q_1} \leq \overline{C}_1  \left(\norm{F_1^h}_{\bV_1'} + \norm{G_1^{\vartheta_h}}_{Q_1'}\right),
\end{align}
where the constant $\overline{C}_1>0$ does not depend on $h$ and the physical parameters.} 
\end{lemma}
\begin{proof}
The same arguments in Lemma~\ref{lem:elast} can be used to carry over the properties of $c_1(\cdot,\cdot)$ and $G_1^{\vartheta_h}(\cdot)$ to the discrete formulation. On the other hand, the positive semi-definiteness of the stabilisation term $S^E_1(\cdot,\cdot)$ allows us to extend the properties of $a_1(\cdot,\cdot)$ to the discrete operator $a_1^{h}(\cdot,\cdot)$. 
Indeed, concerning boundedness, the second inequality in \eqref{stabilisation-elast} leads for all $\bu_h \in \bV_1^{h,k_1}(E)$ to 
\begin{align*}
    a_1^{h,E}(\bu_h,\bu_h) \lesssim a_1^{E}(\bPi_1^{\beps,k_1}\bu_h,\bPi_1^{\beps,k_1}\bu_h) +  a_1^{E}(\bu_h - \bPi_1^{\beps,k_1}\bu_h,\bu_h - \bPi_1^{\beps,k_1}\bu_h)= \norm{\bu_h}_{\bV_1(E)}^2.
\end{align*}
This and the Cauchy--Schwarz inequality for the inner product show that 
\begin{align*}
    a_1^{h,E}(\bu_h,\bv_h) \leq \sqrt{a_1^{h,E}(\bu_h,\bu_h)}\sqrt{a_1^{h,E}(\bv_h,\bv_h)} \lesssim  \norm{\bu_h}_{\bV_1(E)} \norm{\bv_h}_{\bV_1(E)}\qquad \forall \bu_h,\bv_h\in \bV_1^{h,k_1}.
\end{align*}
\cred{Regarding} the first Brezzi condition \eqref{brezzi-condition-1-discrete}, the first inequality in \eqref{stabilisation-elast} gives 
\begin{align*}
    \norm{\bv_h}_{\bV_1(E)}^2  =  a_1^{E}(\bPi_1^{\beps,k_1}\bv_h,\bPi_1^{\beps,k_1}\bv_h) + a_1^{E}(\bv_h - \bPi_1^{\beps,k_1}\bv_h,\bv_h - \bPi_1^{\beps,k_1}\bv_h)  \lesssim a_1^{h,E}(\bv_h,\bv_h) \quad \forall \bv_h\in\bV_1^{h,k_1}(E).
\end{align*}
Next, thanks to \cite[Proposition 4.2]{daveiga15-stokes} we can state \cred{a non-weighted  discrete inf-sup condition},
\[\norm{\tilde{q_h}}_{0,\Omega}\lesssim \sup_{\bv_h\in \bV_1^{h,k_1}} \frac{b_1(\bv_h,\tilde{q}_h)}{\norm{\nabla \bv_h}_{0,\Omega}}\qquad \forall \tilde{q}_h \in Q_1^{h,k_1},\]
\cred{and with} an analogous argument as in the second part of the proof of Lemma~\ref{lem:elast}, we can show that the second discrete Brezzi condition \eqref{brezzi-condition-2-discrete} holds:
$$\norm{\tilde{q}_h}_{Q_{b_1}} \lesssim \sup_{\bv_h\in \bV_1^{h,k_1}} \frac{b_1(\bv_h,\tilde{q}_h)}{\norm{\bv_h}_{\bV_1}} \qquad \forall \tilde{q}_h \in Q_{1}^{h,k_1}.$$
%
The continuity of $F^h(\cdot)$ can be obtained directly from the Cauchy--Schwarz inequality and the boundedness of \cblue{$\bPi^{0,k_1-2}(\cdot)$}. 
In fact,
\begin{align*}
|F^{h,E}_1(\bv_h)| &\leq \norm{\mathbf{f}}_{0,E}\norm{\bPi^{0,k_1-2}\bv_h}_{0,E} \leq \norm{\mathbf{f}}_{0,E}\norm{\bv_h}_{0,E} \lesssim \norm{\mathbf{f}}_{Q_{b_1}(E)}\norm{\bv_h}_{\bV_1(E)}, \quad \forall \bv_h\in \bV_1^{h,k_1}.
\end{align*}
Finally, 
the discrete Braess condition \eqref{braess-condition-discrete} is obtained as follows: 
for $\bu_h \in \bV_1^{h,k_1}$, we set $\bv_h = \bu_h$ and $\tilde{q}_h=0$. Then,  \eqref{brezzi-condition-1-discrete} for $a_1^h(\cdot,\cdot)$ 
leads to
\begin{align*}
    \norm{\bu_h}_{\bV_1}^2  \lesssim a_1^h(\bu_h,\bu_h) = a_1^{h}(\bu_h,\bu_h)+b_1(\bu_h,0) \text{ and }  \norm{(\bu_h,0)}^2_{\bV_1\times Q_1} = \norm{\bu_h}_{\bV_1}^2.
\end{align*}
This proves that
$$\norm{\bu_h}_{\bV_1} \lesssim \sup_{(\bv_h,\tilde{q}_h)\in \bV_1^{h,k_1}\times Q_1^{h,k_1}} \cblue{\frac{a_1^h(\bu_h,\bv_h)+b_1
(\bu_h,\tilde{q}_h)}{\norm{(\bv_h,\tilde{q}_h)}_{\bV_1\times Q_1}}} \quad \forall \bu_h \in \bV_1^{h,k_1}.$$
Therefore, the conditions of Theorem~\ref{th:unique-solvability-discrete} are verified and we also obtain 
\cgreen{the a priori bound \eqref{dependence-data-elast-discrete}.} \cred{Further details are provided in \cite{krr_preprint}.}
\end{proof}
\begin{lemma}\label{lem:trans-disc}
For a fixed pair $(\cred{\overline{\bw}_h},\tilde{r}_h)\in \bV_1^{h,k_1}\times Q_1^{h,k_1}$, assume that \cgreen{$\theta \leq \frac{1}{M}$}. Then, there exists a unique pair $(\bxi_h,\varphi_h)\in \bV_2^{h,k_2}\times Q_2^{h,k_2}$ solution to \eqref{eq:weak-transport-discrete}. \cgreen{Furthermore, we have 
\begin{align}\label{dependence-data-transport-discrete}
\norm{(\bzeta_h,\varphi_h)}_{\bV_2\times Q_1} \leq \overline{C}_2  \left(\norm{F_2}_{\bV_2} + \norm{G_2}_{Q_1}\right),
\end{align}
where the constant $\overline{C}_2>0$ does not depend on $h$ and the physical parameters.}
\end{lemma}
\begin{proof}
Since $c_2(\cdot,\cdot)$, $F_2(\cdot)$ and $G_2(\cdot)$ remain same in the discrete formulation, the analogous arguments as in Lemma~\ref{lem:trans} will hold. The second stability inequality in \eqref{stabilisation-transport} \cred{implies that}
for all $\bzeta_h \in \bV_2^{h,k_2}(E)$ to 
\begin{align*}
    a_2^{\cred{\bPi_1^{\beps,k_1} \bw_h},\tilde{r}_h,h,E}(\bzeta_h,\bzeta_h) &\lesssim a_2^{\cred{\bPi_1^{\beps,k_1} \bw_h},\tilde{r}_h,E}(\bPi_2^{0,k_2}\bzeta_h,\bPi_2^{0,k_2}\bzeta_h) + a_2^{\cred{\bPi_1^{\beps,k_1} \bw_h},\tilde{r}_h,E}(\bzeta_h - \bPi_2^{0,k_2}\bzeta_h,\bzeta_h - \bPi_2^{0,k_2}\bzeta_h).
\end{align*}
For all $\bzeta_h,\bxi_h\in \bV_2^{h,k_2}$, this results in
\begin{align*}
    a_2^{\cred{\overline{\bw}_h},\tilde{r}_h,h,E}(\bzeta_h,\bxi_h) \leq \sqrt{a_2^{\cred{\overline{\bw}_h},\tilde{r}_h,h,E}(\bzeta_h,\bzeta_h)}\sqrt{a_2^{\cred{\overline{\bw}_h},\tilde{r}_h,h,E}(\bxi_h,\bxi_h)}
    \lesssim \norm{\bzeta_h}_{\bV_2(E)} \norm{\bxi_h}_{\bV_2(E)}.
\end{align*}
Hence $a_2^{\cred{\overline{\bw}_h},\tilde{r}_h,h}(\cdot,\cdot)$ is a bounded operator. For condition \eqref{brezzi-condition-1-discrete}, note that $\mathrm{Ker}(B^h_2(E)) =\{\bxi_h \in \bV_2^{h,k_2}(E): \int_E \varphi_h \vdiv \bxi_h = 0\}\subseteq \mathrm{Ker}(B_2)$ which together with the first inequality in \eqref{stabilisation-transport} imply that
\begin{align*}
    \norm{\hat{\bxi}_h}_{\bV_2(E)}^2 
    & \cred{\lesssim a_2^{\cred{\overline{\bw}_h},\tilde{r}_h,h,E}(\hat{\bxi}_h,\hat{\bxi}_h) \qquad \forall \hat{\bxi}_h\in \mathrm{Ker}({B_2^h(E)})}.
\end{align*}
The Fortin operator defined in Section~\ref{VEM-transport} leads to a 
discrete inf-sup condition in the non-weighted norms.  
Proceeding as in Lemma~\ref{lem:trans}, \cred{we can readily show 
\eqref{brezzi-condition-2-discrete}.
Finally,}  \eqref{braess-condition-discrete} is deduced as follows: for $\bzeta_h \in \bV_2^{h,k_2}$, take $\bxi_h = \bzeta_h$ and $\varphi_h=M\vdiv\bzeta_h$ to obtain
\begin{align*}
    \norm{\bzeta_h}_{\bV_2}^2 = a_2^{\cred{\overline{\bw}_h},\tilde{r}_h}(\bzeta_h,\bzeta_h) + M b_2(\bzeta_h,\vdiv\bzeta_h) &\lesssim a_2^{\bw_h,\tilde{r}_h,h}(\bzeta_h,\bzeta_h)+b_2(\bzeta_h,M\vdiv\bzeta_h),
\end{align*}
and
\begin{align*}
   \norm{(\bzeta_h,M\vdiv\bzeta_h)}^2_{\bV_2\times Q_2} 
    &\cred{\lesssim \norm{\bzeta_h}_{\bV_2}^2 + \frac{1}{M}\norm{M\vdiv\bzeta_h}_{0,\Omega}^2 \lesssim \norm{\bzeta_h}_{\bV_2}^2.}
\end{align*}
Therefore,
$$\norm{\bzeta_h}_{\bV_2} \lesssim \sup_{(\bxi_h,\varphi_h)\in \bV_2^{h,k_2}\times Q_2^{h,k_2}} \cblue{\frac{a_2^{\cred{\overline{\bw}_h},\tilde{r}_h,h}(\bzeta_h,\bxi_h)+b_2(\bzeta_h,\varphi_h)}{\norm{(\bxi_h,\varphi_h)}_{\bV_2\times Q_2}}} \quad \forall \bzeta_h \in \bV_2^{h,k_2}.$$
This proves that the conditions of Theorem~\ref{th:unique-solvability-discrete} hold. \cgreen{Finally, we also get that \eqref{dependence-data-transport-discrete} is satisfied.}
\end{proof}

We finish this subsection by introducing the full coupled discrete formulation. For given $\fb\in \bL^2(\Omega)$, $g\in L^2(\Omega)$, and $\varphi_D\in H^{\frac{1}{2}}(\Gamma_D)$, find 
$(\bu_h,\tilde{p}_h,\bzeta_h,\varphi_h) \in \bV_1^{h,k_1}\times Q_1^{h,k_1} \times \bV_2^{h,k_2} \times Q_2^{h,k_2}$ such that 
\begin{subequations}\label{eq:weak-discrete}
\begin{align}
  a_1^{h}(\bu_h,\bv_h) + b_1(\bv_h,\tilde{p}_h)  &= F_2^h(\bv_h) \qquad \forall \bv_h \in \bV_1^{h,k_1},\\
  b_1(\bu_h,\tilde{q}_h) - \frac{1}{\lambda} c_1(\tilde{p}_h,\tilde{q}_h) \cgreen{- G_1^{\varphi_h}(\tilde{q}_h) }  &= \cgreen{0} \qquad \forall \tilde{p}_h \in Q_1^{h,k_1},\\
a_2^{\cred{\overline{\bu}_h},\tilde{p}_h,h}(\bzeta_h,\bxi_h) + b_2(\bxi_h,\varphi_h)  &= F_2(\bxi_h) \qquad \forall \bxi_h \in \bV_2^{h,k_2},\\
  b_2(\bzeta_h,\psi_h) - \theta c_2(\varphi_h,\psi_h)  &= G_2(\psi_h) \qquad \forall \psi_h \in Q_2^{h,k_2}.
\end{align}
\end{subequations}
\subsection{Discrete fixed-point strategy}\label{sec:discrete-fix-point} 
Following the approach in the continuous case, we define the discrete solution operators as follows 
\[\cS_1^h: Q_2^{h,k_2} \to \bV_1^{h,k_1}\times Q_1^{h,k_1}, \quad 
\vartheta_h \mapsto \cS_1^{h}(\vartheta_h)= (\cS_{11}^{h}(\vartheta_h),\cS_{12}^{h}(\vartheta_h)) := (\bu_h,\tilde{p}_h),
\]
where $(\bu_h,\tilde{p}_h)$ is the unique solution to \eqref{eq:weak-elast-discrete}; and 
\[\cS_2^{h}: \bV_1^{h,k_1}\times Q_1^{h,k_1} \to \bV_2^{h,k_2}\times Q_2^{h,k_2}, \quad 
(\bw_h,\tilde{r}_h) \mapsto \cS_2^h(\bw_h,\tilde{r}_h) = (\cS_{21}^h(\bw_h,\tilde{r}_h),\cS_{22}^h(\bw_h,\tilde{r}_h)) := (\bzeta_h,\varphi_h),
\]
where $(\bzeta_h,\varphi_h)$ is the unique solution to \eqref{eq:weak-transport-discrete}. The discrete version of  \eqref{eq:weak} is  thus equivalent to the following discrete fixed-point equation:
\[\text{Find $\varphi_h\in Q_2^{h,k_2}$ such that $\cA^h(\varphi_h) = \varphi_h$},\]
where $\cA^h: Q_2^{h,k_2}\to Q_2^{h,k_2}$ is defined as 
$\varphi_h \mapsto \cA^h(\varphi_h) := (\cS_{22}^h\circ \cS_{1}^h)(\varphi_h)$. 
\begin{lemma}\label{lipschitz-elast-disc}
    The operator $\cS_1^h$ is Lipschitz continuous with Lipschitz constant $L_{\cS_1^h} = \sqrt{M} \overline{C}_1 L_\ell$.
\end{lemma}

\begin{proof}
The result follows 
applying to Lemma~\ref{lem:elast-disc} analogous arguments used to prove Lemma~\ref{lipschitz-elast}.
\end{proof}

\begin{lemma}\label{lipschitz-transport-discrete}
    The operator $\cS_2^h$ is Lipschitz continuous with Lipschitz constant
    \begin{equation*}
        L_{\cS_2^h} = \max\left\{\frac{1}{\sqrt{2\mu}},\sqrt{2\mu} \right\}\sqrt{M^{3}}\overline{C}_2^2L_{\bbM}\left(\norm{\varphi_D}_{\frac{1}{2},\Gamma_D} + \norm{g}_{0,\Omega}\right).
    \end{equation*} 
\end{lemma}

\begin{proof}
    Similar techniques from Lemma~\ref{lipschitz-transport} adapted to Lemma~\ref{lem:trans-disc} conclude the proof.
\end{proof}

\begin{lemma}\label{lipschitz:cA-discrete}
    The operator $\cA^h$ is Lipschitz continuous with Lipschitz constant $L_{\cA^h}=L_{\cS_2^h}L_{\cS_1^h}$.
\end{lemma}
\begin{proof}
    It follows directly from Lemmas~\ref{lipschitz-elast-disc}-\ref{lipschitz-transport-discrete}.
\end{proof}
Next, we define the set
\begin{equation}\label{w-discrete}
\cblue{W^h} =\left\{ \cblue{w_h} \in Q_2^{h,k_2} \colon \norm{w_h}_{Q_2} \leq \sqrt{M} \overline{C}_2 \left(\norm{\varphi_D}_{\frac{1}{2},\Gamma_D} + \norm{g}_{0,\Omega}\right) \right\},    
\end{equation}
which satisfies that \cgreen{$\cA^h(Q_2^{h,k_2})\subseteq W^h$}. Indeed, from  \eqref{dependence-data-transport-discrete}, we readily see that 
$$\norm{\cA^h(\varphi_h)}_{Q_2} = \norm{S_{22}^h(S_1^h(\varphi_h))}_{Q_2} \leq \norm{S_2^h(S_1^h(\varphi_h))}_{\bV_2\times Q_2} \leq  \sqrt{M} \overline{C}_2 \left(\norm{\varphi_D}_{\frac{1}{2},\Gamma_D} + \norm{g}_{0,\Omega}\right),$$
\cgreen{it is clear that $\cA(W^h)\subseteq W^h$}. The following theorem provides the well-posedness of the discrete coupled problem.
\begin{theorem}\label{discrete-coupled-well-poss} \cblue{Under the assumptions of Lemmas~\ref{lem:elast-disc}-\ref{lem:trans-disc}, let $W^h$} be as in \eqref{w-discrete} and assume that \cgreen{small data is given such that $$ \overline{C}_1 L_\ell \max\left\{\frac{1}{\sqrt{2\mu}},\sqrt{2\mu} \right\}M^{2}\overline{C}_2^2 L_{\bbM}\left(\norm{\varphi_D}_{\frac{1}{2},\Gamma_D} + \norm{g}_{0,\Omega}\right) < 1.$$}
    Then the operator $\cA^h$ has a unique fixed point \cblue{$\varphi_h\in W^h$}. Equivalently, the discrete coupled problem \eqref{eq:weak-discrete} has a unique solution $(\bu_h,\tilde{p}_h,\bzeta_h,\varphi_h) \in \bV_1^{h,k_1}\times Q_1^{h,k_1} \times \bV_2^{h,k_2} \times Q_2^{h,k_2}$ and the following continuous dependence on data holds
    \begin{align*}    \norm{(\bu_h,\tilde{p}_h)}_{\bV_1\times Q_1} &\leq \overline{C}_1 \left( \norm{F_1^h}_{\bV'_1} + \norm{G^{\varphi_h}_1}_{Q'_1} \right),\qquad 
\norm{(\bzeta_h,\varphi_h)}_{\bV_2\times Q_2} \leq \overline{C}_2 \left( \norm{F_2}_{\bV'_2} +\norm{G_2}_{Q'_2} \right),
    \end{align*}
    where the corresponding constants $\overline{C}_1$ and $\overline{C}_2$ do not depend on the physical parameters. 
\end{theorem}
\begin{proof}
\cgreen{The unique solvability} follows readily from Lemma~\ref{lipschitz:cA-discrete} together with the Banach fixed-point theorem. \cgreen{The continuous dependence on data is a consequence of the corresponding bounds derived in Lemmas~\ref{lem:elast-disc}-\ref{lem:trans-disc}.}
\end{proof}

\section{A priori error analysis}\label{sec:error-analysis}
In this section, we aim to provide the convergence of the VE discretisation developed in Section~\ref{sec:vem} and derive the corresponding convergence  result which preserves the robustness proved in Theorem~\ref{th:unique-solvability} and Theorem~\ref{th:unique-solvability-discrete}. 
\begin{lemma} \label{lem:approximation-inequality}
    In addition to the assumptions of Theorems~\ref{coupled-well-poss} and \ref{discrete-coupled-well-poss}, let $(\bu,\tilde{p},\bzeta,\varphi)\in \bV_1 \times Q_1 \times \bV_2 \times Q_2$ and $(\bu_h,\tilde{p}_h,\bzeta_h,\varphi_h)\in \bV_1^{h,k_1}\times Q_1^{h,k_1}\times \bV_2^{h,k_2}\times Q_2^{h,k_2}$ be the unique solutions to \eqref{eq:weak} and \eqref{eq:weak-discrete}, respectively. 
    Then the following estimates hold
   \begin{subequations} \begin{align}
        \norm{(\bu-\bu_h,\tilde{p}-\tilde{p}_h)}_{\bV_1\times Q_1} \cblue{\leq} \,&\cblue{\overline{C}_1} \left( \cblue{\norm{\bu-\bPi_1^{\beps,k_1}\bu}_{\bV_1}} + \norm{F_1^h-F_1}_{\bV_1'} + \norm{\tilde{p}-\Pi_1^{0,k_1-1}\tilde{p}}_{Q_{b_1}}  \right .  \notag\\ &\left. \quad \cblue{+\norm{\bu-\bPi_1^{F,k_1}\bu}_{\bV_1}}\cblue{+ \sqrt{M}L_{\ell}\norm{\varphi-\varphi_h}_{Q_2}} \right)\label{approximation-inequality-elast} ,\\
        \norm{(\bzeta-\bzeta_h,\varphi-\varphi_h)}_{\bV_2\times Q_2} \cblue{\leq} \,& \cblue{\overline{C}_2} \left( \norm{\bzeta-\bPi_2^{0,k_2}\bzeta}_{\bbM} + \cblue{\norm{\bzeta-\bPi_2^{F,k_2}\bzeta}_{\bbM^h}} + \norm{\varphi-\Pi_2^{0,k_2}\varphi}_{Q_{b_2}} \right. \notag\\
        &\hspace{-2cm} \left. \cblue{+\max\left\{\frac{1}{\sqrt{2\mu}},\sqrt{2\mu} \right\}\sqrt{M^{3}} \overline{C}_2 L_{\bbM}(\norm{\varphi_D}_{\frac{1}{2},\Gamma_D} + \norm{g}_{0,\Omega})\norm{(\bu-\overline{\bu}_h,\tilde{p}-\tilde{p}_h)}_{\bV_1\times Q_1}}\right), \label{approximation-inequality-transport}
    \end{align}\end{subequations}
    where the \cblue{following discrete norm has been used $\displaystyle{\norm{\cdot}_{\bbM^{h}}^2} := a_2^{\cred{\overline{\bu}_h},\tilde{p}_h}(\cdot,\cdot)$}.
\end{lemma}
\begin{proof}
    From \eqref{eq:weak-elast} and \eqref{eq:weak-elast-discrete} we can readily see that $(\bu_h-\cblue{\bPi_1^{F,k_1}\bu},\tilde{p}_h-\cblue{\Pi_1^{0,k_1-1}\tilde{p}}) \in \bV_1^{h,k_1}\times Q_1^{h,k_1}$ is the unique solution to
    \begin{align*}
      a_1^h(\bu_h-\cblue{\bPi_1^{F,k_1}\bu},\bv_h) + b_1(\bv_h,\tilde{p}_h-\cblue{\Pi_1^{0,k_1-1}\tilde{p}}) &= \check{F}_1(\bv_h) \qquad \forall \bv_h \in \bV_1^{h,k_1},\\
      b_1(\bu_h-\cblue{\bPi_1^{F,k_1}\bu},\tilde{q}_h) - \frac{1}{\lambda}c_1(\tilde{p}_h-\cblue{\Pi_1^{0,k_1-1}\tilde{p}},\tilde{q}_h) &= \check{G}_1(\tilde{q}_h) \qquad \forall \tilde{q}_h \in Q_1^{h,k_1},
    \end{align*}
    where \begin{align*}
   \check{F}_1(\bv_h) &= a_1(\bu,\bv_h)-a_1^{h}(\cblue{\bPi_1^{F,k_1}\bu},\bv_h) + (F_1^{h}-F_1)(\bv_h) - b_1(\bv_h,\cblue{\Pi_1^{0,k_1-1}\tilde{p}}-\tilde{p}),\\
    \check{G}_1(\tilde{q}_h)&=-b_1(\cblue{\bPi_1^{F,k_1}\bu}-\bu,\tilde{q}_h)+\frac{1}{\lambda}c_1(\cblue{\Pi_1^{0,k_1-1}\tilde{p}}-\tilde{p},\tilde{q}_h) + \cblue{(G_1^{\varphi_h}-G_1^{\varphi})(\tilde{q}_h)}. 
    \end{align*}
    The continuous dependence on data \eqref{dependence-data-elast-discrete} shows
    \begin{align*}
        \norm{(\bu_h-\cblue{\bPi_1^{F,k_1}\bu},\tilde{p}_h-\cblue{\Pi_1^{0,k_1-1}\tilde{p}})}_{\bV_1\times Q_1} \cblue{\leq \overline{C}_1} \left(\norm{\check{F}_1}_{\bV'_1} + \norm{\check{G}_1}_{Q'_1}\right).
    \end{align*}
    The continuity from Lemma~\ref{lem:elast} (resp. Lemma~\ref{lem:elast-disc}) for $a_1(\cdot,\cdot)$ and $b_1(\cdot,\cdot)$ (resp. $a_1^{h}(\cdot,\cdot)$) together with $a_1^h(\cblue{\bPi_1^{\beps,k_1}\bu},\bu_h)=a_1(\cblue{\bPi_1^{\beps,k_1}\bu},\bu_h)$ and the Lipschitz continuity of $\ell(\cdot)$ given in \ref{eq:lipschitz-l-bound} provide 
    \begin{align*}
        \norm{\check{F}_1}_{\bV'_1} 
       & \cblue{\leq \overline{C}_1} \left( \cblue{\norm{\bu-\bPi_1^{\beps,k_1}\bu}_{\bV_1}} + \cblue{\norm{\bPi_1^{\beps,k_1}\bu-\bPi_1^{F,k_1}\bu}_{\bV_1}} + \norm{F_1^h-F_1}_{\bV'_1} + \norm{\tilde{p}-\Pi_1^{0,k_1-1}\tilde{p}}_{Q_{b_1}} \right),
  \\
        \norm{\check{G}_1}_{Q'_1} 
      &  \cblue{\leq \overline{C}_1} \left( \norm{\bu-\bPi_1^{F,k_1}\bu}_{\bV_1} + \norm{\tilde{p}-\Pi_1^{0,k_1-1}\tilde{p}}_{Q_{b_1}} + \cblue{\sqrt{M}L_{\ell}\norm{\varphi-\varphi_h}_{Q_2}} \right).
    \end{align*}
    On the other hand, the triangle inequality leads to
    \begin{gather*}
        \norm{\bu-\bu_h}_{\bV_1} - \norm{\bu-\cblue{\bPi_1^{F,k_1}\bu}}_{\bV_1} \leq \norm{\bu_h-\cblue{\bPi_1^{F,k_1}\bu}}_{\bV_1},\\ 
        \norm{\tilde{p}-\tilde{p}_h}_{Q_{b_1}} - \norm{\tilde{p}-\cblue{\Pi_1^{0,k_1-1}\tilde{p}}}_{Q_{b_1}}\leq \norm{\tilde{p}_h-\cblue{\Pi_1^{0,k_1-1}\tilde{p}}}_{Q_{b_1}},
    \\ 
        \cblue{\norm{\bPi_1^{\beps,k_1}\bu-\bPi_1^{F,k_1}\bu}_{\bV_1} \leq \norm{\bu-\bPi_1^{F,k_1}\bu}_{\bV_1} + \norm{\bu - \bPi_1^{\beps,k_1}\bu}_{\bV_1}}.
    \end{gather*}
    The combination of the estimates above proves \eqref{approximation-inequality-elast}.
    For the second inequality, it is easy to check from \eqref{eq:weak-transport} and \eqref{eq:weak-transport-discrete}  that $(\bzeta_h-\cblue{\bPi_2^{F,k_2}\bzeta},\varphi_h-\cblue{\Pi_2^{0,k_2}\varphi}) \in \bV_2^{h,k_2}\times Q_2^{h,k_2}$ is the unique solution to
    \begin{align*}
      \cblue{a_2^{\cred{\overline{\bu}_h},\tilde{p}_h,h}(\bzeta_h-\cblue{\bPi_2^{F,k_2}\bzeta},\bxi_h)} + b_2(\bxi_h,\varphi_h-\cblue{\Pi_2^{0,k_2}\varphi}) &= \check{F}_2(\bxi_h) \qquad \forall \bxi_h \in \bV_2^{h,k_2},\\
      b_2(\bzeta_h-\cblue{\bPi_2^{F,k_2}\bzeta},\psi_h) - \theta c_2(\varphi_h-\cblue{\Pi_2^{0,k_2}\varphi},\psi_h) &= \check{G}_2(\psi_h) \qquad \forall \psi_h \in Q_2^{h,k_2},
    \end{align*}
    where \begin{align*}\check{F}_2(\bxi_h) &= \cblue{a_2^{\bu,\tilde{p}}(\bzeta,\bxi_h)-a_2^{\cred{\overline{\bu}_h},\tilde{p}_h,h}(\bPi_2^{F,k_2}\bzeta,\bxi_h)} - b_2(\bxi_h,\cblue{\Pi_2^{0,k_2}\varphi}-\varphi),\\
    \check{G}_2(\psi_h)&=-b_2(\cblue{\bPi_2^{F,k_2}\bzeta}-\bzeta,\psi_h)+\theta c_2(\cblue{\Pi_2^{0,k_2}\varphi}-\varphi,\psi_h).\end{align*}
   Owing to the continuous dependence on data \eqref{dependence-data-transport-discrete} we arrive at
    \begin{align*}
        \norm{(\bzeta_h-\cblue{\bPi_2^{F,k_2}\bzeta},\varphi_h-\cblue{\Pi_2^{0,k_2}\varphi})}_{\bV_2\times Q_2} \cblue{\leq \overline{C}_2} \left( \norm{\check{F}_2(\bxi_h)}_{\bV'_2} + \norm{\check{G}_2(\psi_h)}_{Q'_2}\right).
    \end{align*}
    \cblue{The continuity proved} in Lemma~\ref{lem:trans} (resp. Lemma~\ref{lem:trans-disc}) for $b_2(\cdot,\cdot)$ (resp. \cblue{$a_2^{\cred{\overline{\bu}_h},\tilde{p}_h,h}(\cdot,\cdot)$)}, in addition with \cblue{$a_2^{\cred{\overline{\bu}_h},\tilde{p}_h,h}(\bPi_2^{0,k_2}\bzeta,\bzeta_h)=a_2^{\cred{\overline{\bu}_h},\tilde{p}_h}(\bPi_2^{0,k_2}\bzeta,\bzeta_h)$}, \cblue{imply} that 
    \begin{align*}
        \norm{\check{F}_2}_{\bV'_2} 
        &\cblue{\leq \overline{C}_2} \biggl( \norm{\bzeta-\cblue{\bPi_2^{0,k_2}\bzeta}}_{\bbM} + \cblue{\norm{\bPi_2^{0,k_2}\bzeta-\bPi_2^{F,k_2}\bzeta}_{\bbM^h}} + \norm{\varphi-\cblue{\Pi_2^{0,k_2}\varphi}}_{Q_{b_2}} +\\
        & \cblue{\max\left\{\frac{1}{\sqrt{2\mu}},\sqrt{2\mu} \right\}\sqrt{M^{3}} \overline{C}_2 L_{\bbM}\left(\norm{\varphi_D}_{\frac{1}{2},\Gamma_D} + \norm{g}_{0,\Omega}\right)\norm{(\bu-\overline{\bu}_h,\tilde{p}-\tilde{p}_h)}_{\bV_1\times Q_1}}\biggr),\\
        \norm{\check{G}_2}_{Q'_2} 
        &\cblue{\leq \overline{C}_2} \left( \norm{\bzeta-\cblue{\bPi_2^{F,k_2}\bzeta}}_{\bV_2} + \norm{\varphi-\cblue{\Pi_2^{0,k_2}\varphi}}_{Q_{b_2}}\right).
    \end{align*}
Then, as before, it suffices to apply triangle inequality properly,  
%
 allows us to \cblue{assert} \eqref{approximation-inequality-transport}.
\end{proof}
\begin{theorem}\label{convergence-rates}
    Adopt the assumptions of Lemma~\ref{lem:approximation-inequality} \cblue{and let} $(\bu,\tilde{p},\bzeta,\varphi)\in (\bH^{s_1+1}(\Omega)\cap \bV_1,|\cdot|_{s_1+1,\bV_1})\times (H^{s_1}(\Omega)\cap Q_{b_1},|\cdot|_{s_1,Q_{b_1}}) \times (\bH^{s_2+1}\cap \bV_2,|\cdot|_{s_2+1,\bV_2}) \times (H^{s_2+1}\cap Q_{b_2},|\cdot|_{s_2+1,Q_{b_2}})$, $\fb\in (\bH^{s_1-1}\cap \bQ_{b_1},|\cdot|_{s_1-1,\bQ_{b_1}})$ and $g\in (H^{s_2+1}(\Omega)\cap Q_{b_2},|\cdot|_{s_2+1,Q_{b_2}})$ with $0< s_1\leq k_1$ and $0\leq s_2\leq k_2$. \cblue{Assume further that 
    \begin{align*}
        \overline{C}_1 \sqrt{M} L_\ell + \overline{C}_2^2 \sqrt{M^3} L_\bbM \max\left\{\frac{1}{\sqrt{2\mu}},\sqrt{2\mu} \right\}  \left(\norm{\varphi_D}_{\frac{1}{2},\Gamma_D} + \norm{g}_{0,\Omega}\right) < \frac{1}{2}.
    \end{align*}
    Then the total error $\overline{\textnormal{e}}:=\norm{(\bu-\bu_h,\tilde{p}-\tilde{p}_h, \bzeta-\bzeta_h,\varphi-\varphi_h)}_{\bV_1\times Q_{1} \times \bV_2\times Q_2}$, satisfies the following rate}
    \begin{align}\label{convergence-rate}
         \cblue{\overline{\textnormal{e}} \lesssim h^{ \min \left\{s_1,s_2+1\right\}} (|\fb|_{s_1-1,\bQ_{b_1}} + |\bu|_{s_1+1,\bV_1} + |\tilde{p}|_{s_1,Q_{b_1}} + |g|_{s_2+1,Q_{b_2}} + |\bzeta|_{s_2+1,\bV_2}+|\varphi|_{s_2+1,Q_{b_2}})}.
    \end{align}
\end{theorem}
\begin{proof}
    \cblue{The} following estimate (see \cite[Lemma 3.2]{daveiga15-stokes})
    \begin{align*}
            |(F_1^{h,E}-F_1^E)(\bv_h)|\lesssim h_E^{s_1} |\fb|_{s_1-1,E}|\bv_h|_{1,E},
    \end{align*}
    leads to
    \begin{align*}
        |(F_1^h-F_1)(\bv_h)| \leq \sum_{E\in \mathcal{T}_h} |(F_1^{h,E}-F_1^E)(\bv_h)| \lesssim \sum_{E\in \mathcal{T}_h} h_E^{s_1} |\fb|_{s_1-1,E}|\bv_h|_{1,E} \lesssim h^{s_1} |\fb|_{s_1-1,\bQ_{b_1}}\norm{\bv_h}_{\bV_1}.
    \end{align*} 
    On the other hand, note that \cblue{$\norm{\cdot}_{\bbM^h}\leq \sqrt{M}\norm{\cdot}_{0,\Omega}$ and}
    \begin{align*}
         \sqrt{M}\norm{\vdiv \bzeta - \vdiv \bPi_2^{F,k_2}\bzeta}_{0,\Omega} &= \sqrt{M}\norm{\vdiv \bzeta - \Pi_2^{0,k_2}\vdiv \bzeta}_{0,\Omega} \\&= \sqrt{M}\norm{(g - \theta \varphi) - \Pi_2^{0,k_2}(g - \theta \varphi)}_{0,\Omega}\lesssim h^{s_2+1}(|g|_{s_2+1,Q_{b_2}}+|\varphi|_{s_2+1,Q_{b_2}}).
    \end{align*}
    Finally, \cblue{\eqref{convergence-rate} follows directly from adding \eqref{approximation-inequality-elast} and \eqref{approximation-inequality-transport} and a proper use of Lemmas~\ref{estimates-poly-proj-elast},~\ref{fortin-elast},~\ref{estimates-poly-proj-transport}, and \ref{fortin-transport}, together with the previous estimates}.
\end{proof}
\section{Numerical tests}\label{sec:results}
First, we verify the theoretical convergence rates proved in Section~\ref{sec:error-analysis} with different polygonal meshes and polynomial orders (to be specified in Section~\ref{sec:convergence-numerical}). Next, we evaluate the robustness provided in Theorem~\ref{discrete-coupled-well-poss} for different physical parameters (to be specified in Section~\ref{sec:robustness}). The numerical implementation is done with in-house MATLAB routines based on \cite{sutton2017virtual,yu2022mvem}. We implemented \eqref{eq:weak-elast-discrete} and \eqref{eq:weak-transport-discrete} separately for $k_1=2$ and $k_2=0,1$ respectively. 
The nonlinear coupled problem \eqref{eq:weak-discrete} is solved with a Picard iteration, following the same structure as in the fixed-point analysis from  Section~\ref{sec:discrete-fix-point}. 
The fixed-point tolerance is set to $10^{-8}$ and the errors are calculated 
as follows
\[\cblue{
    \overline{\textnormal{e}}_*^2 = \norm{(\bu-\bPi_1^{\beps,k_1}\bu_h,\tilde{p}-\Pi_1^{0,k_1-1}\tilde{p}_h)}_{\bV_1^{h,k_1}\times Q_1^{h,k_1}}^2 + \norm{(\bzeta-\bPi_2^{0,k_2}\bzeta_h,\varphi-\Pi_2^{0,k_2}\varphi_h)}_{\bV_2^{h,k_2}\times Q_2^{h,k_2}}^2.}
\]
\vspace{-0.435cm}
\begin{figure}[h!]
    \centering
    \subfigure[Non--convex.]{\includegraphics[width=0.155\textwidth]{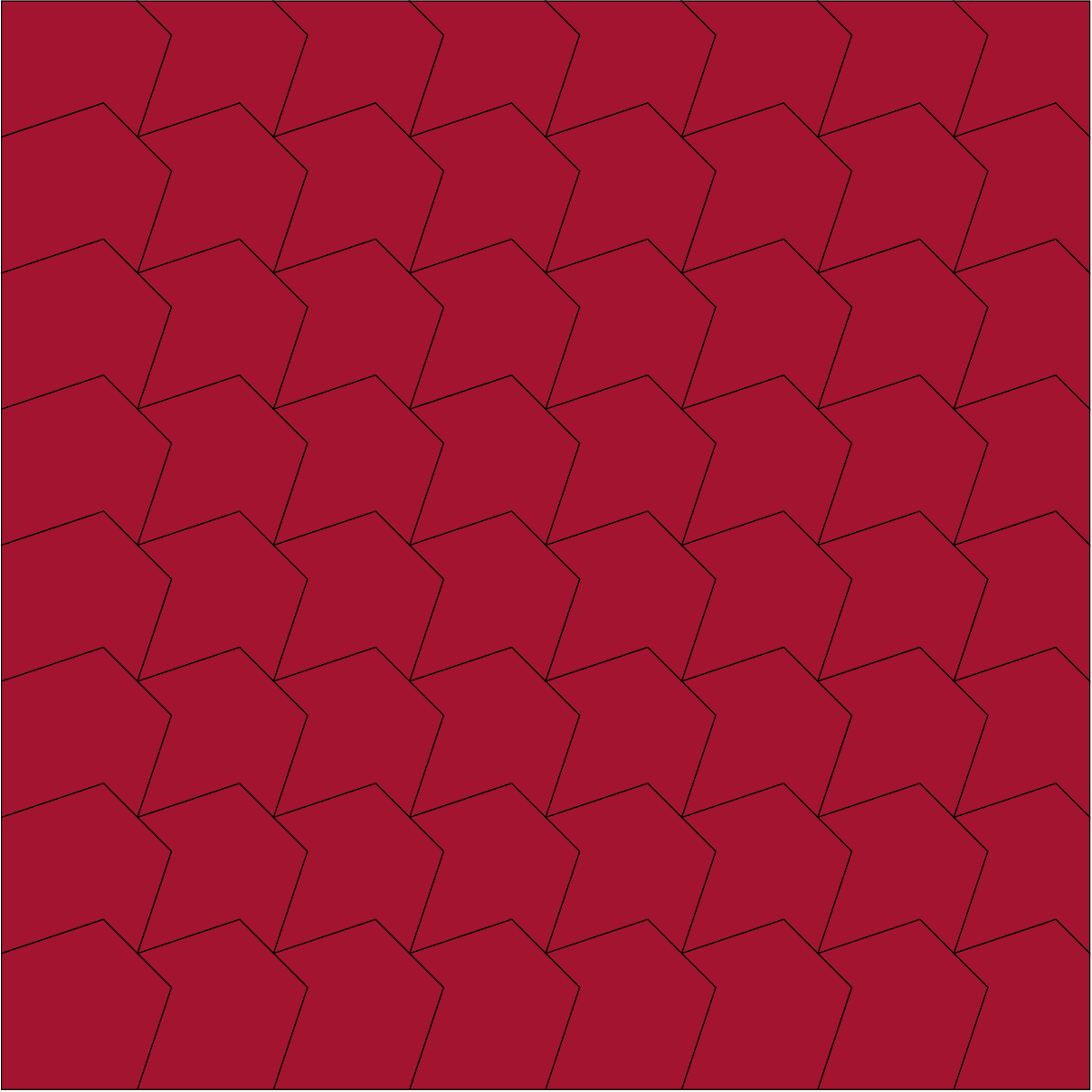}} \label{fig:non-convex}
    \subfigure[Voronoi.]{\includegraphics[width=0.155\textwidth]{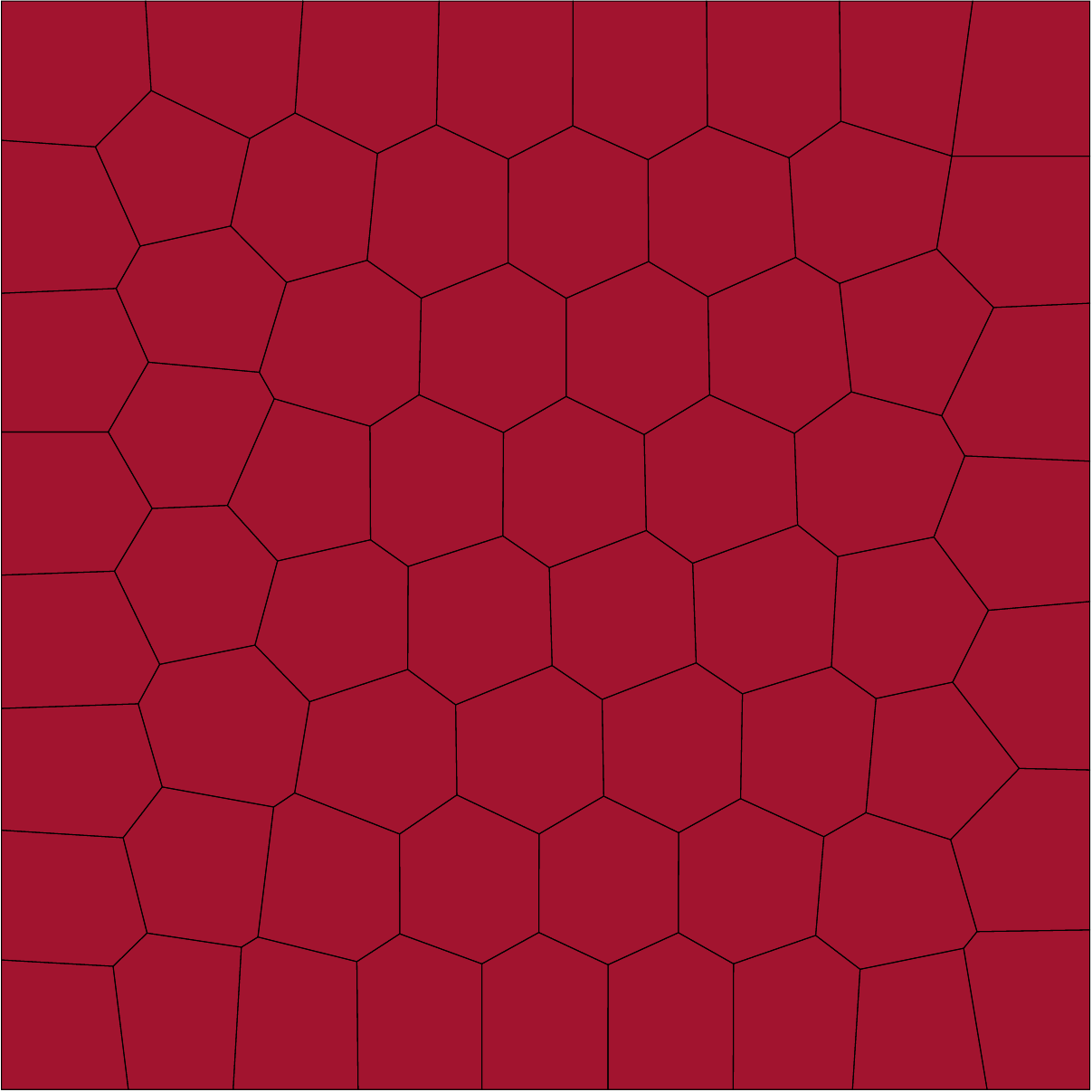}}  \label{fig:polygonal}
    \subfigure[Square.]{\includegraphics[width=0.155\textwidth]{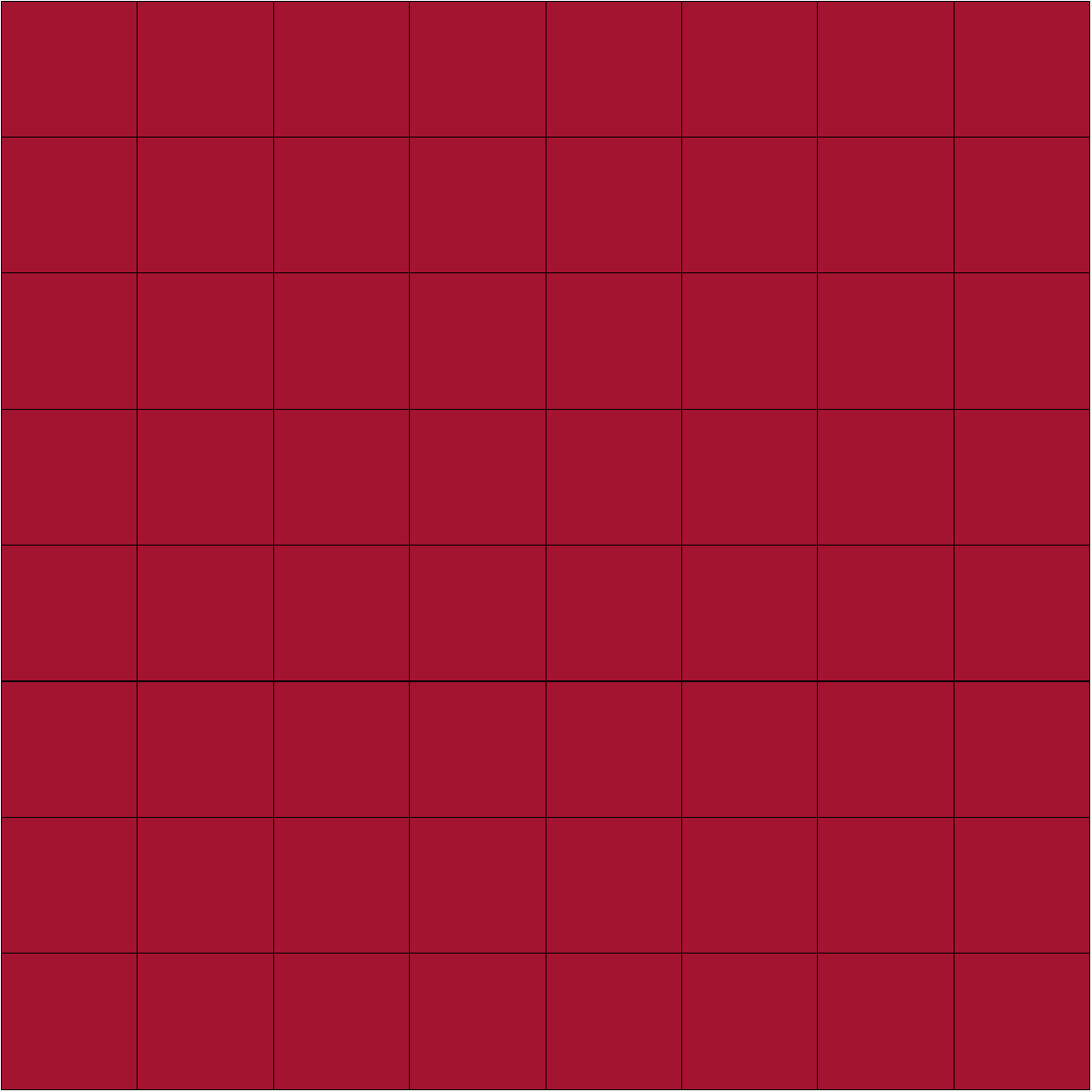}} \label{square}
    \subfigure[\cblue{Crossed}.]{\includegraphics[width=0.155\textwidth]{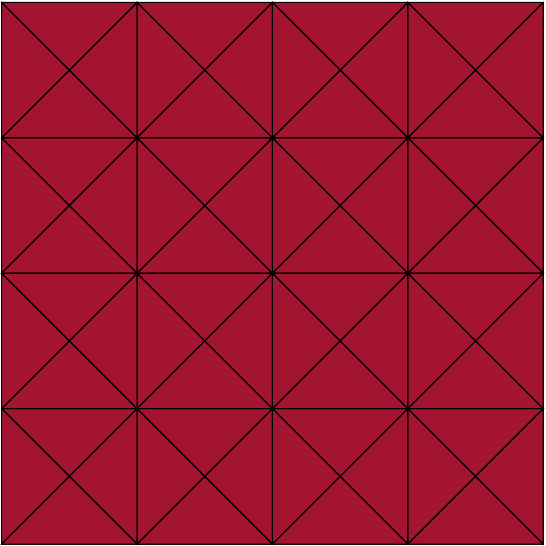}} \label{fig:crossed}
    \subfigure[\cblue{Per. Voronoi}.]{\includegraphics[width=0.155\textwidth]{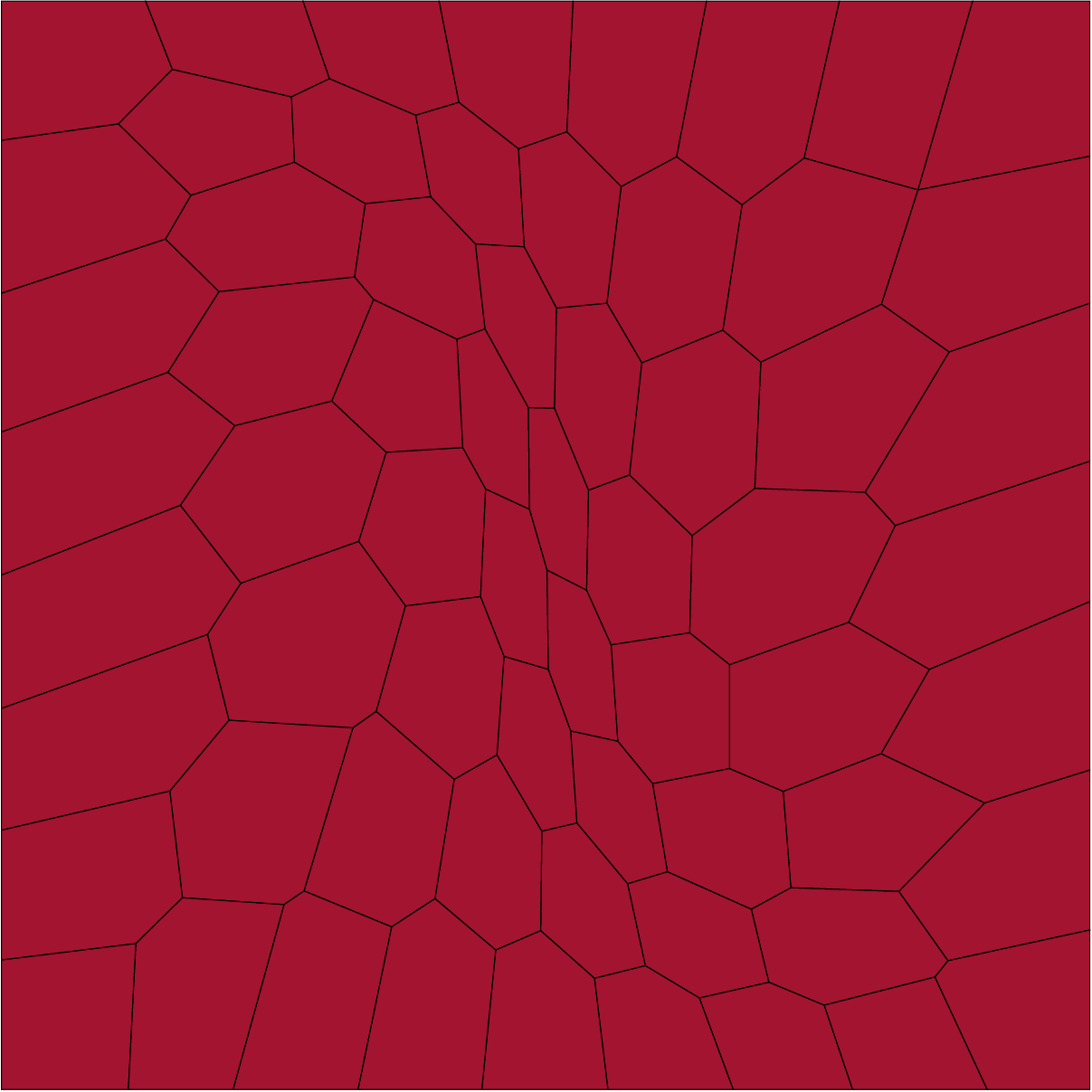}} \label{fig:distortionpolygonal}
    \subfigure[\cblue{Kangaroo.}]{\includegraphics[width=0.155\textwidth]{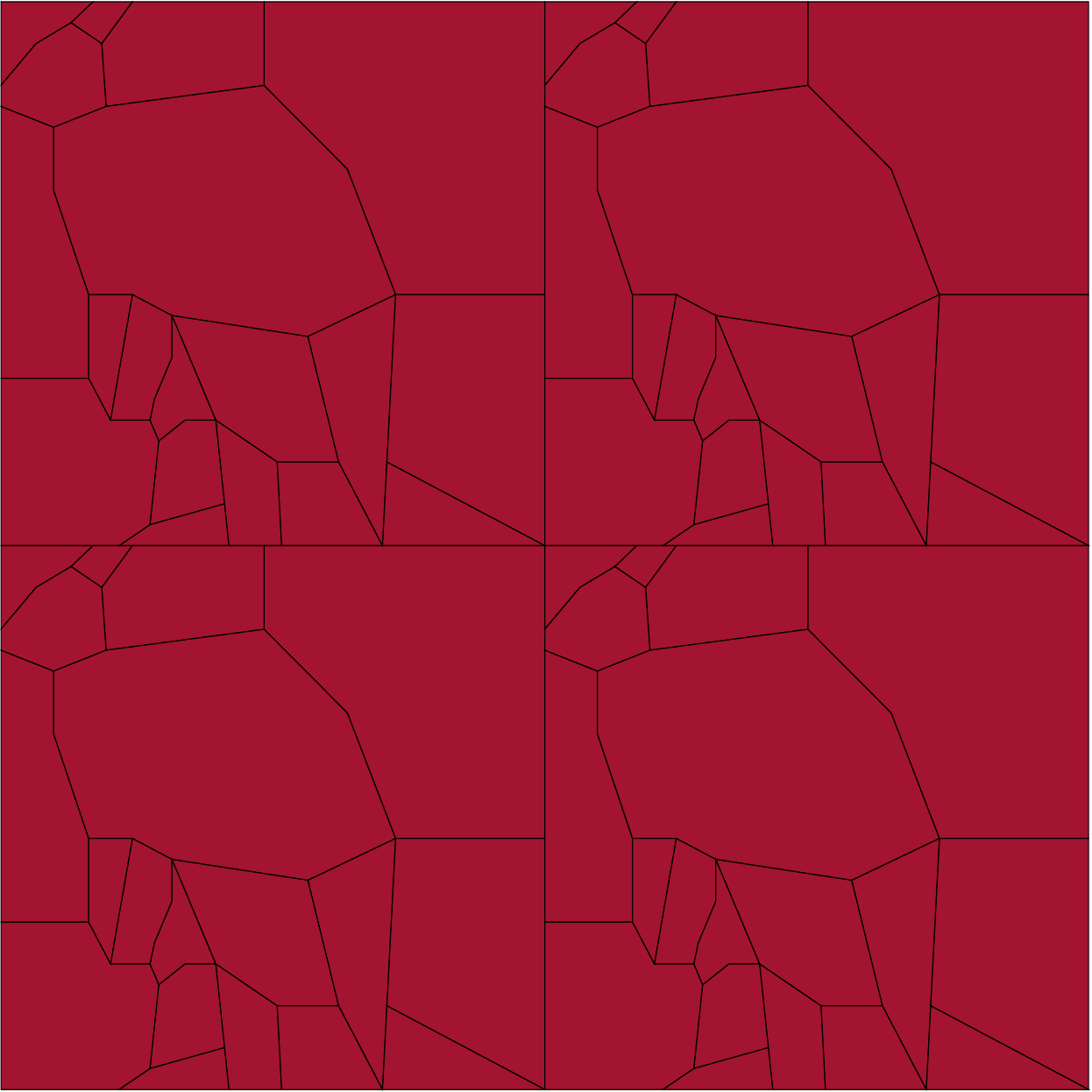}} \label{fig:kangaroo}
    \caption{\cblue{An illustration of six distinct coarse meshes used in the numerical tests}.}
    \label{fig:meshes}
\end{figure}

\begin{figure}[h!]
    \centering
    \subfigure[\cred{DoFs $\bV_1^{h,2}$}]{\includegraphics[width=0.12\textwidth]{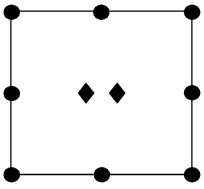}} 
    \subfigure[DoFs $Q_1^{h,2}$]{\includegraphics[width=0.12\textwidth]{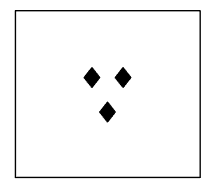}} 
    \subfigure[DoFs $\bV_2^{h,1}$]{\includegraphics[width=0.12\textwidth]{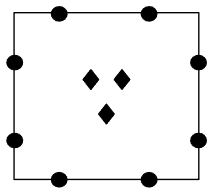}}
    \subfigure[DoFs $Q_2^{h,1}$]{\includegraphics[width=0.12\textwidth]{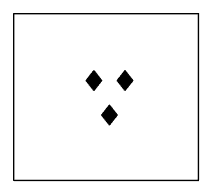}}
    \caption{Illustration of the DoFs on a square element with $k_1=2$ and $k_2=1$.}
    \label{fig:dofs}
\end{figure}

The linear systems arising from the VE discretisation are solved using a direct solver. \cred{In addition, we refer to \cite{veiga14} and \cite{daVeiga16} for the explicit construction of the linear systems involved.} For the initial guess $\varphi_0 \in Q_2^{h,k_2}$, the scaled monomial basis introduced in Section~\ref{sec:vem} leads to $\varphi_0 = \mathbf{c}\cdot (1)$ and $\varphi_0 = \mathbf{c}\cdot (1,\frac{x-x_E}{h_E},\frac{y-y_E}{h_E})^{\tt t}$ for $k_2=0,1$ respectively and $\mathbf{c}$ a constant vector of the appropriate size. With a slight abuse of notation, we define $\overline{\bzeta}_h = \bPi_2^{0,k_2}\bzeta_h$ in Algorithm \ref{alg:fp}, where the implementation of the fixed point iteration is illustrated.

\begin{algorithm}
\caption{Fixed point iteration}\label{alg:fp}
\cred{\begin{algorithmic}
\State $i \gets 0$, $\text{tol} \gets 1.\text{e}^{-8}$, $\text{error} \gets 1$, $\varphi_{i} \gets \mathbf{c}$
\While{$\text{tol} < \text{error}$}
    \State \textbf{solve} for $(\bu_h,\tilde{p}_h)$ from  $\begin{bmatrix}
        A_1 & B_1 \\
        B^{\tt t}_1 & - C_1 
    \end{bmatrix}
    \begin{bmatrix}
        \bu_h \\
        \tilde{p}_h 
    \end{bmatrix} =
    \begin{bmatrix}
        F_1 \\
        G^{\varphi_{i}}_1 
    \end{bmatrix}$
    \State $(\bu_{i+1},\tilde{p}_{i+1}) \gets (\bu_h, \tilde{p}_h)$
    \State \textbf{solve} for $(\bzeta_h,\varphi_h)$ from $\begin{bmatrix}
        A^{\overline{\bu}_{i+1},\tilde{p}_{i+1}}_2 & B_2 \\
        B^{\tt t}_2 & - C_2 
    \end{bmatrix}
    \begin{bmatrix}
        \bzeta_h \\
        \varphi_h 
    \end{bmatrix} =
    \begin{bmatrix}
        F_2 \\
        G_2 
    \end{bmatrix}$
    \State $(\bzeta_{i+1},\varphi_{i+1}) \gets (\bzeta_h, \varphi_h)$
    
\If{$0<i$}
    \State error $\gets \left(\norm{(\overline{\bu}_{i+1}-\overline{\bu}_{i},\tilde{p}_{i+1}-\tilde{p}_{i})}_{\bV_1^{h,k_1}\times Q_1^{h,k_1}}^2 + \norm{(\overline{\bzeta}_{i+1}-\overline{\bzeta}_{i},\varphi_{i+1}-\varphi_{i})}_{\bV_2^{h,k_2}\times Q_2^{h,k_2}}^2\right)^\frac{1}{2}$
    \If{$\text{error}<\text{tol}$}
        \State \textbf{break}: fixed point converged
    \EndIf
\EndIf
\State $(\bu_{i},\tilde{p}_{i}) \gets (\bu_{i+1}, \tilde{p}_{i+1})$
\State $(\bzeta_{i},\varphi_{i}) \gets (\bzeta_{i+1}, \varphi_{i+1})$
\State $i \gets i+1$
\EndWhile
\end{algorithmic}}
\end{algorithm}

\subsection{Example 1: Convergence test with smooth solution} \label{sec:convergence-numerical}
This test is designed with the following manufactured solutions 
\[
    \bu(x,y) = \frac15\left( \cos(x)\sin(y)x+x^2, \sin(x)\cos(y)x+y^2 \right)^{\tt t},\quad 
    \varphi(x,y) = x^2+y^2 + \sin(\pi x) + \cos(\pi y),
\]
on $\Omega = (0,1)^2$ with Dirichlet part $\Gamma_D = \{(x,y) \in \partial \Omega \colon x=0\; \text{or}\; y = 0\}$ and Neumann part $\Gamma_N = \partial \Omega\setminus \Gamma_D$. The exact displacement \cblue{$\bu$ and concentration $\varphi$} are used to compute exact Herrmann pressure $\tilde{p}$ and total flux $\bzeta$, as well as appropriate forcing term, source, and non-homogeneous traction, displacement, and flux boundary data. 
The nonlinear terms are taken in the form
\[
    \bbM(\bsigma) = m_0\exp(-m_1 \tr\bsigma)\bbI,\quad 
    \ell(\vartheta) = K_0+\frac{\vartheta^n}{K_1+\vartheta^n},
\]
with the arbitrary parameters (dimensionless in this case) $\lambda = 10^3, \quad
\mu = 10^2,\quad
\theta = 10^{-3}, \quad 
m_0 = 10^{-1}, \quad
m_1 = 10^{-4}, \quad
K_0 = 1, \quad
K_1 = 1, \quad n=2$.
For the remaining parameter $M$, we use the \texttt{fmincom} MATLAB optimisation subroutine \cblue{and the Frobenius matrix norm $\norm{\cdot}_{F}$ to compute  
$M = \max_{(x,y)\in \Omega} \quad \norm{\bbM(\sigma)}_{F} = 1.157701\times 10^1$.} 

The performance of the VEM is verified on \cblue{six} different types of meshes: \cblue{non-convex, Voronoi, square, crossed, perturbed Voronoi and kangaroo-shaped} (see Figure~\ref{fig:meshes}). As an illustrative example, the DoFs for a single square element are shown in Figure~\ref{fig:dofs} for $k_1=2$ and $k_2=0,1$. It is easy to check that the total number of DoFs for the spaces  in Section~\ref{sec:vem} are given by
$\text{DOF}_1^j(k_1=2) = 3N_E^j + 2N_e^j + 2N_v^j + 3N_E^j, \text{ DOF}_2^j(k_2=0) = N_e^j + N_E^j \text{ and }
    \text{DOF}_2^j(k_2=1) = 3N_E^j+2N_e^j + 3N_E^j$,
with $N_E^j$ (resp. $N_e^j$, $N_v^j$) corresponding to the number of elements in the mesh $j$ (resp. edges, vertices), \cblue{the total number of DoFs is defined as $\text{DOF}_*^j = \text{DOF}_1^j + \text{DOF}_2^j$}. Besides, the experimental order of convergence \cblue{$r$} for the error \cblue{$\overline{\text{e}}_*$} of the refinement $j+1$ is computed from the formula \cblue{$r^{j+1} = \log(\overline{\text{e}}_{*}^{j+1}/\overline{\text{e}}_{*}^{j})/\log(h^{j+1}/h^{j})$, with $j=1,2,3$}. \cblue{Finally, the stabilisation terms are selected as follows $$S_1^E(\bu_h,\bv_h) = 2\mu \sum_{i=1}^{\text{DOF}_1}\text{dof}_{i}(\bu_h)\text{dof}_{i}(\bv_h), \, S_2^{\cred{\overline{\bu}_h,\tilde{p}_h},E}(\bzeta_h,\bxi_h) = \norm{\int_E \bbM^{-1}(\beps(\cred{\overline{\bu}_h}),\tilde{p}_h)}_{F} \sum_{i=1}^{\text{DOF}_2}\text{dof}_{i}(\bzeta_h)\text{dof}_{i}(\bxi_h).$$} 

\cblue{\cblue{Tables}~\ref{tab:convergence-0}-\ref{tab:convergence-1}} summarise the convergence history of the mixed VEM \eqref{eq:weak-discrete}, and we observe the convergence rate \cblue{given by the minimum between $O(h^{k_1})$ and $O(h^{k_2+1})$} as predicted in Theorem~\ref{convergence-rates}. \cblue{In addition}, we present \cblue{snapshots} of the approximated solutions \cblue{for the magnitude of displacement $|\bu_h|$ for the magnitude of the total flux $|\bzeta_h|$}, the Herrmann pressure $\tilde{p}_h$ and the concentration $\varphi_h$ (see Figure~\ref{fig:sol_snap}).

\begin{table}[h!]
\caption{\cgreen{Example 1. Convergence history for smooth manufactured solutions, using the polynomial degrees $k_1=2$ and $k_2=0$.}}
\label{tab:convergence-0}
\centering

\cblue{\begin{tabular}{| c | c | c | c | c | c | c | c | c | c |}

\hline &  &  &  & \\[-2.5ex] 
 
Mesh & $\text{DOF}$ & $\overline{\text{e}}_*$ & $r$ & it & Mesh & $\text{DOF}$ & $\overline{\text{e}}_*$ & $r$ & it \\ 
 \hline
             & 1538  & 3.83e-01 &  *   & 3  &            & 3026  & 2.13e-01 &  *   & 3 \\ 

 Non--convex & 3458  & 2.52e-01 & 1.04 & 3  & Crossed    & 6746  & 1.41e-01 & 1.01 & 3 \\

             & 6146  & 1.88e-01 & 1.02 & 3  &            & 11938 & 1.06e-01 & 1.00 & 3 \\

             & 9602  & 1.50e-01 & 1.01 & 3  &            & 18602 & 8.46e-02 & 1.00 & 3 \\ 
 
\hline &  &  &  & \\[-2.5ex] 
 
             & 1218  & 3.69e-01 &  *   & 3  &            & 1218  & 4.39e-01 &  *   & 3 \\ 

 Voronoi     & 2728  & 2.44e-01 & 1.03 & 3  & Perturbed  & 2728  & 2.88e-01 & 1.03 & 3 \\

             & 4826  & 1.80e-01 & 1.04 & 3  & Voronoi    & 4826  & 2.14e-01 & 1.04 & 3 \\

             & 7592  & 1.44e-01 & 0.99 & 3  &            & 7592  & 1.70e-01 & 1.04 & 3 \\ 

\hline &  &  &  & \\[-2.5ex] 

             & 978   & 3.70e-01 &  *   & 3  &            & 1502  & 5.52e-01 &  *   & 3 \\ 

 Square      & 2138  & 2.44e-01 & 1.03 & 3  & Kangaroo   & 3371  & 3.69e-01 & 0.99 & 3 \\

             & 3746  & 1.82e-01 & 1.01 & 3  &            & 5986  & 2.80e-01 & 0.96 & 3 \\

             & 5802  & 1.45e-01 & 1.01 & 3  &            & 9347  & 2.26e-01 & 0.95 & 3 \\ 
 
\hline 

\end{tabular}}
\end{table}

\begin{table}[h!]
\caption{\cgreen{Example 1. Convergence history for smooth manufactured solutions, using the polynomial degrees $k_1=2$ and $k_2=1$.}}
\label{tab:convergence-1}
\centering
\cblue{\begin{tabular}{| c | c | c | c | c | c | c | c | c | c |}

\hline &  &  &  & \\[-2.5ex] 
 
Mesh & $\text{DOF}$ & $\overline{\text{e}}_*$ & $r$ & it & Mesh & $\text{DOF}$ & $\overline{\text{e}}_*$ & $r$ & it \\ 
 \hline
             & 2114  & 9.67e-02 &  *   & 3  &            & 4706  & 2.70e-02 &  *   & 3 \\ 

 Non--convex & 4754  & 4.27e-02 & 2.02 & 3  & Crossed    & 10514 & 1.18e-02 & 2.03 & 3 \\

             & 8450  & 2.39e-02 & 2.01 & 3  &            & 18626 & 6.62e-03 & 2.02 & 3 \\

             & 13202 & 1.53e-02 & 2.01 & 3  &            & 29042 & 4.22e-03 & 2.02 & 3 \\ 
 
\hline &  &  &  & \\[-2.5ex] 
 
             & 1730  & 8.10e-02 &  *   & 3  &            & 1730  & 1.17e-01 &  *   & 3 \\ 

 Voronoi     & 3878  & 3.59e-02 & 2.01 & 3  & Perturbed  & 3878  & 5.19e-02 & 2.01 & 3 \\

             & 6866  & 2.04e-02 & 1.96 & 3  & Voronoi    & 6866  & 3.00e-02 & 1.91 & 3 \\

             & 10790 & 1.25e-02 & 2.20 & 3  &            & 10790 & 1.88e-02 & 2.09 & 3 \\ 

\hline &  &  &  & \\[-2.5ex] 

             & 1442  & 8.14e-02 &  *   & 3  &            & 2126  & 1.43e-01 &  *   & 4 \\ 

 Square      & 3170  & 3.62e-02 & 2.00 & 3  & Kangaroo   & 4775  & 6.13e-02 & 2.10 & 4 \\

             & 5570  & 2.03e-02 & 2.00 & 3  &            & 8482  & 3.38e-02 & 2.07 & 3 \\

             & 8642  & 1.30e-02 & 2.00 & 3  &            & 13247 & 2.14e-02 & 2.05 & 3 \\ 
 
\hline
\end{tabular}}
\end{table}

\subsection{Example 2: Robustness test for a variation in the physical parameters} \label{sec:robustness}
In this example, we consider \cblue{successive refinement for the meshes shown in Figure~\ref{fig:meshes}}, and expand upon the example introduced in Subsection~\ref{sec:convergence-numerical} by introducing variations in the physical parameters $\lambda$, $\mu$ and $\theta$. We recall that the parameters that remain constant throughout the analysis are held fixed with the values outlined in Subsection~\ref{sec:convergence-numerical}. As predicted in Theorem~\ref{convergence-rates}, Figure~\ref{fig:robust_graph_complete} shows that the variation of the physical parameters $\lambda$, $\mu$, \cblue{$M$}, and $\theta$ does not affect the expected convergence rates \cblue{for all the meshes implied in this test. In addition, Figure~\ref{fig:robust_graph_separated} illustrates how the lowest accuracy (in this case $k_2$) dominates the convergence rate (see Figures~\ref{fig:robust_graph_complete}, \ref{fig:robust_graph_separated})}.

\subsection{Example 3: Lithiation of an anode}

In this example, we investigate the 2D version of the model presented in \cite{Taralov2015}, and simulate the microscopic lithiation of a perforated circular anode particle.  The outer and inner radius of the particle are $\rho_o = \qty{5}{\um}$ and $\rho_i = \qty{1}{\um}$, respectively. We discretise the domain with a Voronoi mesh of 10000 elements and the VE spaces are selected as in Section~\ref{sec:vem} with $k_1=2$ and $k_2=1$. In this test, the particle is clamped and zero lithium fluxes are prescribed on \cblue{the inner circumference} $\Gamma_i$. The maximum lithium concentration \cblue{$\hat{\Omega}=\qty{2.29d-14}{\frac{\mole}{\mu \m^3}}$} is fixed on the \cblue{outer circumference} $\Gamma_o$. We check two cases: \cblue{no} traction and a traction of \cblue{$\qty{-2d-4}{\frac{\N}{\mu \m^2}}$} on $\Gamma_o$. The Young's modulus is given by $E = \qty{1d-2}{\frac{\N}{\mu \m^2}}$ and the Poisson ratio is $\nu=0.3$, and from these values, we compute $\lambda = \frac{E\nu}{(1+\nu)(1-2\nu)}$ and $\mu =\frac{E}{2(1+\nu)}$. The diffusive source is zero and there is no body load force. The stress-assisted diffusion coefficient is given by \cblue{$\bbM(\bsigma)= m_0(\bbI+m_0m_1 \sigma^2)$ with $m_0=\qty{1d2}{\frac{\mu \m^2}{\s}}$, $m_1=0$  for the decoupled case and $m_1= \qty{1d3}{\frac{\mu \m^2}{\N}}$ for the coupled case}. The active stress is measured as follows $\ell(\varphi)=K_0 \varphi$, where $K_0=\tilde{\Omega}\frac{2\mu+3\lambda}{3}$ and the partial molar volume is \cblue{$\tilde{\Omega}=\qty{3.497d12}{\frac{\mu \m^3}{\mole}}$}. Finally, the remaining parameters are set as $\theta = M = 1$.

In Figure~\ref{fig:radial-concentration} we can observe \cblue{that the non--linear coupling term} $\bbM(\bsigma)$ in each case (no traction and traction applied) has an impact on the behaviour of the concentration along the radial direction of the particle. \cblue{In addition, the pressure profiles in Figure~\ref{fig:pcp} show a change given by the traction boundary conditions imposed. The results presented coincide} qualitatively with the simulations reported in \cite{Taralov2015}.

\begin{figure}[h!]
    \centering
    \subfigure[Displacement, pressure for $k_1=2$ and $k_2=0$.]{\includegraphics[width=0.49\textwidth]{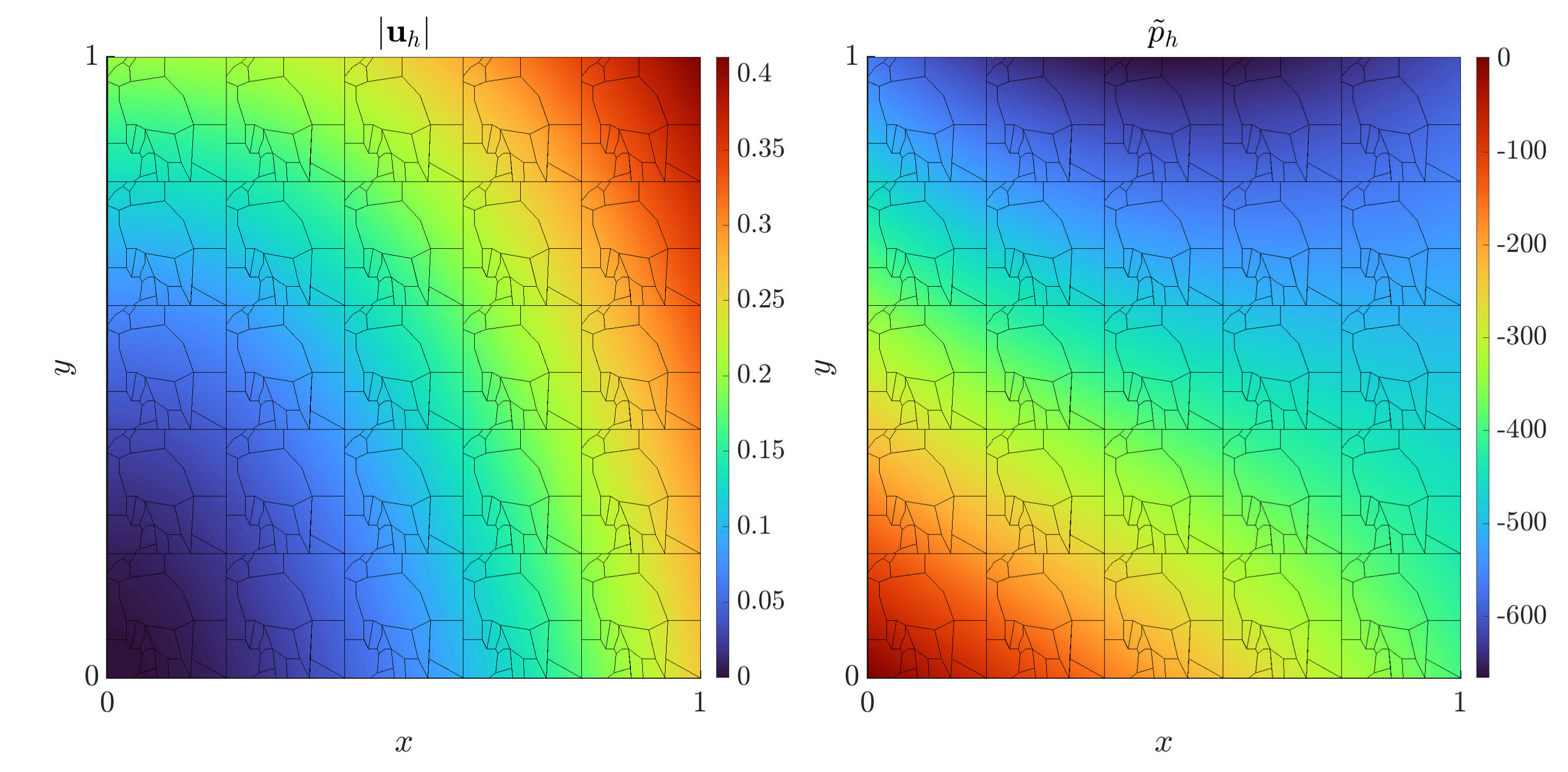}} 
    \subfigure[Displacement, pressure for $k_1=2$ and $k_2=1$.]{\includegraphics[width=0.49\textwidth]{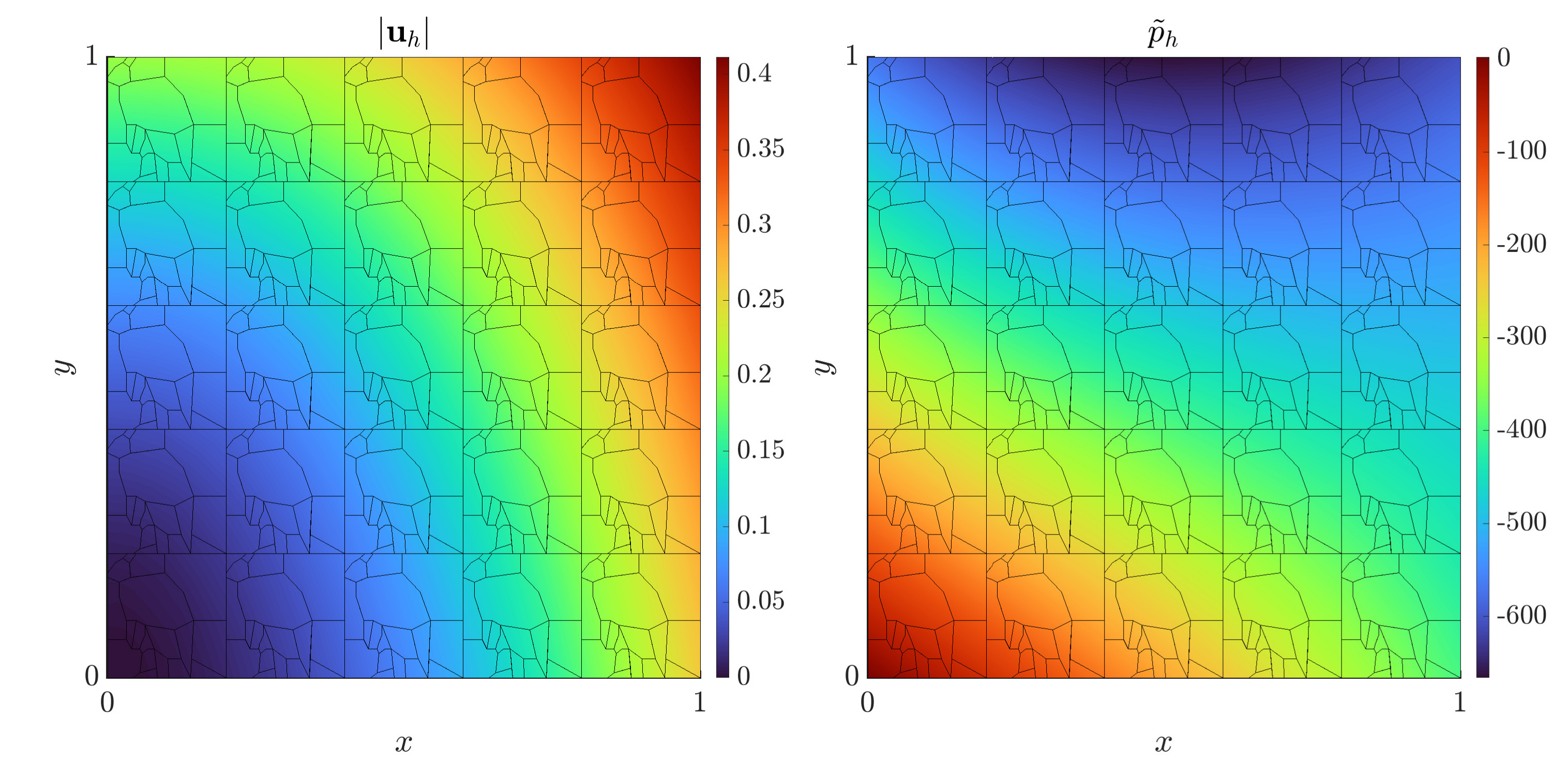}} 
    \subfigure[Total flux, concentration for $k_1=2$ and $k_2=0$.]{\includegraphics[width=0.49\textwidth]{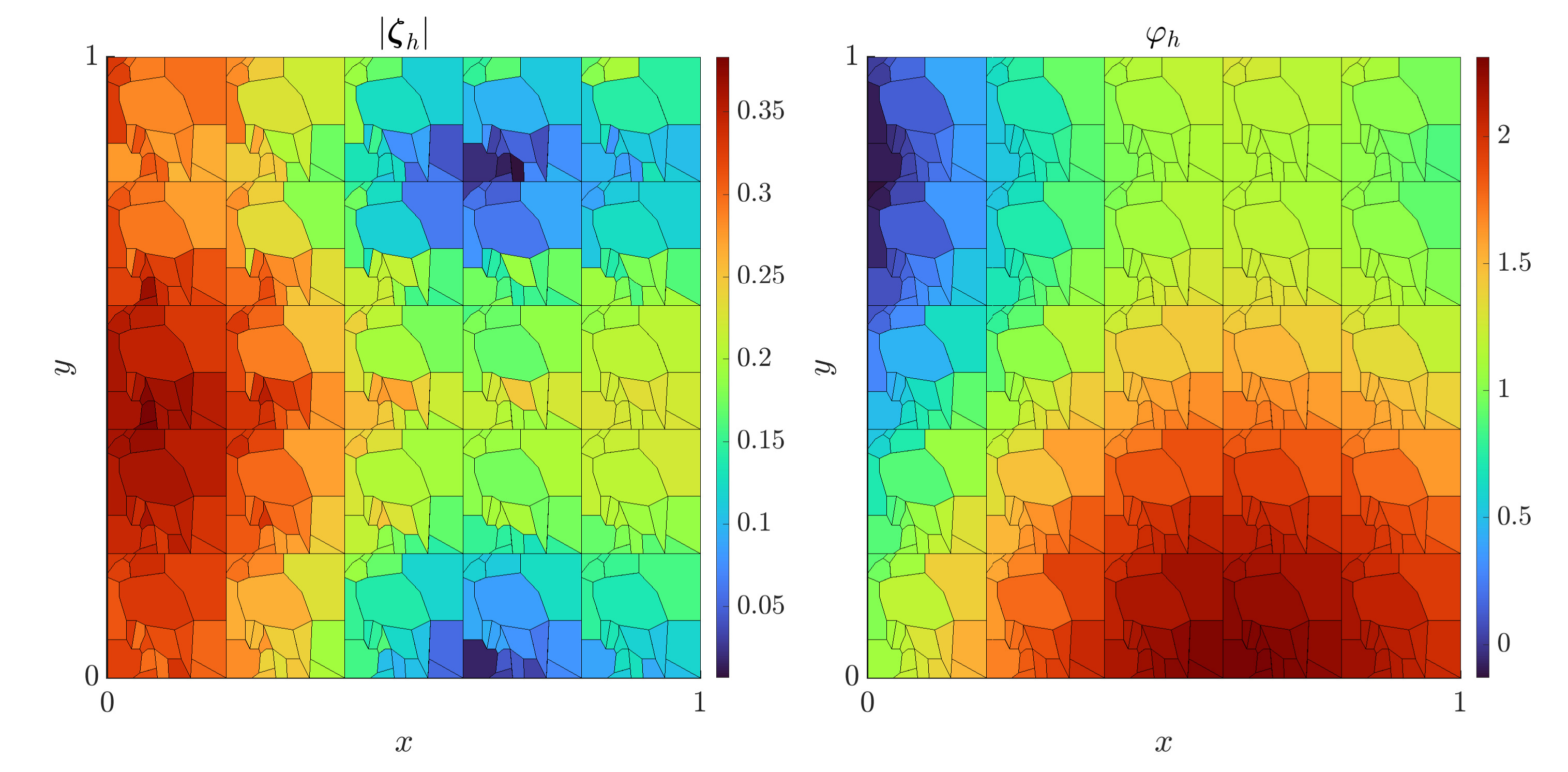}} 
    \subfigure[Total flux, concentration for $k_1=2$ and $k_2=1$.]{\includegraphics[width=0.49\textwidth]{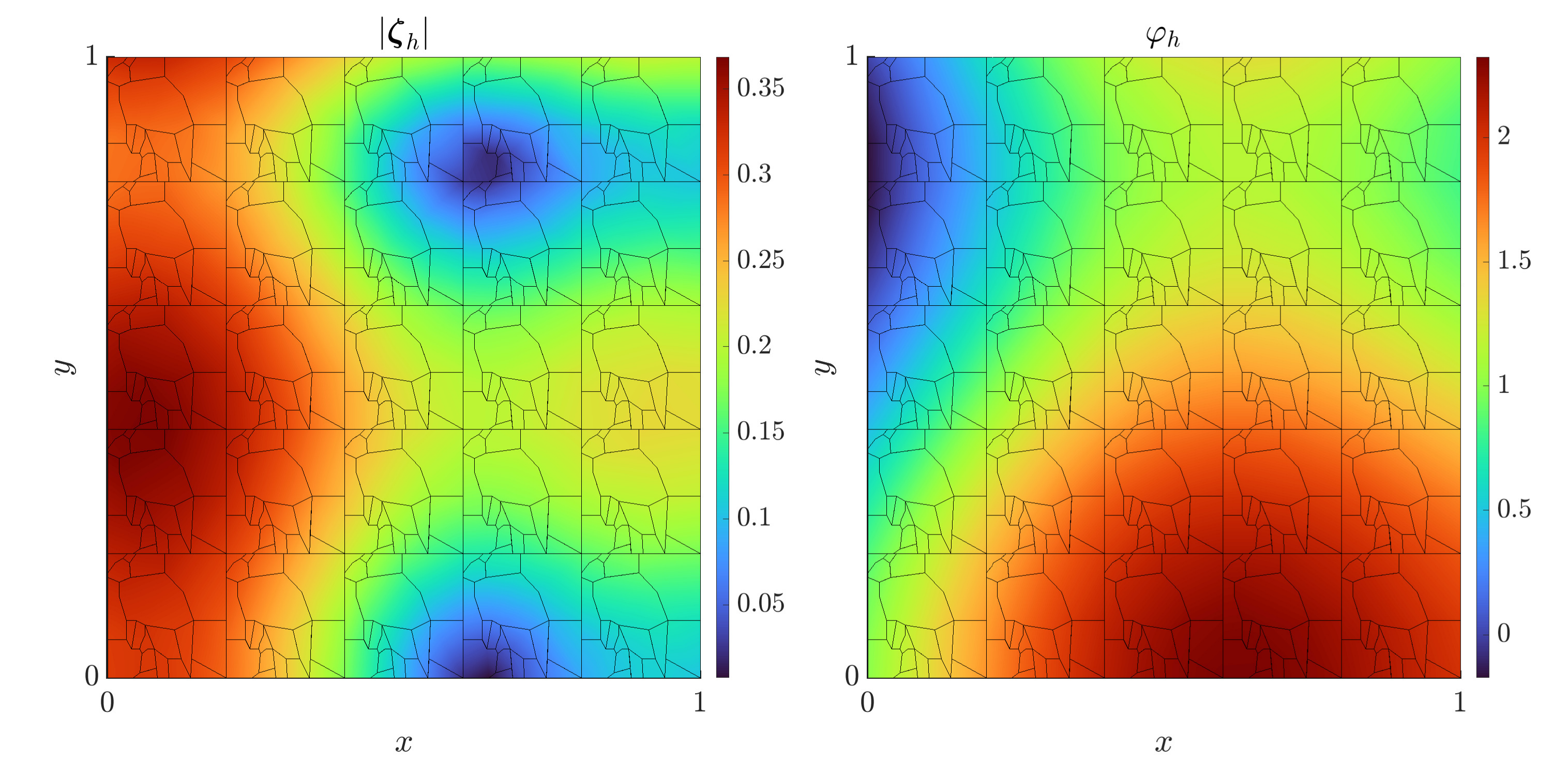}} 
    \caption{\cblue{Example 1. Snapshot of the approximated solution with the displacement magnitude $\bu_h$, flux magnitude $|\bzeta_h|$, Herrmann pressure $\tilde{p}_h$, and concentration $\varphi_h$, shown over an in-house kangaroo mesh for $k_1=2$ and $k_2=0,1$}.}
    \label{fig:sol_snap}
\end{figure}

 \begin{figure}[h!]
    \centering
    \subfigure[Non-convex: $k_2=0$.]{\includegraphics[width=0.31\textwidth,trim={0 0.5cm 2.2cm 0},clip]{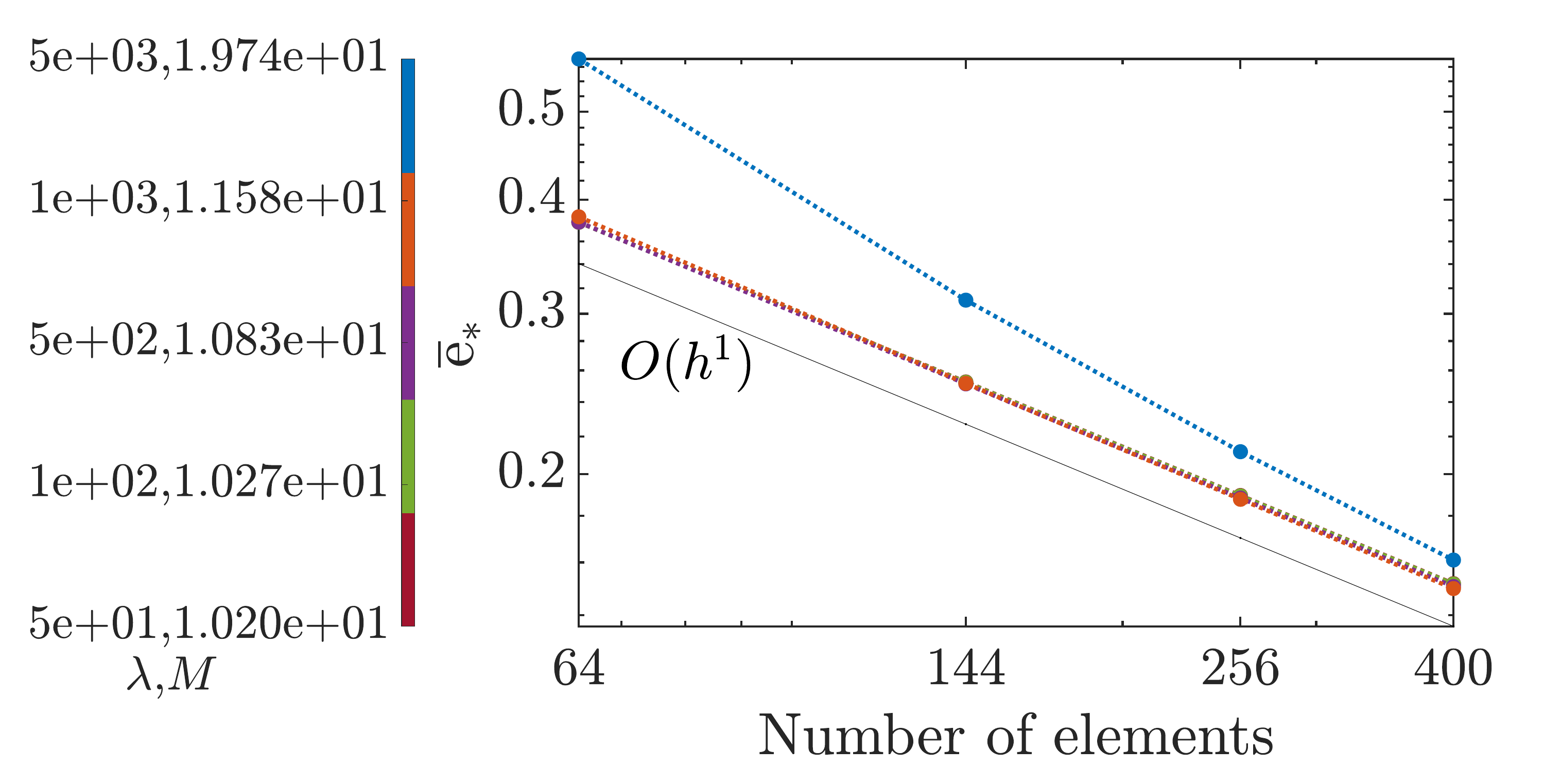}}
    \subfigure[Non-convex: $k_2=1$.]{\includegraphics[width=0.223\textwidth,trim={14.3cm 0.5cm 2.2cm 0},clip]{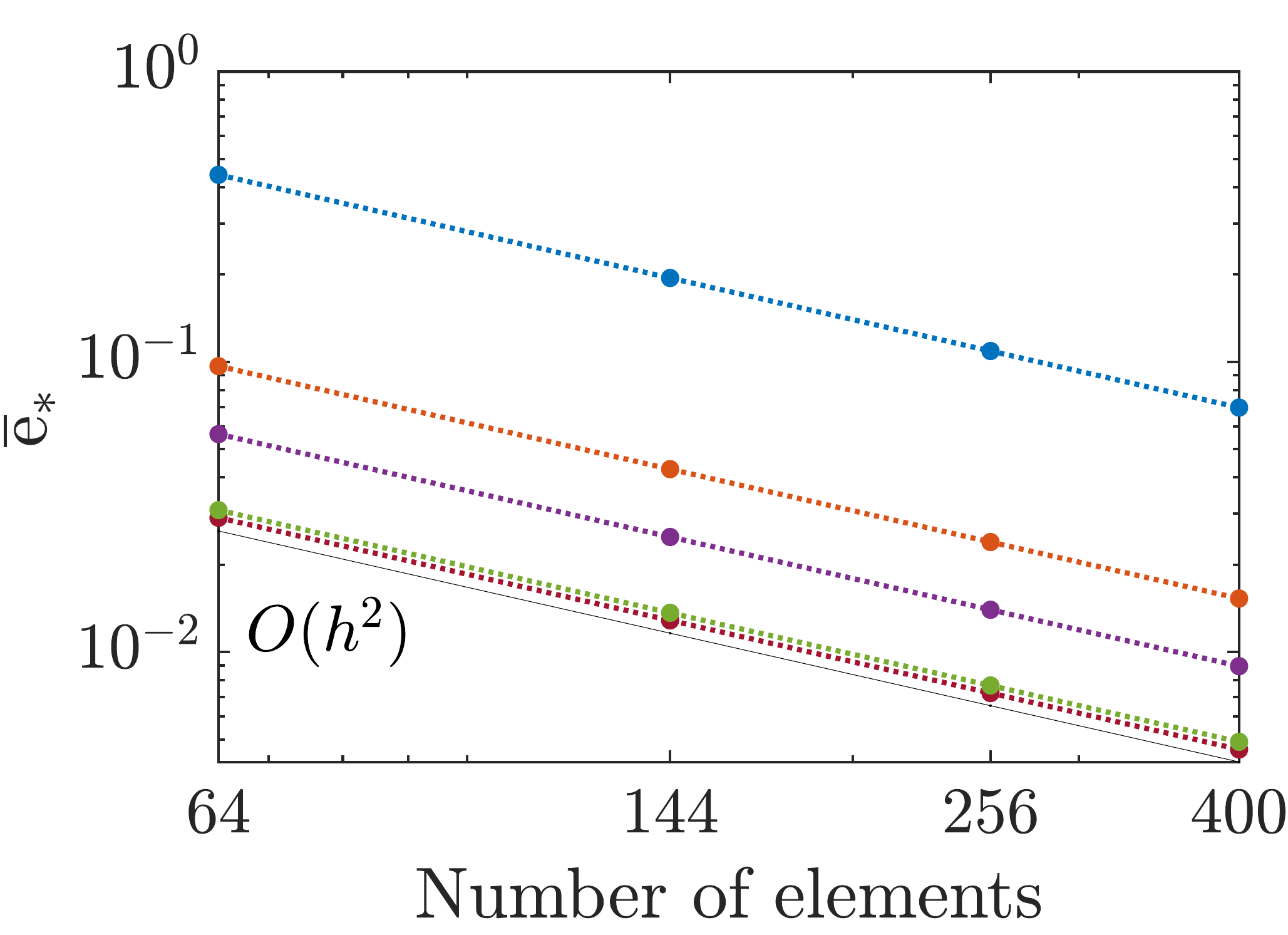}}
    \subfigure[Crossed: $k_2=0$.]{\includegraphics[width=0.223\textwidth,trim={14.3cm 0.5cm 2.2cm 0},clip]{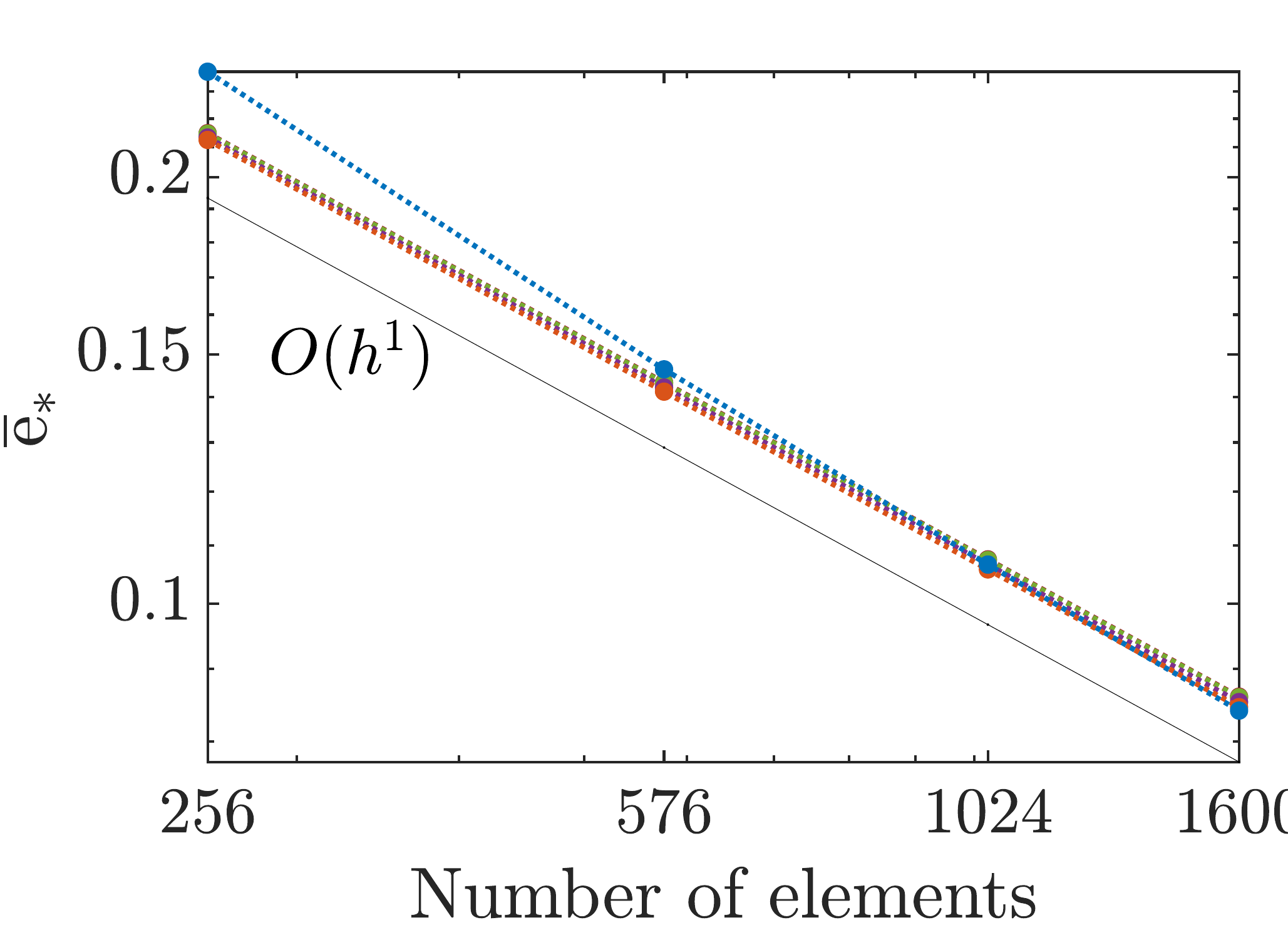}} 
    \subfigure[Crossed: $k_2=1$.]{\includegraphics[width=0.223\textwidth,trim={14.3cm 0.5cm 2.2cm 0},clip]{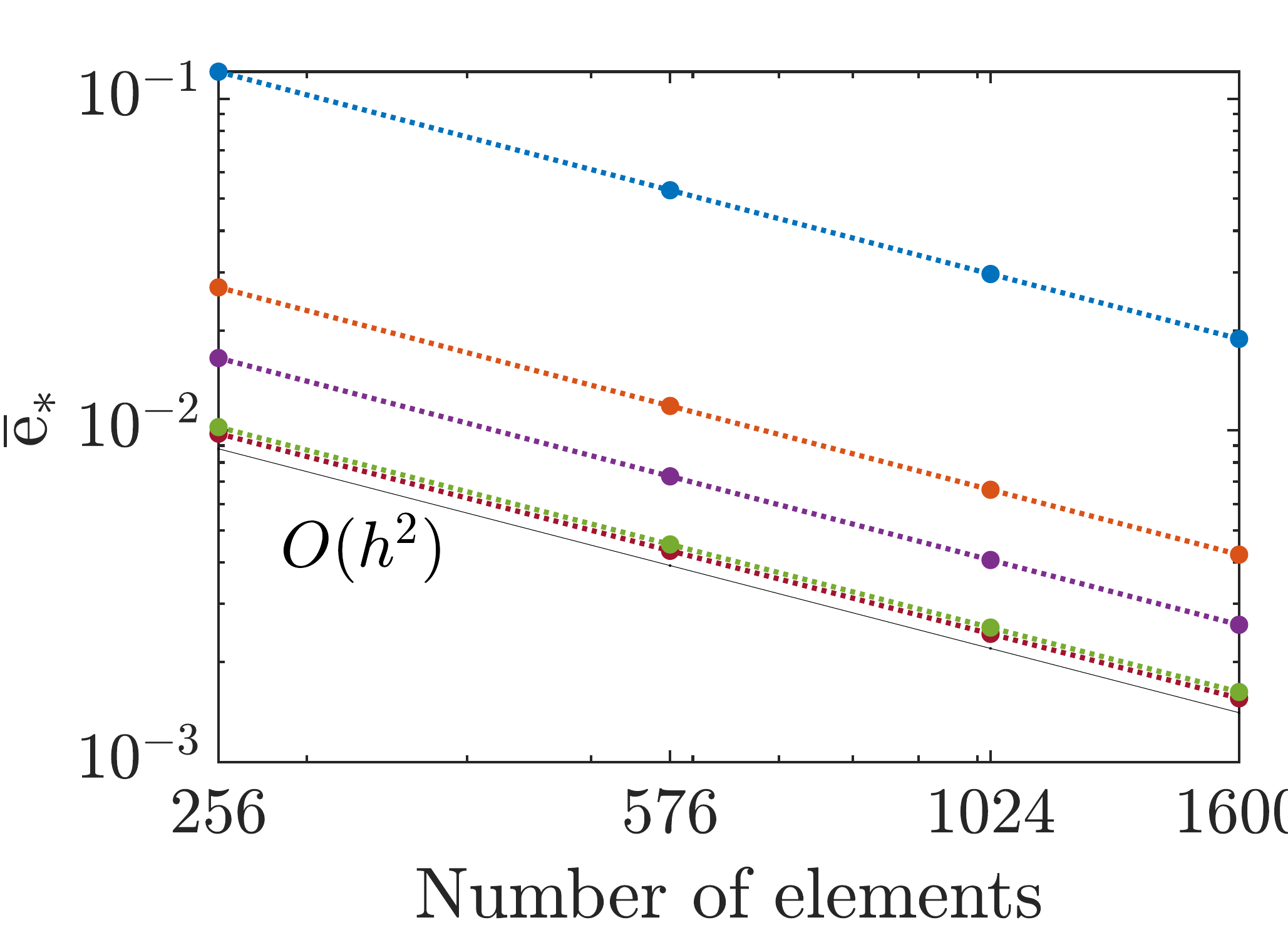}}\\ 
    \subfigure[Voronoi: $k_2=0$.]{\includegraphics[width=0.31\textwidth,trim={0 0.5cm 2.2cm 0},clip]{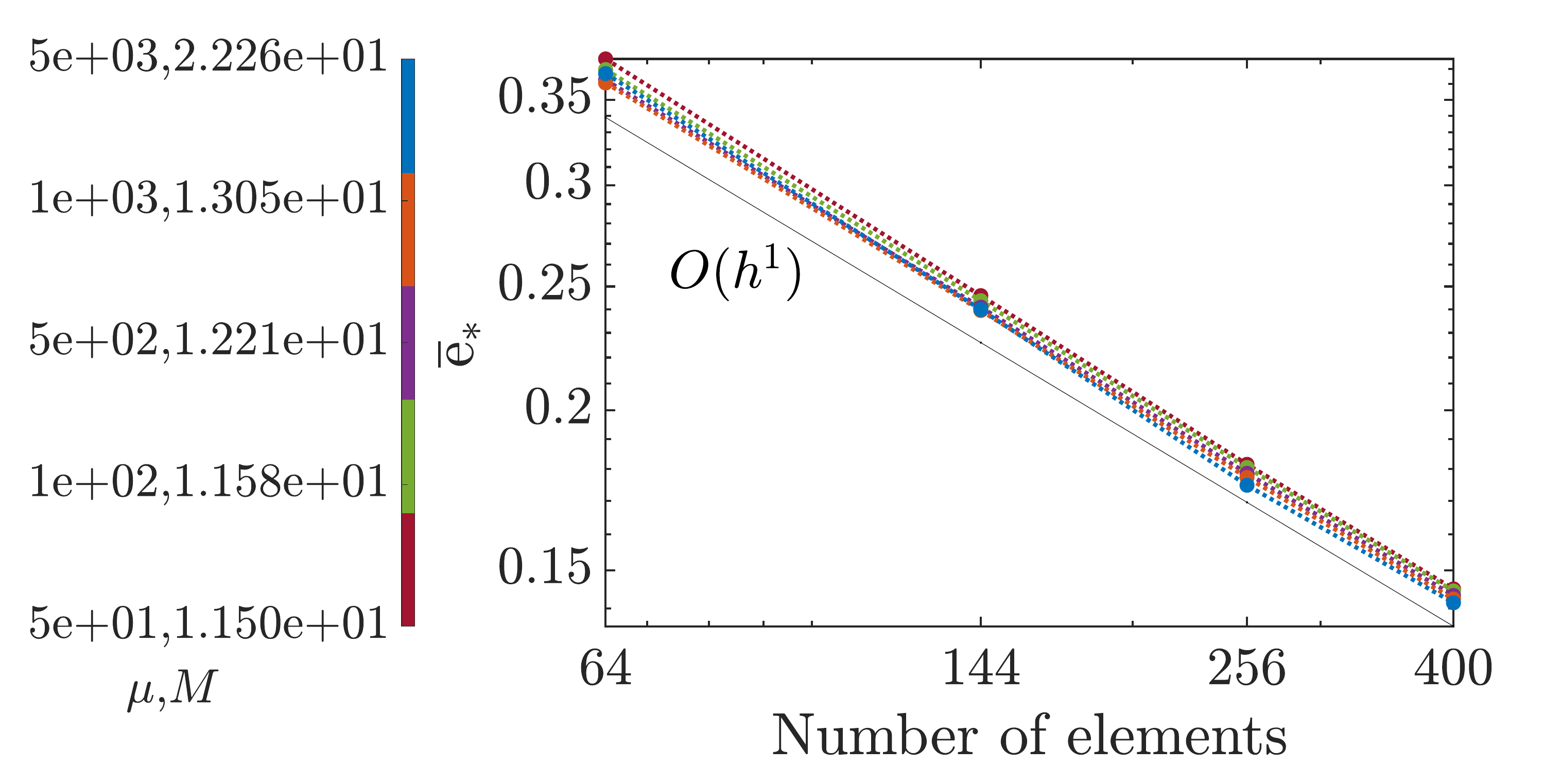}}
    \subfigure[Voronoi: $k_2=1$.]{\includegraphics[width=0.223\textwidth,trim={14.3cm 0.5cm 2.2cm 0},clip]{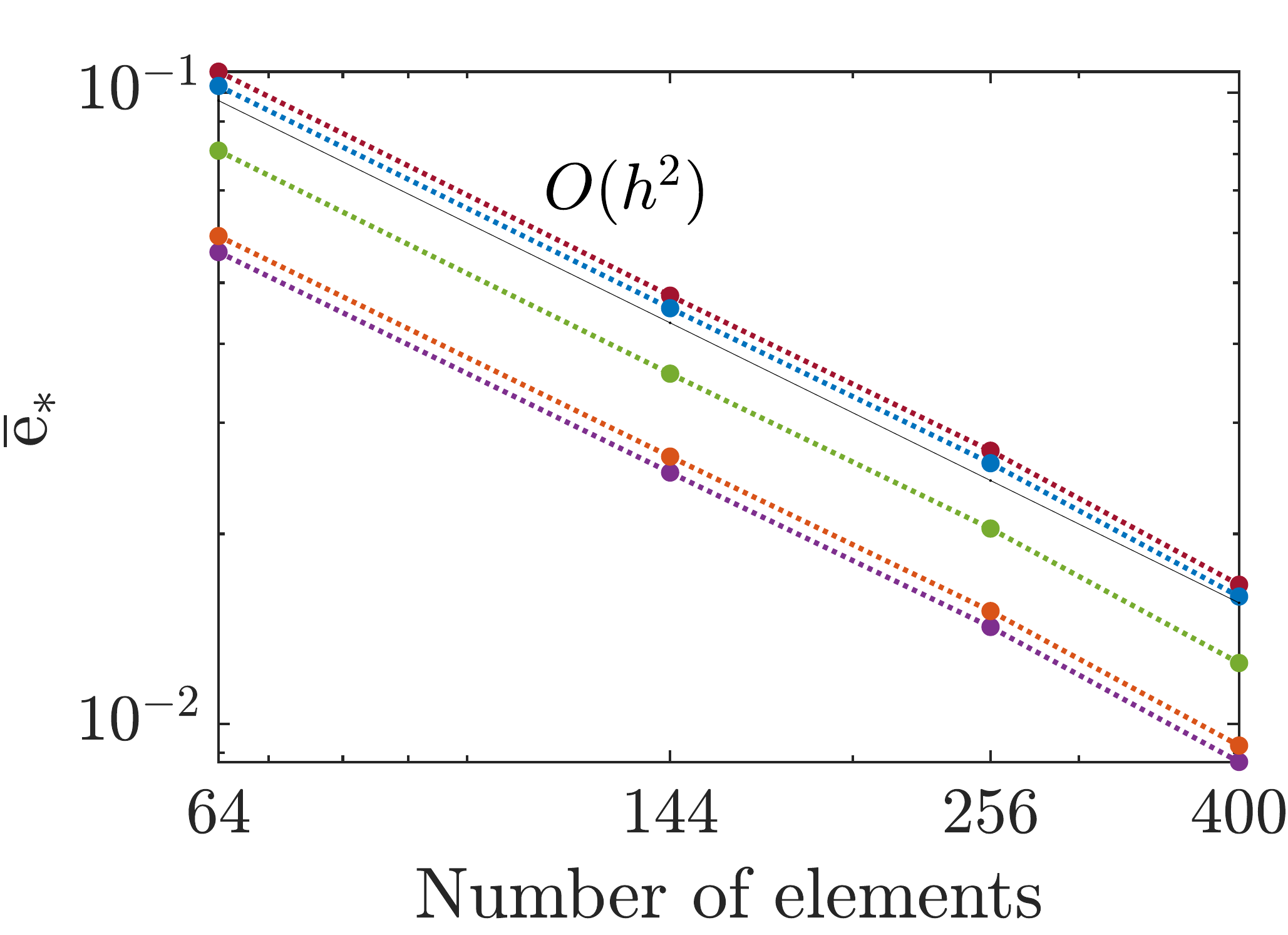}}
    \subfigure[P. Voronoi: $k_2=0$.]{\includegraphics[width=0.223\textwidth,trim={14.3cm 0.5cm 2.2cm 0},clip]{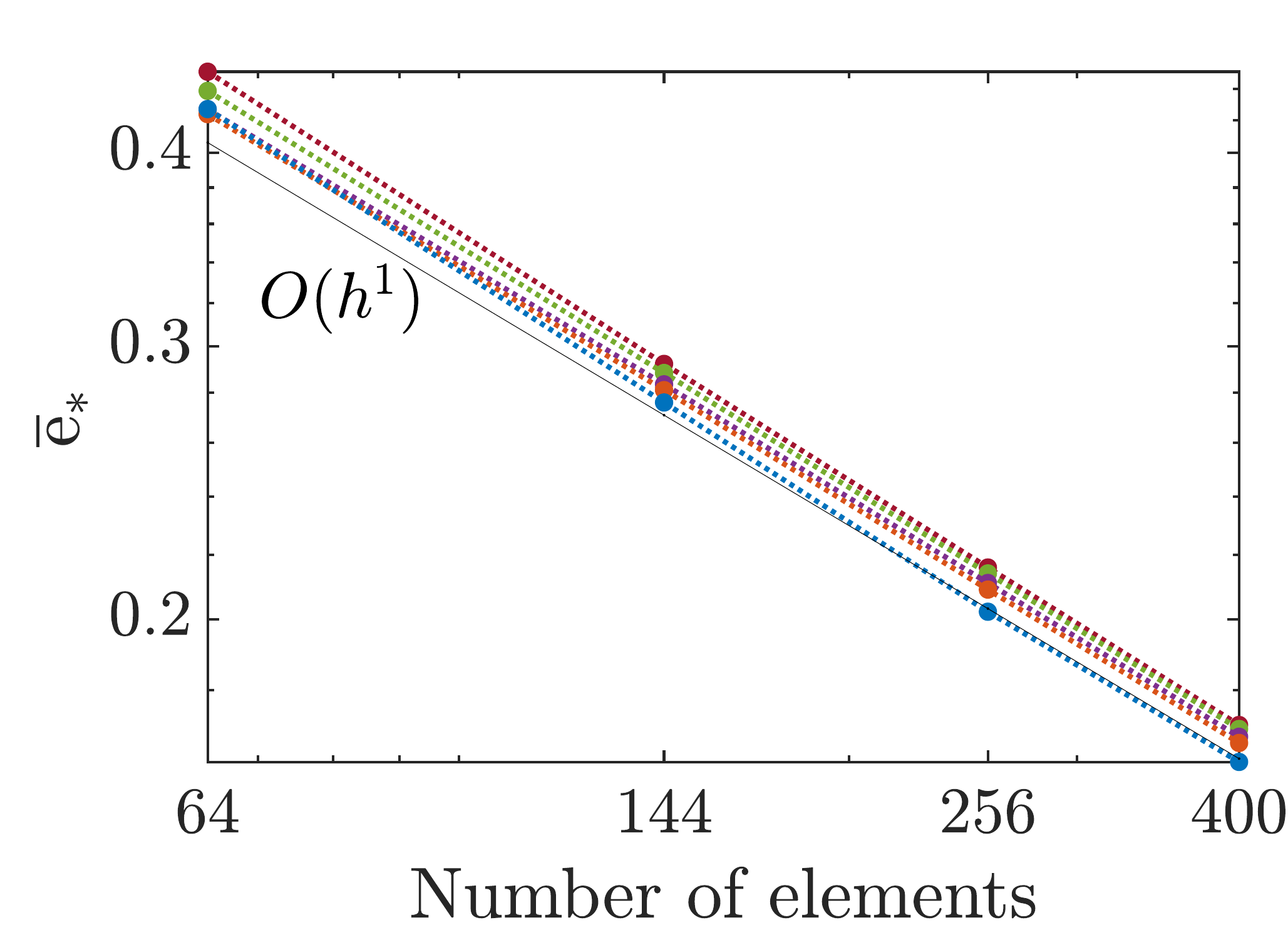}}
    \subfigure[P. Voronoi: $k_2=1$.]{\includegraphics[width=0.223\textwidth,trim={14.3cm 0.5cm 2.2cm 0},clip]{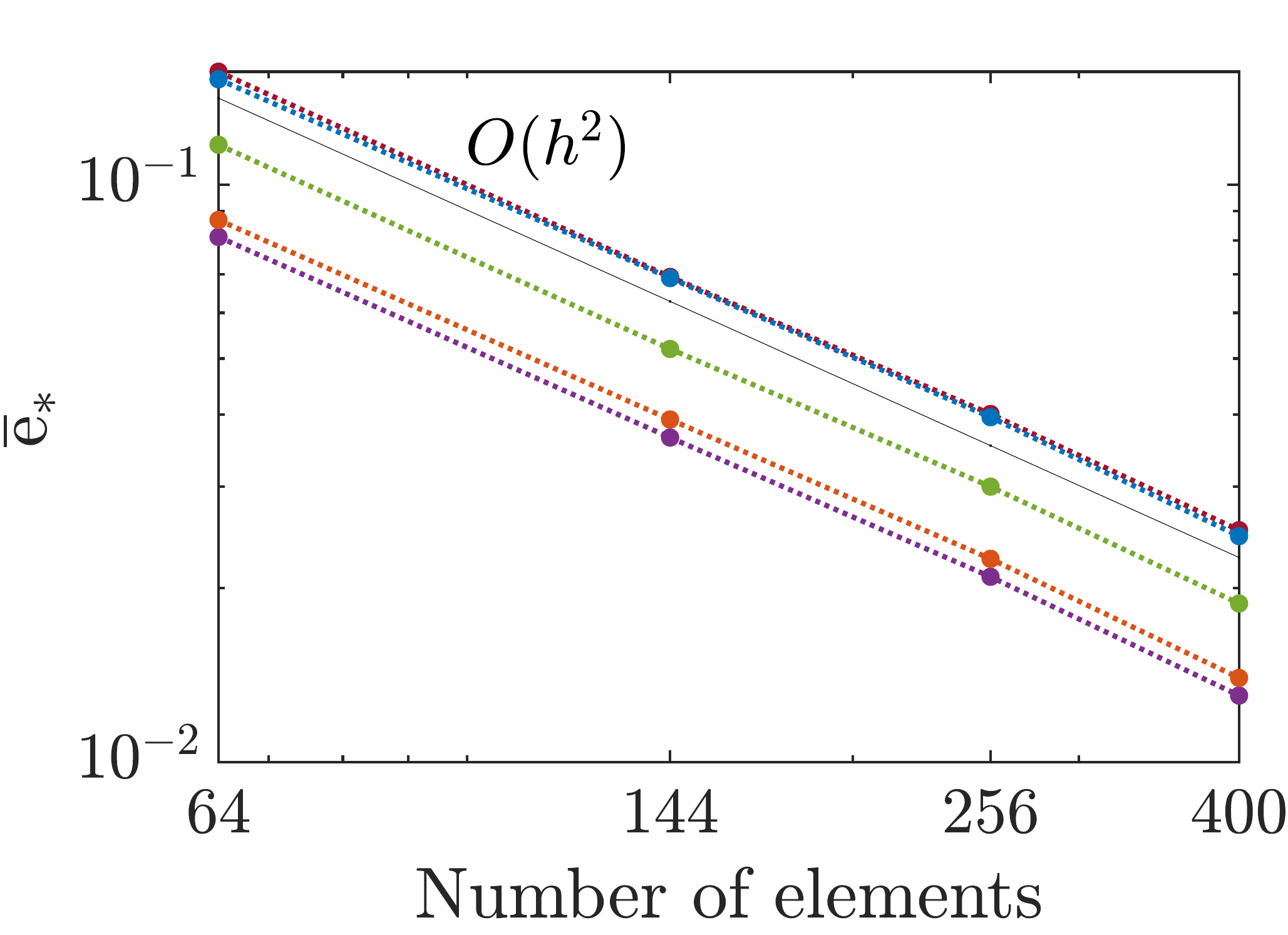}}\\
    \subfigure[Square: $k_2=0$.]{\includegraphics[width=0.31\textwidth,trim={0 0.5cm 2.2cm 0},clip]{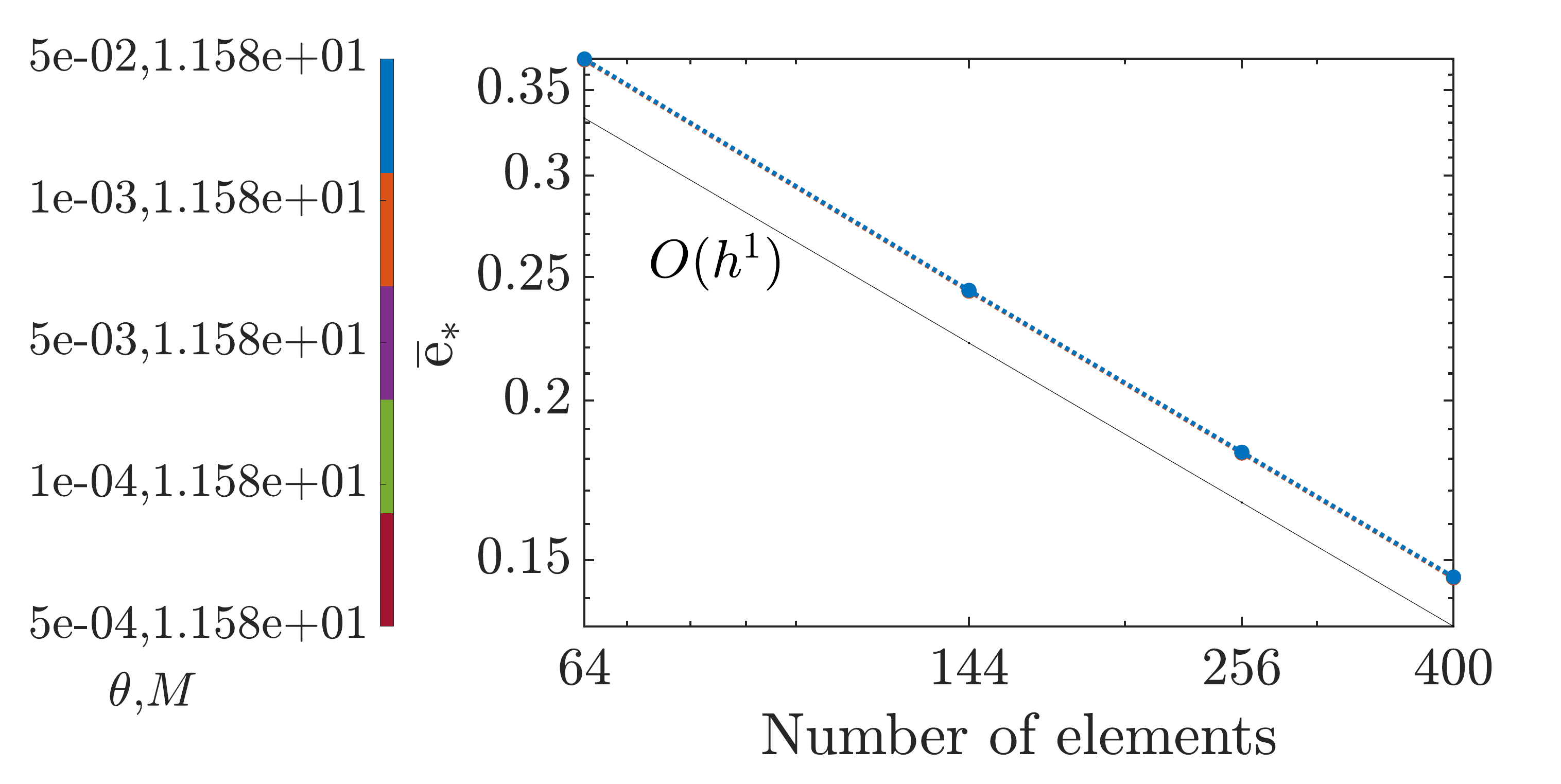}}
    \subfigure[Square: $k_2=1$.]{\includegraphics[width=0.223\textwidth,trim={13.7cm 0.5cm 2.2cm 0},clip]{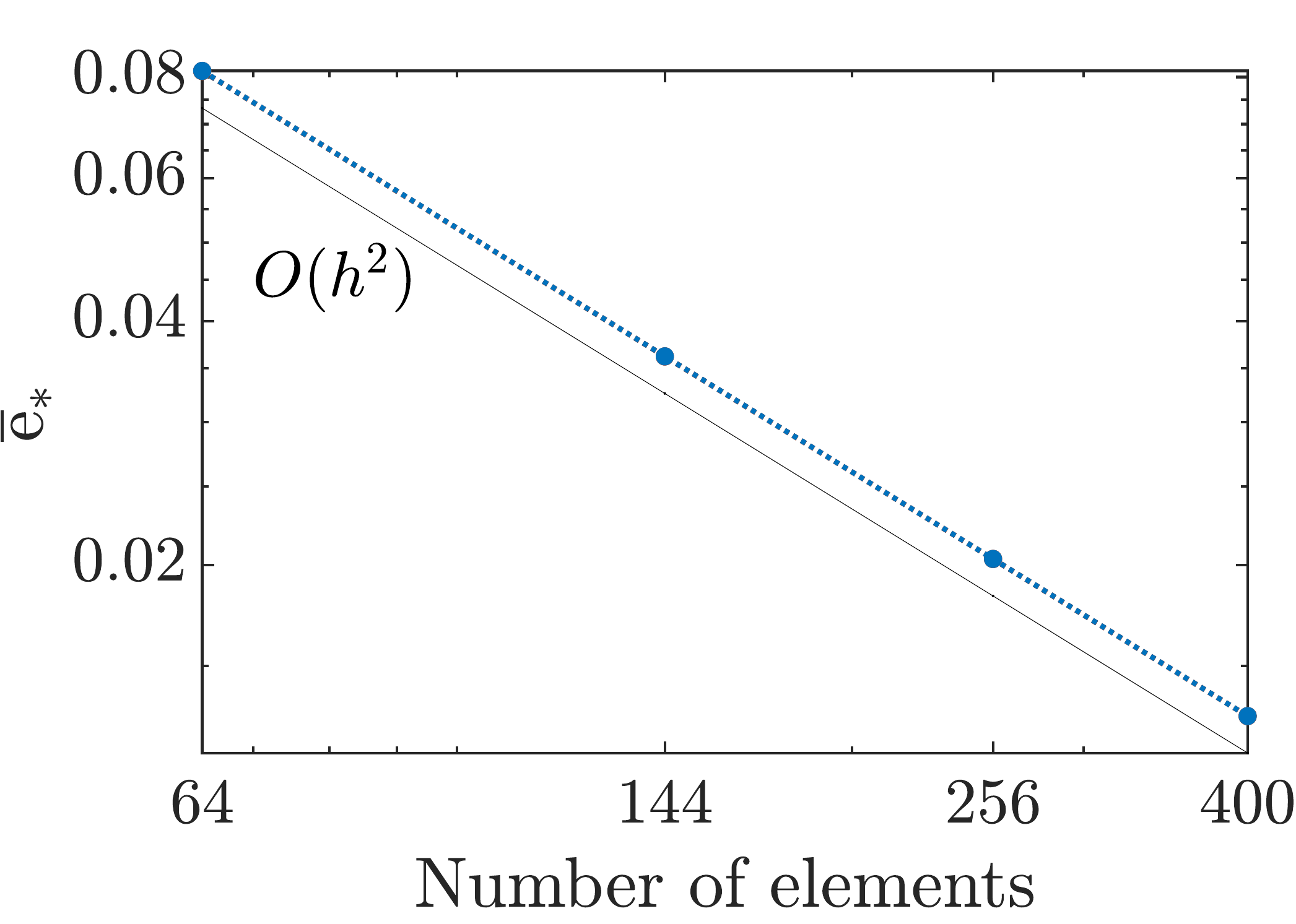}}
    \subfigure[Kangaroo: $k_2=0$.]{\includegraphics[width=0.223\textwidth,trim={13.7cm 0.5cm 2.2cm 0},clip]{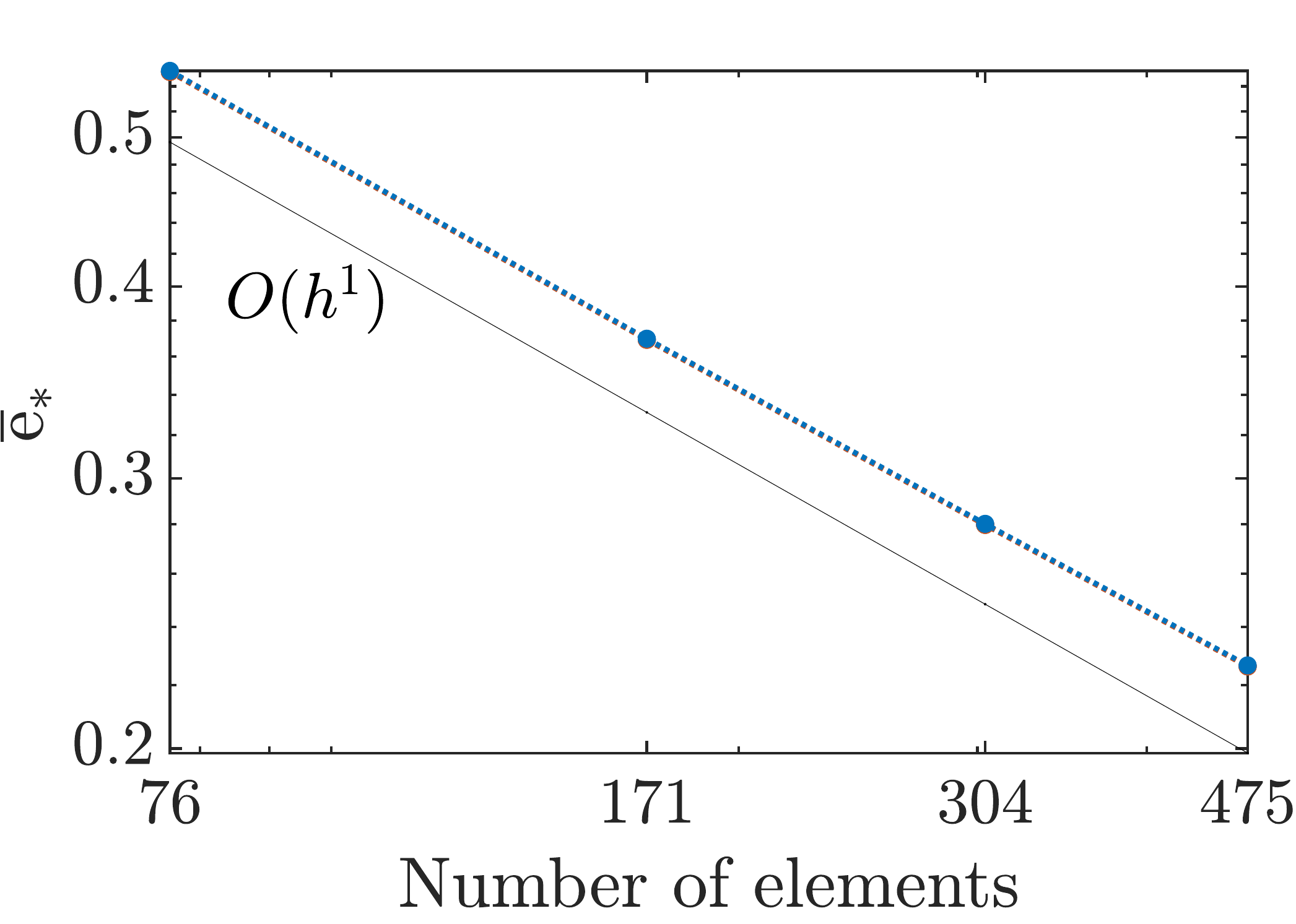}}
    \subfigure[Kangaroo: $k_2=1$.]{\includegraphics[width=0.223\textwidth,trim={13.7cm 0.5cm 2.2cm 0},clip]{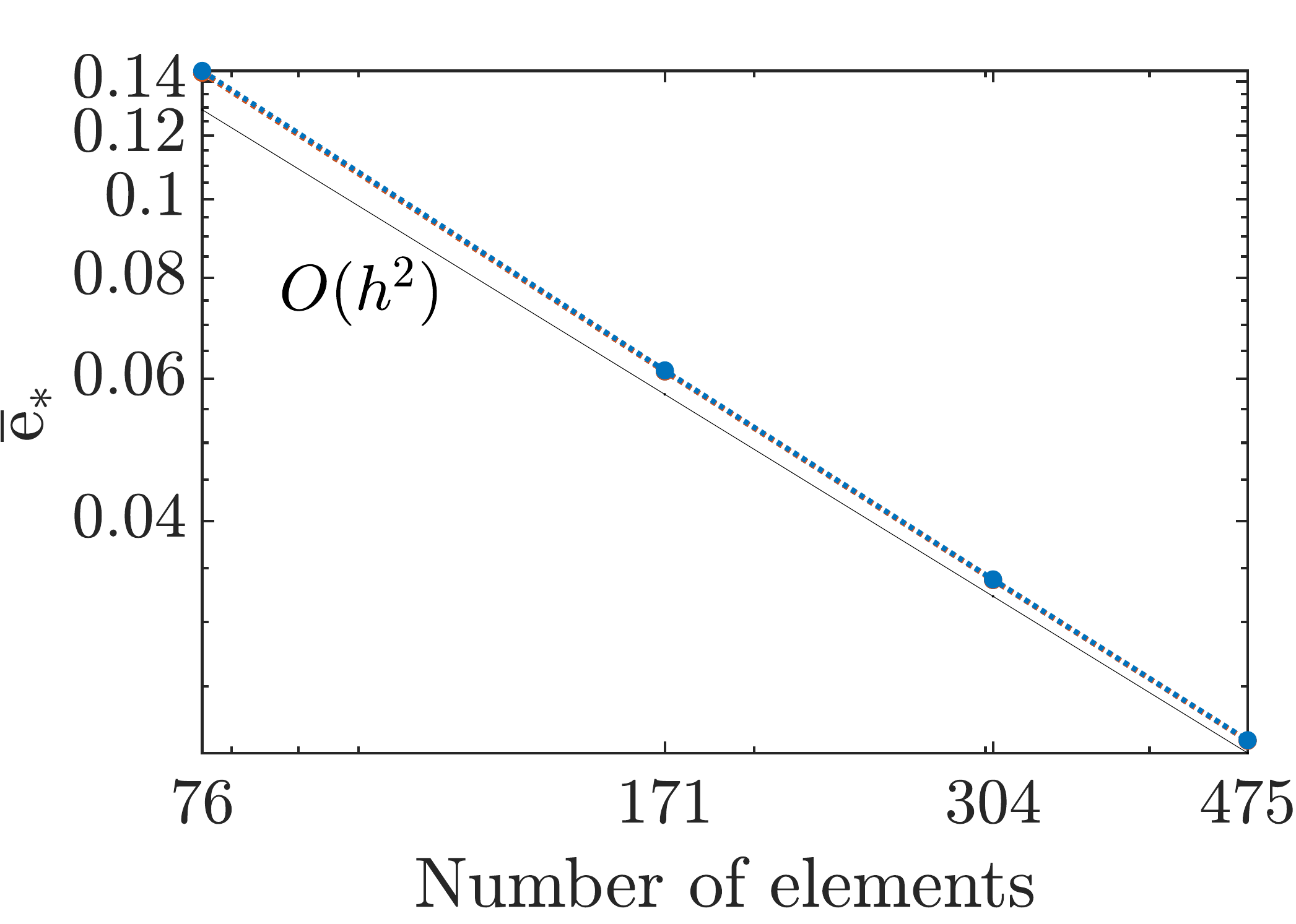}}
    \vspace{-0.25cm}
    \caption{\cblue{Example 2. Error history of the total error $\overline{\text{e}}_*$ for $k_1=2$ and $k_2=0,1$ is illustrated by examining the impact of varying the physical parameters $\lambda$, $\mu$, and $\theta$ (first, second, and third row) for a variety of meshes}.}
    \label{fig:robust_graph_complete}
\end{figure}

\begin{figure}[h!]
    \centering
    \begin{tikzpicture}[scale=1]
    \node{\subfigure[Non-convex: $k_2=0$.]{\includegraphics[width=0.485\textwidth]{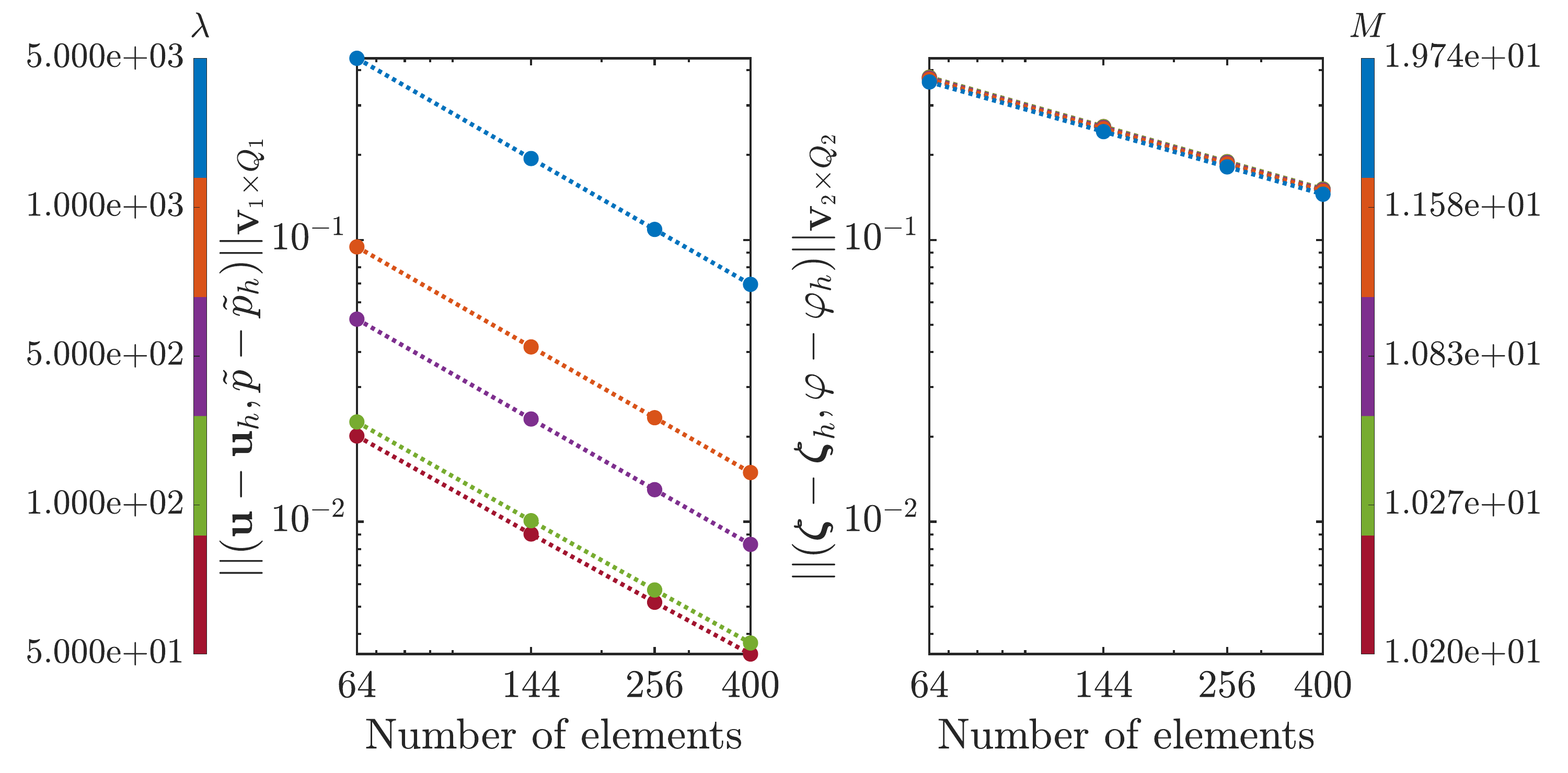}}
          \subfigure[Non-convex: $k_2=1$.]{\includegraphics[width=0.485\textwidth]{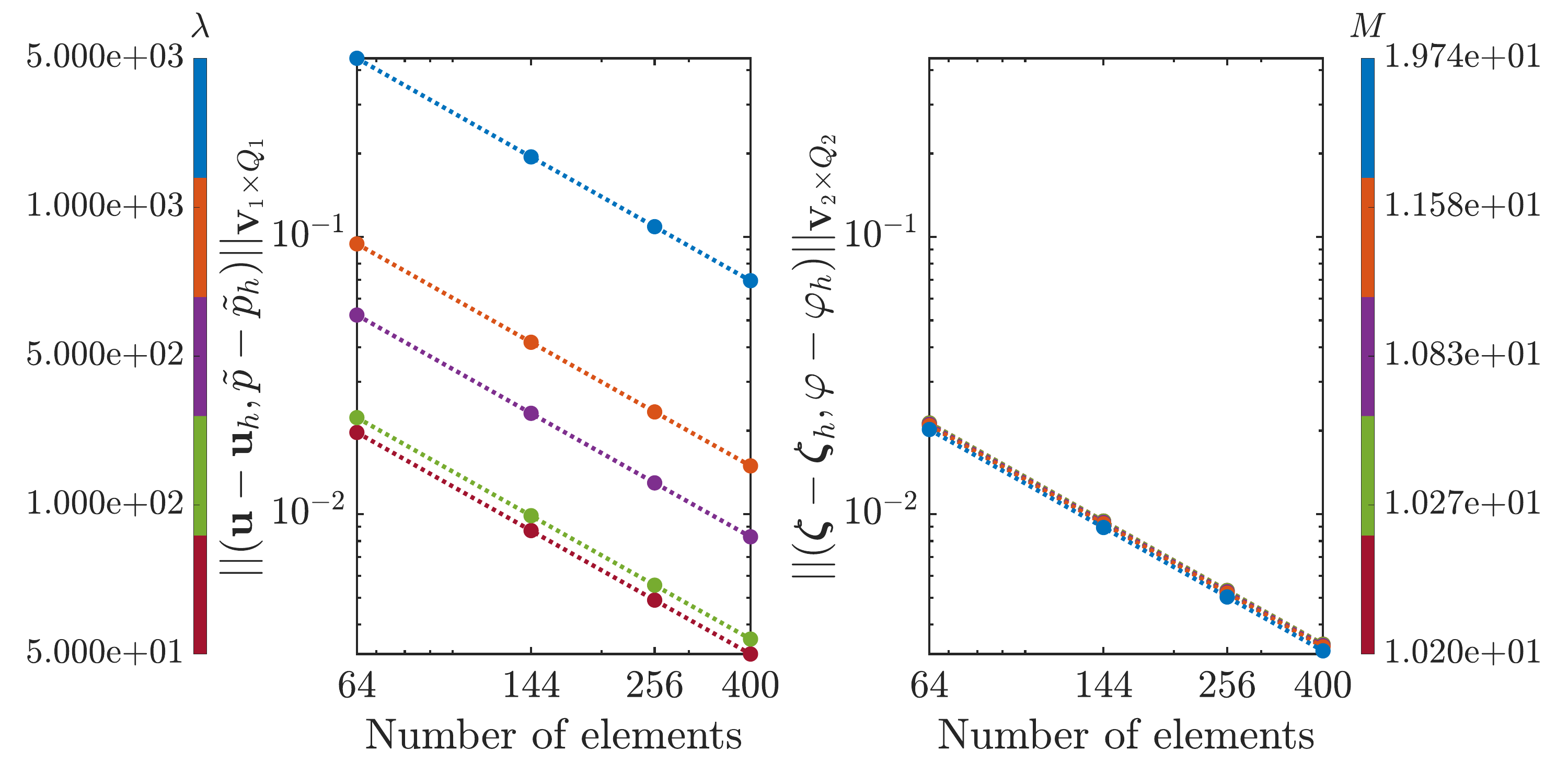}}};
    \begin{scope}[draw=black,thick]
        \draw [line width=0.15mm] (-6.05,-1) node[anchor=north]{}
        -- (-6.05,-0.5) node[anchor=north]{}
        -- (-5.15,-1) node[anchor=south]{} node[above,xshift=0.1cm,scale=0.55] {$O(h^2)$}
        -- cycle;
        \draw [line width=0.15mm] [line width=0.15mm] (-3.2,-1) node[anchor=north]{}
        -- (-3.2,-0.75) node[anchor=north]{}
        -- (-2.3,-1) node[anchor=south]{} node[above,xshift=0.1cm,scale=0.55] {$O(h^1)$}
        -- cycle;
        \draw [line width=0.15mm] (1.9,-1) node[anchor=north]{}
        -- (1.9,-0.5) node[anchor=north]{} 
        -- (2.8,-1) node[anchor=south]{} node[above,xshift=0.1cm,scale=0.55] {$O(h^2)$}
        -- cycle;
        \draw [line width=0.15mm] (4.75,-1) node[anchor=north]{}
        -- (4.75,-0.5) node[anchor=north]{}
        -- (5.65,-1) node[anchor=south]{} node[above,xshift=0.1cm,scale=0.55] {$O(h^2)$}
        -- cycle;
  \end{scope}
\end{tikzpicture}
  \begin{tikzpicture}[scale=1]
  \node{\subfigure[Non-convex: $k_2=0$.]{\includegraphics[width=0.485\textwidth]{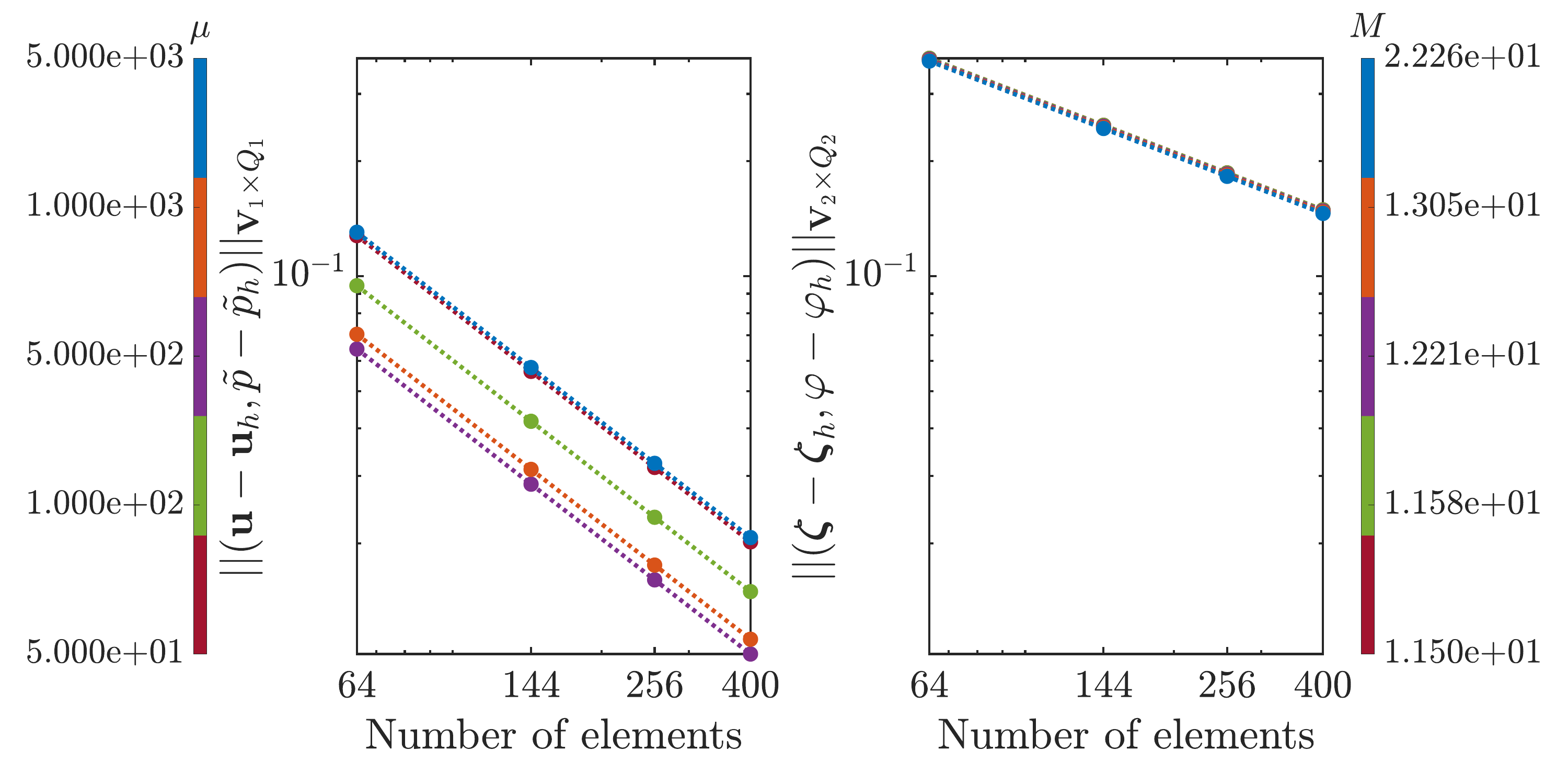}} 
        \subfigure[Non-convex: $k_2=1$.]{\includegraphics[width=0.485\textwidth]{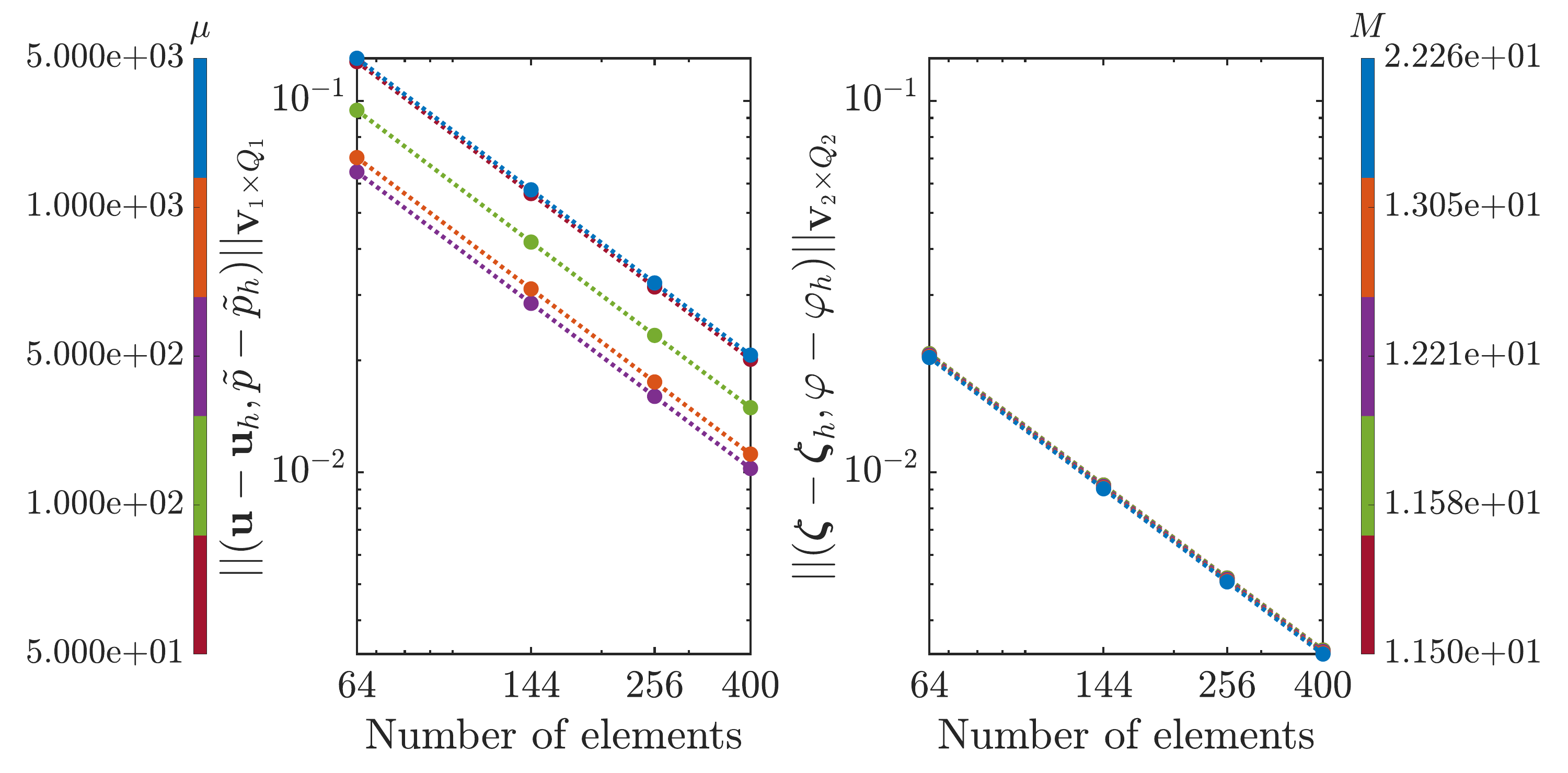}}};
        \begin{scope}[draw=black,thick]
        \draw [line width=0.15mm] (-6.05,-1) node[anchor=north]{}
        -- (-6.05,-0.5) node[anchor=north]{}
        -- (-5.15,-1) node[anchor=south]{} node[above,xshift=0.1cm,scale=0.55] {$O(h^2)$}
        -- cycle;
        \draw [line width=0.15mm] [line width=0.15mm] (-3.2,-1) node[anchor=north]{}
        -- (-3.2,-0.75) node[anchor=north]{}
        -- (-2.3,-1) node[anchor=south]{} node[above,xshift=0.1cm,scale=0.55] {$O(h^1)$}
        -- cycle;
        \draw [line width=0.15mm] (1.9,-1) node[anchor=north]{}
        -- (1.9,-0.5) node[anchor=north]{} 
        -- (2.8,-1) node[anchor=south]{} node[above,xshift=0.1cm,scale=0.55] {$O(h^2)$}
        -- cycle;
        \draw [line width=0.15mm] (4.75,-1) node[anchor=north]{}
        -- (4.75,-0.5) node[anchor=north]{}
        -- (5.65,-1) node[anchor=south]{} node[above,xshift=0.1cm,scale=0.55] {$O(h^2)$}
        -- cycle;
  \end{scope}
\end{tikzpicture}
\begin{tikzpicture}[scale=1]
  \node{\subfigure[Non-convex: $k_2=0$.]{\includegraphics[width=0.485\textwidth]{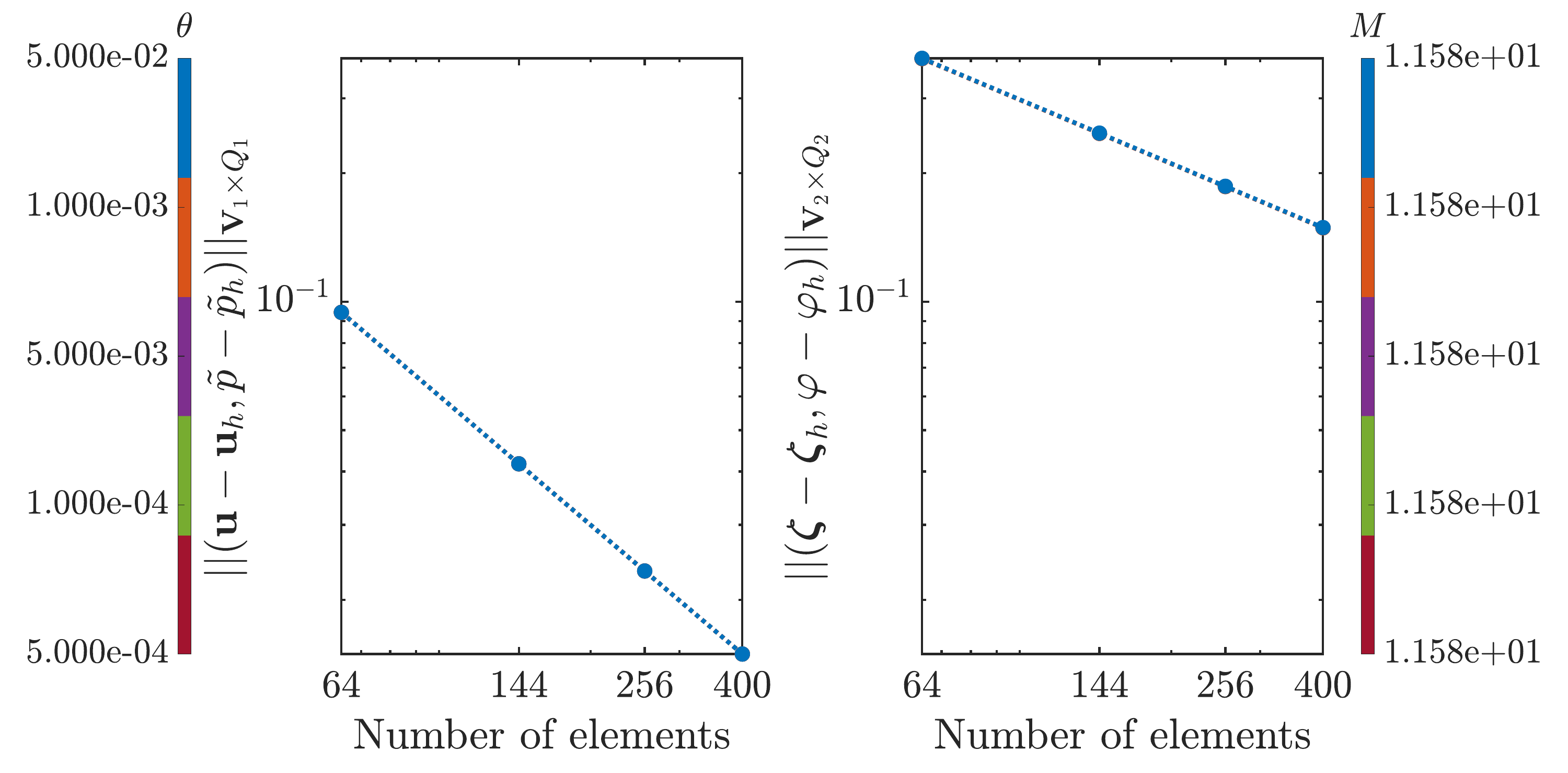}} 
        \subfigure[Non-convex: $k_2=1$.]{\includegraphics[width=0.485\textwidth]{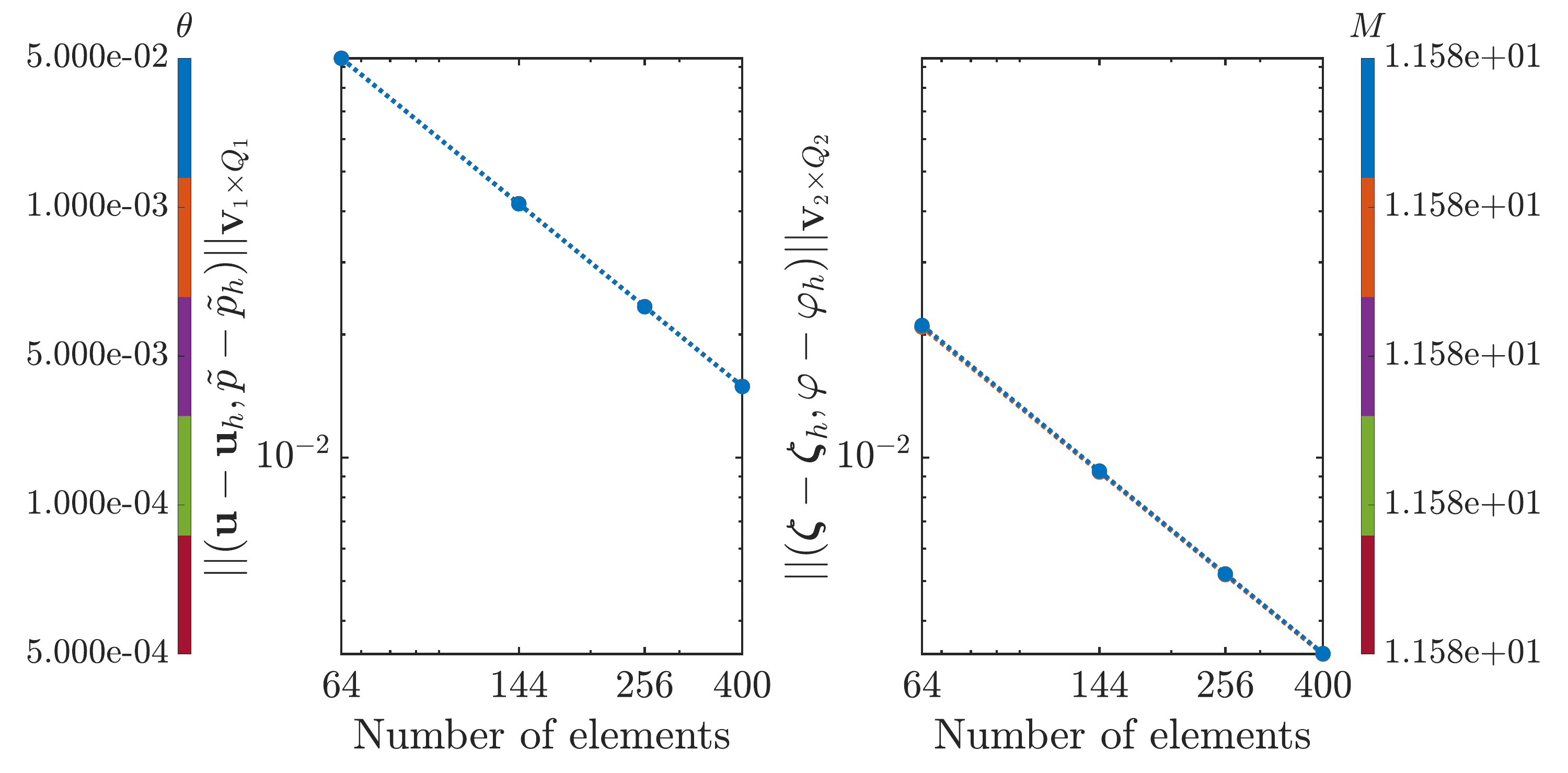}}};
        \begin{scope}[draw=black,thick]
        \draw [line width=0.15mm] (-6.125,-1) node[anchor=north]{}
        -- (-6.125,-0.4) node[anchor=north]{}
        -- (-5.325,-1) node[anchor=south]{} node[above,xshift=0.15cm,scale=0.55] {$O(h^2)$}
        -- cycle;
        \draw [line width=0.15mm] [line width=0.15mm] (-3.225,-1) node[anchor=north]{}
        -- (-3.225,-0.7) node[anchor=north]{}
        -- (-2.425,-1) node[anchor=south]{} node[above,xshift=0.1cm,scale=0.55] {$O(h^1)$}
        -- cycle;
        \draw [line width=0.15mm] (1.825,-1) node[anchor=north]{}
        -- (1.825,-0.4) node[anchor=north]{} 
        -- (2.625,-1) node[anchor=south]{} node[above,xshift=0.15cm,scale=0.55] {$O(h^2)$}
        -- cycle;
        \draw [line width=0.15mm] (4.75,-1) node[anchor=north]{}
        -- (4.75,-0.4) node[anchor=north]{}
        -- (5.55,-1) node[anchor=south]{} node[above,xshift=0.15cm,scale=0.55] {$O(h^2)$}
        -- cycle;
  \end{scope}
\end{tikzpicture}
    \caption{\cblue{Example 2. Error history for the partial error with $k_1=2$ and $k_2=0,1$ for the non-convex mesh is illustrated by examining the impact of varying the physical parameters $\lambda$, $\mu$, and $\theta$ (first, second, and third row)}.}
    \label{fig:robust_graph_separated}
\end{figure}
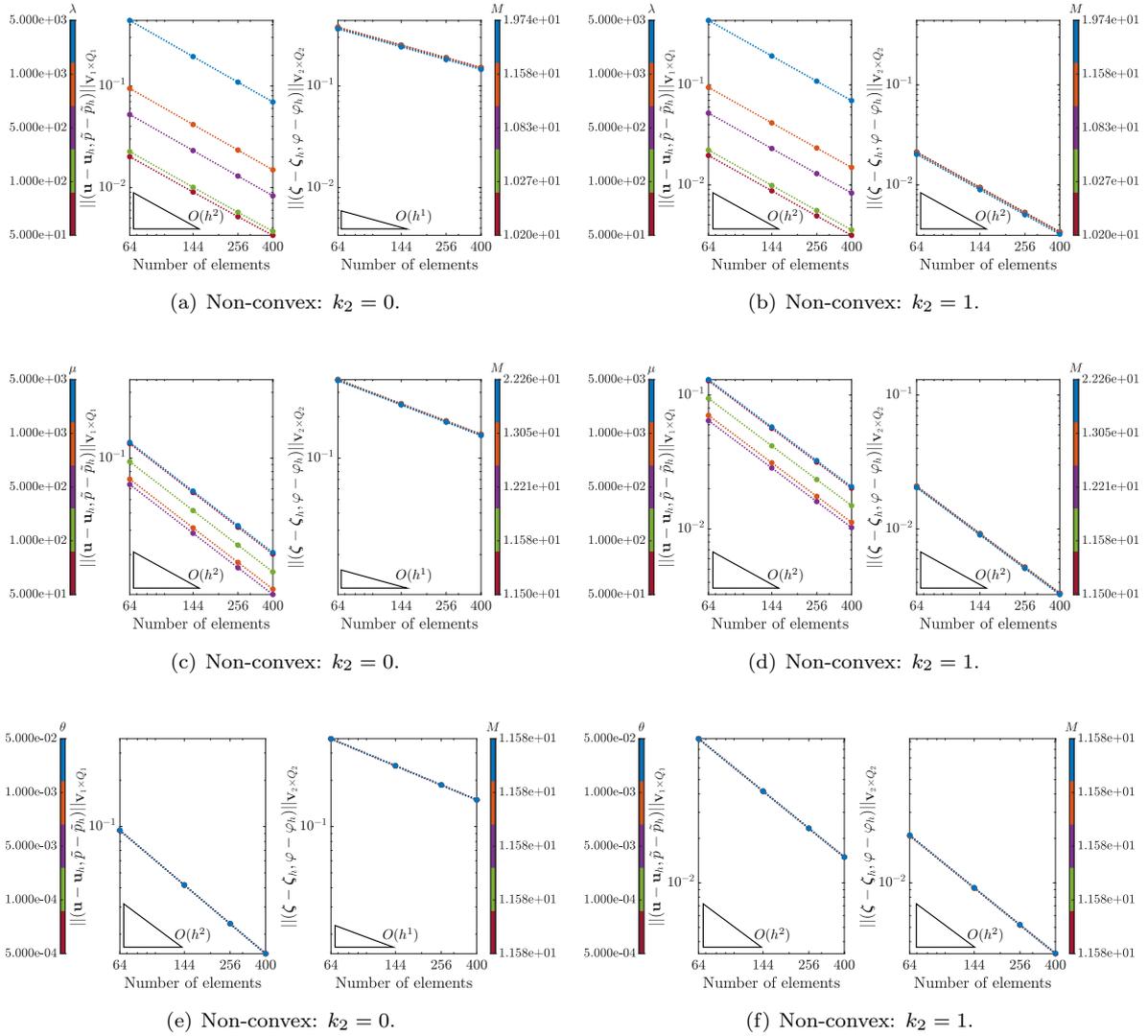

\begin{figure}[ht]\label{fig:pcp}

\centering

\subfigure[Sketch of a silicon anode.]{\includegraphics[width=0.33\textwidth]{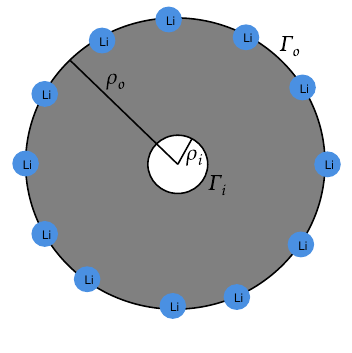}}
\subfigure[\cblue{Zero traction on $\Gamma_o$ (centre) and $\qty{-2d-4}{\frac{\N}{\mu \m^2}}$ traction on $\Gamma_o$ (right).}]{\includegraphics[width=0.66\textwidth]{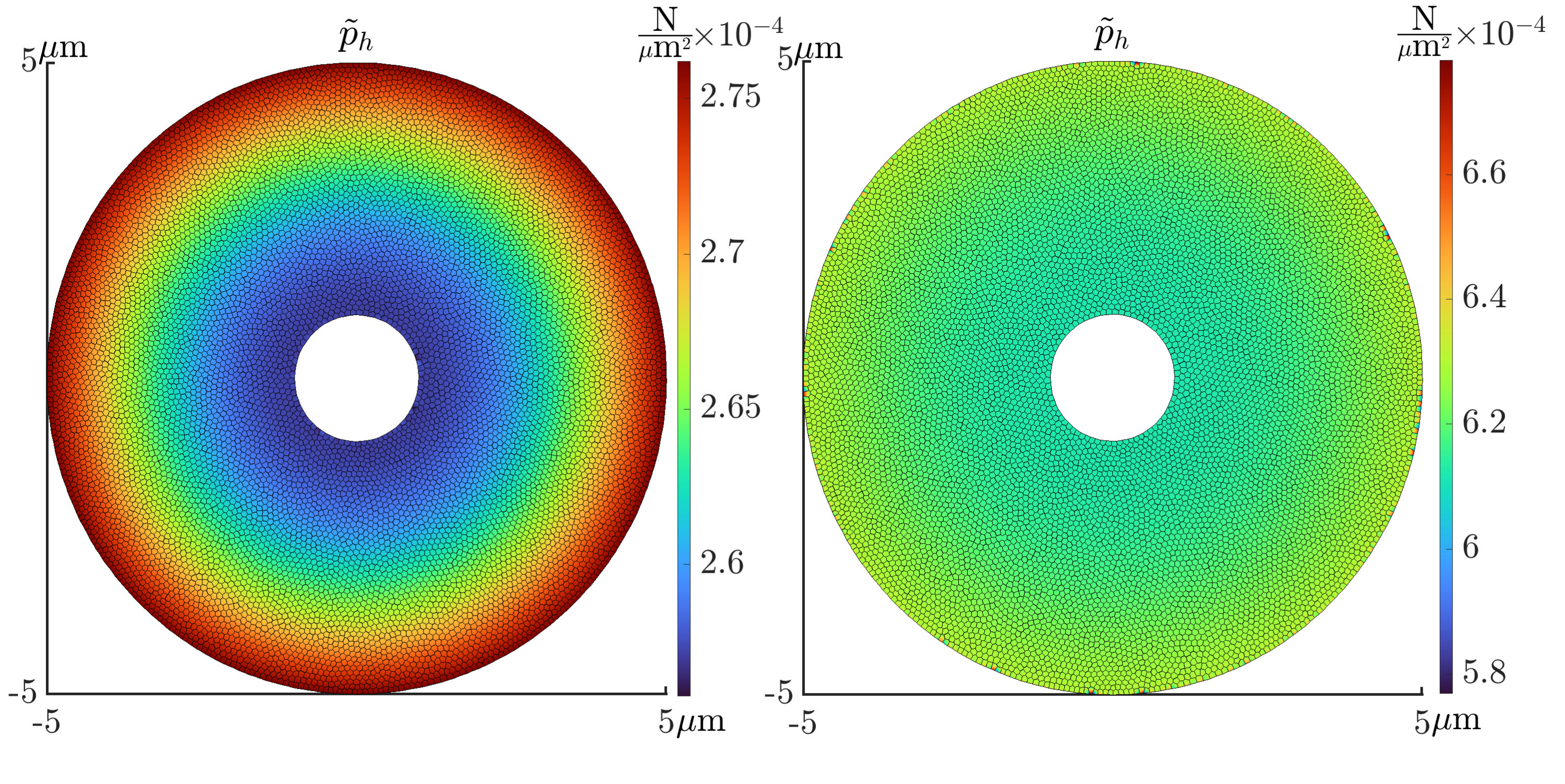}}  
\caption{\cblue{Example 3. (a) Perforated circle anode particle with $\rho_i = 1\mu\text{m}$ and $\rho_o=5\mu\text{m}$, the maximal Lithium concentration in $\Gamma_o$ is shown and the particle is clamped on $\Gamma_i$, and (b) solution of $\tilde{p}_h$ for the coupled case}.}
\end{figure}

\begin{figure}[ht]
    \centering
    \subfigure[Complete radial plot.]{\includegraphics[width=0.49\textwidth]{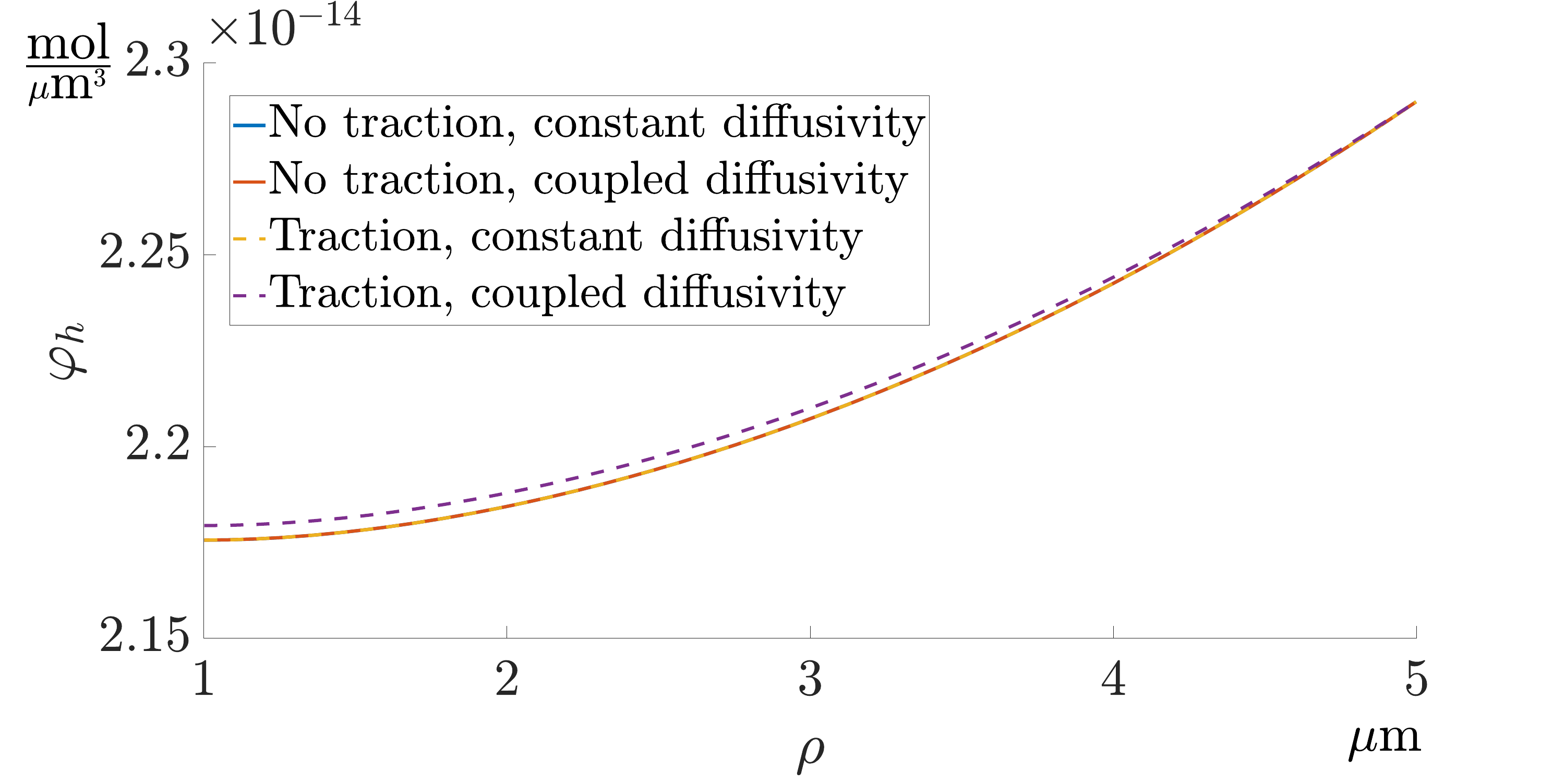}} 
    \subfigure[Zoom around $\rho = 3.05 \, \mu \text{m}$.]{\includegraphics[width=0.49\textwidth]{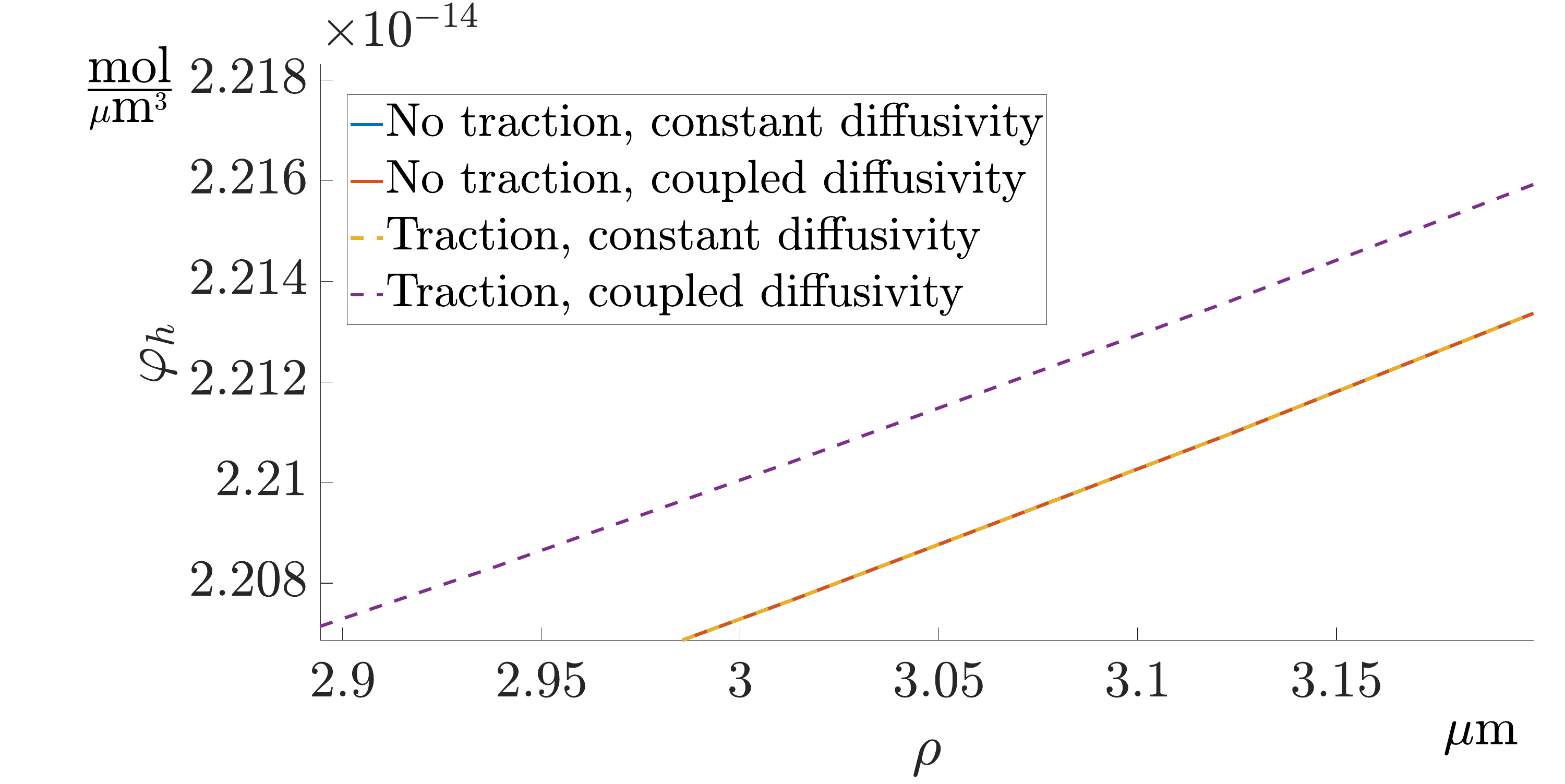}}
    \subfigure[Zoom around $\rho = 4.714267 \, \mu \text{m}$.]{\includegraphics[width=0.49\textwidth]{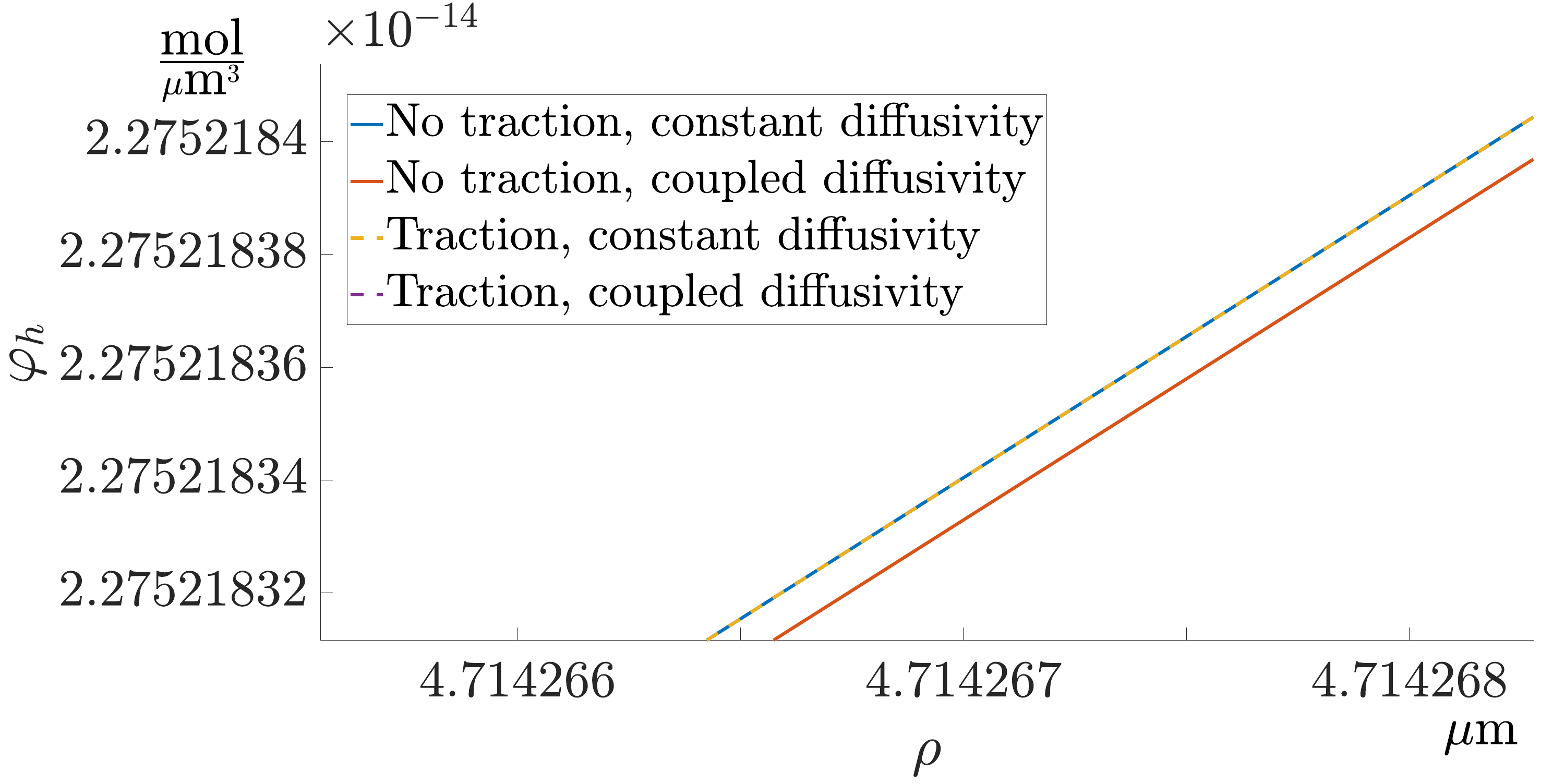}}
    \subfigure[Zoom around $\rho = 1.7314 \, \mu \text{m}$.]{\includegraphics[width=0.49\textwidth]{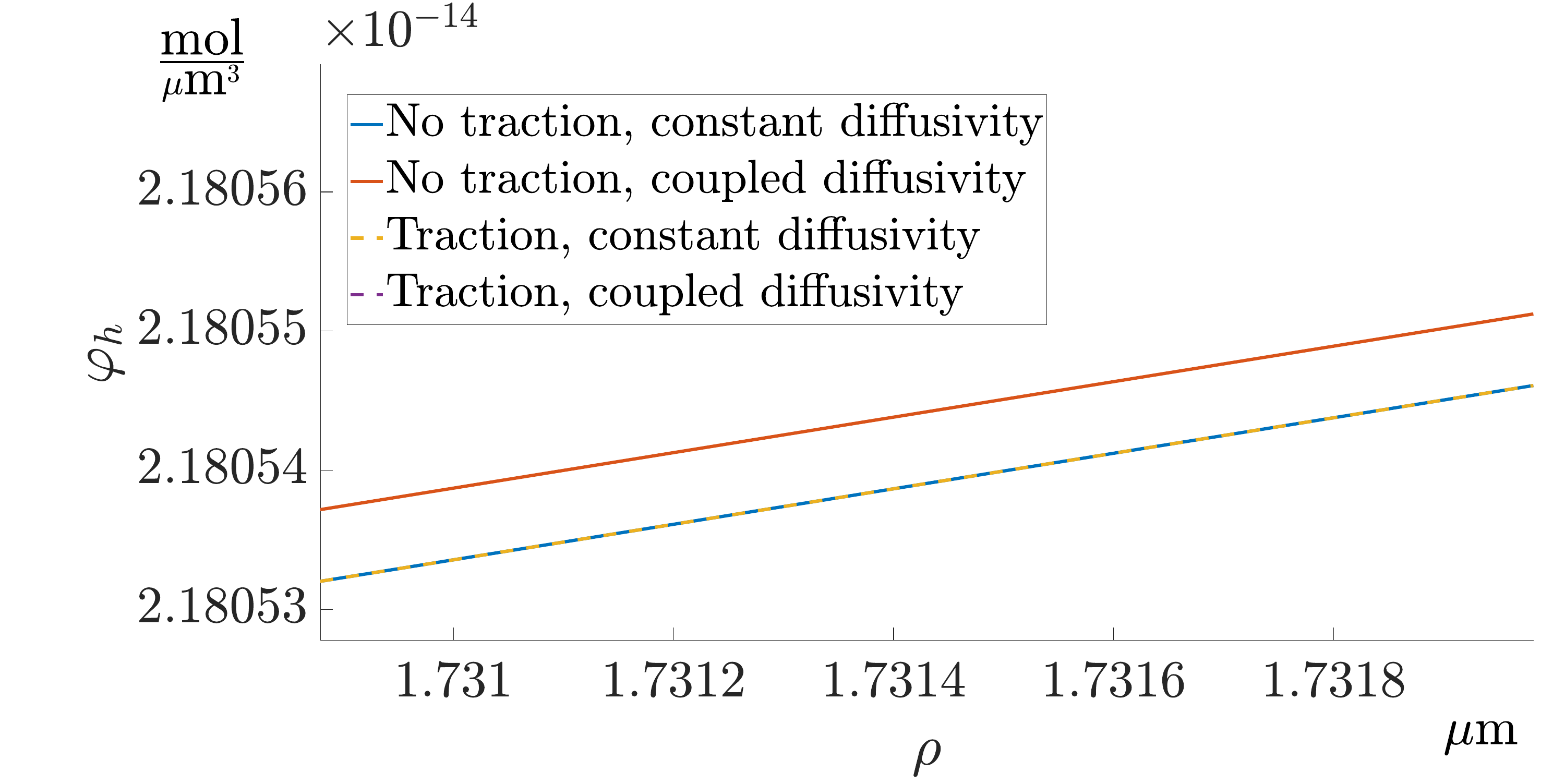}}
    \caption{\cblue{Example 3. Concentration $\varphi_h$ in the radial direction $\rho$, the blue and yellow lines coincide}.}
    \label{fig:radial-concentration}
\end{figure}

\section*{Acknowledgement} \cblue{We kindly thank Prof. Barbara Wohlmuth for pointing out an issue in the initial version of the error estimates for the nonlinear coupled problem.}

\bibliographystyle{siam}
\bibliography{bibliography}

\end{document}